\documentclass[leqno, 10pt, a4paper]{amsart}
\usepackage{amsmath, color}
\usepackage{amssymb, latexsym, mathrsfs, a4wide}

 \newtheorem{definition}{Definition}[section]
 \newtheorem{theorem}[definition]{Theorem}
 \newtheorem{lemma}[definition]{Lemma}
 \newtheorem{proposition}[definition]{Proposition}

 \newtheorem*{theorem*}{Theorem}
\newtheorem*{proposition*}{Proposition}
\newtheorem*{lemma*}{Lemma}

 \theoremstyle{remark}
 \newtheorem{example}[definition]{Example}
 \newtheorem{remark}[definition]{Remark}

   \newtheorem{claim}{Claim}

\newcommand{\op}[1]{\operatorname{#1}}

\newcommand{\acou}[2]{\ensuremath{\langle #1 , #2 \rangle}}

\newcommand{\Tr}{\ensuremath{\op{Tr}}}
\newcommand{\tr}{\op{tr}}

\newcommand{\Res}{\ensuremath{\op{Res}}}

\def\Xint#1{\mathchoice
{\XXint\displaystyle\textstyle{#1}}%
{\XXint\textstyle\scriptstyle{#1}}%
{\XXint\scriptstyle\scriptscriptstyle{#1}}%
{\XXint\scriptscriptstyle\scriptscriptstyle{#1}}%
\!\int}
\def\XXint#1#2#3{{\setbox0=\hbox{$#1{#2#3}{\int}$}
\vcenter{\hbox{$#2#3$}}\kern-.5\wd0}}

\def\dashint{\Xint-}

\newcommand{\bint}{\ensuremath{\dashint}}

\newcommand{\Str}{\op{Str}}
\newcommand{\str}{\op{str}}
\newcommand{\ind}{\op{ind}}
\newcommand{\Ch}{\op{Ch}}
\newcommand{\bCh}{\op{\mathbf{Ch}}}

\newcommand{\Pf}{\op{Pf}}

\newcommand{\Cl}{\op{Cl}}

\newcommand{\Hol}{\op{Hol}}

\newcommand{\C}{\ensuremath{\mathbb{C}}} 
 
\newcommand{\N}{\ensuremath{\mathbb{N}}} 
 
\newcommand{\R}{\ensuremath{\mathbb{R}}} 
 
\newcommand{\Z}{\ensuremath{\mathbb{Z}}}

\newcommand{\Rno}{\R^{n}\!\setminus\! 0}

\newcommand{\UR}{U\times\R}

\newcommand{\fa}{\ensuremath{\mathfrak{a}}}

\newcommand{\Ca}[1]{\ensuremath{\mathcal{#1}}}
\newcommand{\cA}{\Ca{A}}
\newcommand{\cB}{\Ca{B}}
\newcommand{\cD}{\ensuremath{\mathcal{D}}}
\newcommand{\cE}{\Ca{E}}

\newcommand{\cH}{\ensuremath{\mathcal{H}}}

\newcommand{\cK}{\ensuremath{\mathcal{K}}}
\newcommand{\cL}{\ensuremath{\mathcal{L}}}

\newcommand{\cN}{\Ca{N}}
\newcommand{\cP}{\Ca{P}}

\newcommand{\cS}{\ensuremath{\mathcal{S}}}

\newcommand{\scA}{\mathscr{A}}

\newcommand{\sD}{\ensuremath{{/\!\!\!\!D}}}
\newcommand{\sS}{\ensuremath{{/\!\!\!\!\!\;S}}}

\newcommand{\pdo}{\ensuremath{\Psi}}

\newcommand{\psido}{$\Psi$DO} 
\newcommand{\psidos}{$\Psi$DOs}

\newcommand{\pvdo}{\ensuremath{\Psi_{\textup{v}}}}
\newcommand{\OP}{\ensuremath{OP}}

\newcommand{\Sv}{\ensuremath{S_{\textup{v}}}}

\newcommand{\ord}{{\op{ord}}}

\newcommand{\End}{\ensuremath{\op{End}}}

\newcommand{\dom}{\op{dom}}

\newcommand{\bt}{\bullet}
\newcommand{\hotimes}{\hat\otimes}

\newcommand{\CM}{\textup{CM}}
\newcommand{\JLO}{\textup{JLO}}
\newcommand{\HP}{\op{HP}}
\newcommand{\bHP}{\op{\mathbf{HP}}}

\newcommand{\FP}{\underset{t=0^{+}}{\Pf}}

\newcommand{\ii}{\sqrt{-1}}

\numberwithin{equation}{section}

\begin{document}

\title{NONCOMMUTATIVE GEOMETRY AND CONFORMAL GEOMETRY, II.  CONNES-CHERN CHARACTER AND THE LOCAL EQUIVARIANT INDEX THEOREM.}
 \author{Rapha\"el Ponge}
 \address{Department of Mathematical Sciences, Seoul National University, Seoul, South Korea}
 \email{ponge.snu@gmail.com}
 \author{Hang Wang}
 \address{School of Mathematical Sciences, University of Adelaide, Adelaide, Australia}
 \email{hang.wang01@adelaide.edu.au}



\thanks{R.P.\ was partially supported by Research Resettlement Fund and Foreign Faculty Research Fund of Seoul National University and  Basic Research Grant 2013R1A1A2008802 
 of National Research Foundation of Korea (South Korea)}
 
\begin{abstract}
    This paper is the second part of a series of papers on noncommutative geometry and conformal geometry. In this paper, 
    we compute explicitly the Connes-Chern character of an equivariant Dirac spectral triple. The formula that we obtain
     for which was used in the first paper of the series. The computation has two main steps. The first step is the 
    justification that the CM cocycle represents the Connes-Chern character. The second step is the computation of the CM 
    cocycle as a byproduct of a new proof of the local equivariant index theorem of Donnelly-Patodi, Gilkey and 
    Kawasaki. The proof combines the rescaling method of Getzler with an equivariant version of the Greiner-Hadamard approach 
    to the heat kernel asymptotics. Finally, as a further application of this approach, we compute the short-time limit of the JLO cocycle of an equivariant 
    Dirac spectral triple. 
\end{abstract}

\maketitle 

\section{Introduction}
The paper is the second part of a series of papers whose aim is to use tools of noncommutative geometry to study conformal 
geometry and noncommutative versions of conformal geometry. In the prequel~\cite{PW:NCGCGI.PartI} (referred throughout this 
paper as Part~I) we derived a local index formula in conformal-diffeomorphism invariant geometry and exhibited a new class of conformal 
invariants taking into account the action by a group of conformal-diffeomorphisms (i.e., the conformal gauge group). 
These results make use of the conformal Dirac spectral triple of Connes-Moscovici~\cite{CM:TGNTQF} and two key features of its 
Connes-Chern character, namely, its conformal invariance and its explicit computation in terms of equivariant 
characteristic forms. The former feature is established in Part I. The latter is one of the main goal of this paper. 

It should be stressed out that the conformal Dirac spectral triple is not an ordinary spectral 
triple, but a \emph{twisted} spectral triple in the sense of~\cite{CM:TGNTQF}. As shown by 
Connes-Moscovici~\cite{CM:LIFNCG}, under suitable conditions, the Connes-Chern is represented by a cocycle, called CM 
cocycle, which is given by formulas that are ``local'' in the sense of noncommutative geometry. Except for a small class of 
twisted spectral triples, which don't include the conformal Dirac spectral triple, there is no analogue of the CM 
cocycle for twisted spectral triples (see~\cite{Mo:LIFTST} and the related discussion in Part~I). However, in the case of the conformal Dirac spectral triple, the conformal invariance of the Connes-Chern character 
allows us to choose \emph{any} metric in the 
given conformal class. In particular, when the conformal structure is non-flat, thanks to the Ferrand-Obata theorem we may choose a metric invariant under the whole 
conformal-diffeomorphism group. Moreover, in this case the conformal Dirac spectral triple becomes an ordinary 
spectral triple, namely, an equivariant Dirac spectral triple $( C^{\infty}(M)\rtimes G, L^{2}_{g}(M,\sS),\ \sD_{g})$. 
We are thus reduced to computing the Connes-Chern character of such a Dirac spectral triple.  

A main goal of this paper is the explicit computation of the Connes-Chern character of an equivariant Dirac 
spectral triple $( C^{\infty}(M)\rtimes G, L^{2}_{g}(M,\sS),\ \sD_{g})$, where $(M^{n},g)$ is a compact Riemannian oriented spin manifold 
of even dimension, $\sD_{g}$ is the Dirac operator acting on the spinor bundle $\sS$, and 
$C^{\infty}(M)\rtimes G$ is the (discrete) crossed product of the algebra $C^{\infty}(M)$ with a group $G$ 
of orientation-preserving isometries preserving the spin structure.  As it turns out, the construction in Part~I of the conformal invariants 
actually used the representation of the Connes-Chern character in the periodic cyclic cohomology of \emph{continuous} cochains on the locally convex algebra $C^{\infty}(M)\rtimes G$, 
rather than the usual representation in the periodic cyclic cohomology of arbitrary cochains. Therefore, we actually 
aim at computing the Connes-Chern character of $( C^{\infty}(M)\rtimes G, L^{2}_{g}(M,\sS),\ \sD_{g})$ in the cyclic cohomology 
$\bHP^{0}\left(C^{\infty}(M)\rtimes G\right)$ of \emph{continuous} cochains. We refer to Theorem~\ref{thm:Connes-Chern-conformal-DiracST} for the explicit formula for this 
Connes-Chern character.  The computation has two main steps. The first step is the justification that the CM cocycle makes sense and represents 
the Connes-Chern character in $\bHP^{0}\left(C^{\infty}(M)\rtimes G\right)$ (and not just in the ordinary periodic cyclic cohomology $\HP^{0}\left(C^{\infty}(M)\rtimes G\right)$. 
The second step is the actual computation of the CM cocycle. 

In the original paper of Connes-Moscovi~\cite{CM:LIFNCG}, the existence of the CM cocycle and the representation of the 
Connes-Chern character by this cocycle was proved under several assumptions, which were subsequently relaxed 
(see, e.g., \cite{Hi:RITCM, CPRS:CCSFST}).  In Part~I, we 
introduced a natural class of spectral triples over locally convex algebras, called smooth 
spectral triples. Given a smooth spectral triple $(\cA,\cH,D)$,  it was shown that the Connes-Chern character descends to a class $\bCh(D)$ in 
$\bHP^{0}(\cA)$. It is natural to ask for further conditions ensuring us that the CM cocycle represents 
the Connes-Chern character in $\bHP^{0}(\cA)$. Such conditions can be figured out by a careful look at the construction of the CM cocycle 
in~\cite{CM:LIFNCG} and~\cite{Hi:RITCM} (see~\cite{Po:SmoothCM} and Section~\ref{sec:spectral-triples}). For our 
purpose, it is useful to re-express these conditions in terms of heat-trace asymptotics (see Proposition~\ref{prop:CM.heat.trace.estimate}). Using the differentiable 
equivariant heat kernel asymptotics that we establish 
in Section~\ref{sec:Equivariant-heat-kernel-asymptotics},  it is fairly straightforward to check these conditions. Therefore, 
the Connes-Chern character of  $(C^{\infty}(M)\rtimes G, L^{2}_{g}(M,\sS),\ \sD_{g})$ is represented in 
$\bHP^{0}\left(C^{\infty}(M)\rtimes G\right)$ by the CM cocycle (Proposition~\ref{prop:CCC.CMC-represents-CCC}). 
 
We then are left with computing the CM cocycle of the equivariant Dirac spectral triple $( C^{\infty}(M)\rtimes G, 
L^{2}_{g}(M,\sS),\ \sD_{g})$. As it turns out, Azmi~\cite{Az:RMJM} computed the CM cocycle when $G$ is a 
finite group. However, she did not check 
the assumptions of~\cite{CM:LIFNCG} hold in her setting, and so she did not show that the CM cocycle represents the 
Connes-Chern character. Likewise, Chern-Hu~\cite{CH:ECCIDO} computed the CM 
cocycle defined as an equivariant periodic cyclic cochain, but they didn't verify the conditions of~\cite{CM:LIFNCG}. 
Note also that in these papers the computations are carried out by elaborating on the approach of 
Lafferty-Yu-Zhang~\cite{LYZ:TAMS} to the local equivariant index theorem, but some additional work  is
required to derive the heat-kernel asymptotics that are needed in the computation of the CM cocycle. 

Another goal of this paper is to give a new proof of the local equivariant index theorem of Donelly-Patodi~\cite{DP:T}, Gilkey~\cite{Gi:LNPAM}, and 
Kawasaki~\cite{Ka:PhD} (see also~\cite{Bi:ASITPA2, BV:BSMF, LYZ:TAMS, LM:DMJ}). More precisely, we produce a proof which, as an immediate byproduct,  yields a \emph{differentiable} version of the local equivariant index theorem. Once it is recasted in terms of heat kernel asymptotics (see 
Proposition~\ref{prop:CM.heat.trace.estimate}), the computation of the CM cocycle then becomes a simple corollary 
of this differentiable version of the local equivariant index theorem. Thus, hardly any additional work is required to compute the CM cocycle. 

Recall that the local equivariant index theorem provides us with a heat 
kernel proof of the equivariant index theorem of Atiyah-Segal-Singer~\cite{AS:IEO2, AS:IEO3}, 
which is a fundamental extension of both the index theorem of Atiyah-Singer~\cite{AS:IEO1, AS:IEO3} and the 
fixed-point formula of Atiyah-Bott~\cite{AB:LFT1, AB:LFT2}. Given a $G$-equivariant Hermitian vector bundle $E$, any $\phi\in G$ gives 
rise to a unitary operator $U_{\phi}$ of $L^{2}(M,\sS\otimes E)$. If we let $\sD_{\nabla^{E}}$ be the Dirac operator associated 
to a $G$-equivariant connection $\nabla^E$ on $E$, then the local equivariant index theorem for Dirac operators states that, for all $\phi \in G$, we have
\begin{equation}
   \lim_{t\rightarrow 0^{+}}   \Str \left[e^{-t\sD_{\nabla^{E}}^{2}}U_{\phi}\right] = (-i)^{\frac{n}{2}} \sum_{\substack{0\leq a \leq n\\ \textup{$a$ even}}} (2\pi)^{-\frac{a}{2}}\int_{M^{\phi}_a} 
  \hat{A}(R^{TM^{\phi}})\wedge \nu_{\phi}\left(R^{\cN^{\phi}}\right)\wedge 
      \Ch_{\phi}\left(F^{E}\right),
    \label{eq:Intro.LEIT}
\end{equation}where $M^{\phi}_a$ is the $a$-th dimensional submanifold component of the fixed-point set $M^\phi$ of $\phi$, and 
$ \hat{A}(R^{TM^{\phi}})$, $\nu_{\phi}\left(R^{\cN^{\phi}}\right)$, $\Ch_{\phi}\left(F^{E}\right)$ are polynomials in $\phi'$ and the respective curvatures $R^{TM^{\phi}}$, $R^{\cN^{\phi}}$, $F^{E}$ of the tangent space $TM^{\phi}$, its normal bundle $\cN^{\phi}$ and the bundle $E$ 
(see Section~\ref{sec:proof-key-thm} for their precise definitions). The differentiable version of the local index theorem gives a similar asymptotic for $\Str 
\left[Pe^{-t\sD_{\nabla^{E}}^{2}}U_{\phi}\right]$ as $t\rightarrow 0^{+}$, where $P$ is
a differential operator acting on the sections of $\sS\otimes E$ (see Theorem~\ref{thm:Diff.local.equiv.index.thm}  for the precise statement). 

Our approach to the proof of the local equivariant index theorem is an equivariant version of the approach 
of~\cite{Po:CMP} to the proof of the local index theorem. Namely, it combines the rescaling method of Getzler~\cite{Ge:SPLASIT} with an 
equivariant version of the Greiner-Hadamard approach to the heat kernel asymptotics~(\cite{Gr:AEHE}; see also~\cite{BGS:HECRM}). Given a positive elliptic 
differential operator $L$, the Greiner-Hadamard approach derives the short-time asymptotics for the heat kernel of $L$ by exploiting the well known fact that 
the heat semi-group enables us to invert the heat operator $L+\partial_{t}$ (see Section~\ref{sec.volterra}). As it turns out, the inverse 
$(L+\partial_{t})^{-1}$ lies in a natural class of \psidos, called Volterra \psidos\ (see~\cite{Gr:AEHE, Pi:COPDTV}). 
These operators have the Volterra property, in the sense that they are causal and time-translation invariant. The short-time asymptotics for the 
heat kernel then follows from the short-time asymptotics for kernels of Volterra \psidos. Observing that, given any 
differential operator $P$, the operator $P(L+\partial_{t})^{-1}$ is a Volterra \psido, the Greiner-Hadamard approach yields \emph{for 
free} a short-time asymptotic for the kernel of $Pe^{-tL}$. That is, it provides us for free with \emph{differentiable} heat kernel 
asymptotics. It should also be mentioned that the Volterra pseudodifferential calculus reduces the derivation of the the short-time asymptotics for kernels of 
Volterra \psidos\ to a mere application of Taylor's formula. 

There is no difficulty to extend the Greiner-Hadamard approach to the equivariant setting by working in tubular coordinates (see Proposition~\ref{TraceOfHeatKernelVB}). 
Like in the non-equivariant setting, the equivariant asymptotics for kernels of Volterra \psidos\ follow from an 
application of Taylor's formula (\emph{cf.}~Lemma~\ref{lem:Heat.asymptotic-Isymbol}). Once these equivariant asymptotics are established, we can extend the approach 
of~\cite{Po:CMP} to prove the local equivariant index theorem. More precisely, as observed in~\cite{Po:CMP}, the rescaling of 
Getzler~\cite{Ge:SPLASIT} defines a natural filtration on Volterra \psidos. Incidentally, this gives rise to a finer 
notion of order for Volterra \psidos\ (called Getzler order) and a notion of model operator (see Section~\ref{sec:proof-key-thm}). The existence of a limit 
in~(\ref{eq:Intro.LEIT}) and its computation by means of the inverse kernel of the model heat operator then 
follow from elementary considerations on Getzler orders and model operators of Volterra \psidos\ (see Lemmas~\ref{lem:GetzlerOrderParametrix} 
and~\ref{lem:AS.approximation-asymptotic-kernel}).  
The proof of the local equivariant index theorem is then completed by the computation of this inverse kernel. This is 
done by using Melher's formula for the heat kernel of an harmonic oscillator (see~\cite{LM:DMJ} and 
Section~\ref{sec:proof-key-thm}).   

 
 As it is based on considerations on Volterra \psidos, this approach to the proof of the local equivariant index 
 theorem actually yields a version of the local equivariant index 
 theorem for Volterra \psidos\ (Theorem~\ref{thm:LEIT.Volterra-LEIT}). In particular, specializing this theorem to operators of the form 
 $P(\sD_{E}^{2}+\partial_{t})^{-1}$ then provides us with the differentiable version of the local equivariant index 
 theorem (Theorem~\ref{thm:Diff.local.equiv.index.thm} ) which we need in the computation of the CM cocycle. This enables us to complete the computation of the 
 Connes-Chern character of the 
 equivariant Dirac spectral triple $(C^{\infty}(M)\rtimes G, L^{2}_{g}(M,\sS),\ \sD_{g})$. 

As mentioned the existence of the CM cocycle, as defined in~\cite{CM:LIFNCG}, requires various assumptions. One of these 
assumptions is the regularity condition which can be regarded as some kind of operator theoretic version of the 
scalarness of the principal symbol of the square of the Dirac operator (\emph{cf.}\ Definition~\ref{eq:CM.regularity}). However, there are natural geometric examples 
of spectral triples associated to hypoelliptic operators on contact or more generally Carnot manifolds which are not regular. Nevertheless, for these 
spectral triples the Connes-Chern character is represented by the short-time \emph{partie finie} (finite part) of the JLO cocycle. In the 
case of the spectral triple is regular the short-time \emph{partie finie} of the JLO cocycle agrees with that of the CM cocycle. 

In the case of a Dirac spectral triple the JLO cocycle actually has a short-time limit, and all previous 
computations of this limit implicitly use the regularity of the Dirac spectral triple as they involve commuting the 
heat semi-group with Clifford elements. As a further application of our version of the local equivariant index theorem 
alluded to above we give a computation of the short-time limit of the JLO cocycle of an equivariant Dirac spectral 
triple that does not assume the regularity of this spectral triple (Theorem~\ref{thm:JLO.PfJLO}). To our knowledge this seems to be the 
first computation of the short-time limit of the JLO cocycle of a Dirac spectral triple that does not use the 
regularity of this spectral triple. In some sense, this paves the way for the computation of the Connes-Chern character of spectral 
triples associated to hypoelliptic operators on contact or Carnot manifolds. 

It is believed that the approach of this paper to the proof of the local equivariant index theorem and computation of 
the Connes-Chern character is fairly general and can be applied to numerous geometric situations. In particular, it can 
be extended to spin manifolds with boundary equipped with $b$-metric (\emph{cf.}~Remark~\ref{rmk:JLO.boundary}; see also~\cite{Me:APSIT}) and to various 
(non-equivariant or equivariant) family settings or even bivariant settings (\emph{cf.}~Remarks~\ref{rmk:LEIT-family-settings} 
and~\ref{rmk:JLO.family}; see also~\cite{YWang:AJM, YWang:JKT14}). 

In this paper, the computation of the Connes-Chern character and short-time limit of the JLO cocycle is carried out in 
even dimension only. We refer to~\cite{Po:Odd} for an extension of these results to odd dimension (see also 
Remark~\ref{rmk:LEIT.odd} and~\ref{rmk:CM.odd}). Note also that~\cite{Po:Odd} contains applications to the construction 
of the eta cochain for a Dirac spectral triple (see also Remark~\ref{rmk:JLO.eta-cochain} on this point).

The paper is organized as follows. In Section~\ref{sec.volterra}, we review the Volterra pseudodifferential calculus and its relationship with the heat kernel. 
In Section~\ref{sec:Equivariant-heat-kernel-asymptotics}, we use the Volterra pseudodifferential calculus to derive equivariant heat kernel asymptotics. 
In Section~\ref{sec:proof-key-thm}, we present our proof of the local equivariant index theorem, which leads us to a 
version of the local equivariant index theorem for Volterra \psidos. In Section~\ref{sec:spectral-triples}, we review the construction of the 
Connes-Chern character and CM cocycle in the framework of smooth spectral triples. In Section~\ref{sec:CM.cocycle.HT.Asym}, we re-interpret 
the CM cocycle  in terms of heat kernel asymptotics. In Section~\ref{sec:Connes-Chern-conformal}, we determine the 
Connes-Chern character of an equivariant Dirac spectral triple by computing its CM cocycle. In Section~\ref{sec:JLO}, we use 
the version of the  local equivariant index theorem for Volterra \psidos\ to compute the short-time limit of the JLO 
cocycle of an equivariant Dirac spectral triple.

\section*{Acknowledgements}
The authors wish to thank Xiaonan Ma and Bai-Ling Wang for useful discussions related to the subject matter of 
this paper. They also would like to thank the following institutions for their hospitality during the
preparation of this manuscript: Mathematical Sciences Center of Tsinghua University, Research Institute of Mathematical Sciences of Kyoto 
University, University of California at Berkeley, and University Paris 6 (RP), Seoul National University (HW), Australian National University, Chern 
Institute of Mathematics of Nankai University, and Fudan University (RP+HW).

\section{Volterra Pseudodifferential Calculus and Heat Kernels}
\label{sec.volterra}
In this section, we recall the main definitions and properties of the Volterra pseudodifferential calculus and its 
relationship with the heat kernel of an elliptic operator. The pseudodifferential representation of the heat kernel appeared in~\cite{Gr:AEHE}, 
but some of the ideas can be traced back to Hadamard~\cite{Ha:LCPLPDE}. The presentation here follows closely that 
of~\cite{BGS:HECRM}.  
  
Let $(M^{n},g)$ be a compact Riemannian manifold and $E$ a Hermitian vector bundle over $M$. The metrics of $M$ and $E$ 
naturally define a continuous inner product on the space $L^{2}(M,E)$ of the $L^{2}$-sections of $E$. In addition, we let 
$L:C^{\infty}(M,E)\rightarrow C^{\infty}(M,E)$ be a selfadjoint 2nd 
order differential operator whose principal symbol is positive-definite. In particular, $L$ is elliptic.  The operator 
$L$ then generates a smooth semigroup $(0,\infty)\ni t \rightarrow e^{-tL}\in \cL(L^{2}(M,E))$ (called heat 
semigroup) such that, for all $u\in L^2(M,E)$, we have 
\begin{equation}
 \lim_{ t\to0^+}e^{-tL}u=u   \qquad \text{and} \qquad \frac{d}{dt} e^{-tL}u= -Le^{-tL}u  \quad \forall t>0.  
    \label{eq:Heat-inverse.semi-group-functions}
\end{equation}

In what follows, we shall make some abuse of notation by denoting by $E$  the vector bundle over $M\times \R$  obtained as the pullback of $E$ by the projection 
 the projection $(x,t)\rightarrow x$ (so that the fiber at $(x,t)$ is $E_x$).  The heat operator 
$L +\partial_{t}$ then acts on the sections of this vector bundle.  As is well known, the heat semigroup enables us to 
invert this operator. More precisely, let us denote by $C^{0}_{+}\left(\R,L^{2}(M,E)\right)$ the space of continuous 
functions $u:\R \rightarrow L^{2}(M,E)$ that are supported on some interval $I_{c}=[c,\infty)$, $c\in \R$. We also 
denote by $C^{\infty}_{+}(M\times \R, E)$ the subspace of $C^{\infty}(M\times \R,E)$ consisting of sections supported 
on $M\times I_{c}$ for some $c\in \R$. We shall regard $C^{\infty}_{+}(M\times \R, E)$ as a subspace of $C^{0}_{+}\left(\R,L^{2}(M,E)\right)$.  
We then define a linear operator $(L+\partial_{t})^{-1}:C^{0}_{+}\left(\R,L^{2}(M,E)\right)\rightarrow C^{0}_{+}\left(\R,L^{2}(M,E)\right)$  by
\begin{equation}
    (L+\partial_{t})^{-1}u(s):=\int_{0}^\infty e^{-tL} u(s-t)dt \qquad \forall u \in C^{0}_{+}(M\times \R,E). 
     \label{eq:volterra.inverse-heat-operator}
\end{equation}
The following result shows that $(L+\partial_{t})^{-1}$ is really an inverse for the heat operator. 

\begin{proposition}[\cite{Gr:AEHE, BGS:HECRM}]
For all $u\in C^{\infty}_{+}(M\times \R,E)$, we have
\begin{equation}
            (L+\partial_{t})^{-1}(L+\partial_{t})u=(L+\partial_{t})(L+\partial_{t})^{-1}u=u. 
            \label{eq:Heat.inverse-heat-equation}
            \end{equation}
\end{proposition}

\begin{remark}
\label{rem:toplogy.C+}
 The space $C^{0}_{+}\left(\R,L^{2}(M,E)\right)$ carries a natural locally convex topology defined as follows. For $I_{c}=[c,\infty)$, $c\in \R$, denote by 
 $C^{0}_{I_{c}}\left(\R,L^{2}(M,E)\right)$ the subspace of 
 $C^{0}\left(\R,L^{2}(M,E)\right)$ consisting of functions 
 supported on $I_{c}$. We equip $C^{0}_{I_{c}}\left(\R,L^{2}(M,E)\right)$ with the induced topology. The topology of 
 $C^{0}_{+}\left(\R,L^{2}(M,E)\right)$ then is the coarsest locally convex topology that makes continuous the inclusion of 
 $C^{0}_{I_{c}}\left(\R,L^{2}(M,E)\right)$ into $C^{0}_{+}\left(\R,L^{2}(M,E)\right)$ for all $c\in \R$. 
 We observe that with respect to this topology the inclusion of $C^{\infty}_{+}(M\times \R, E)$ into $C^{0}_{+}\left(\R,L^{2}(M,E)\right)$ is continuous. 
\end{remark}

\begin{remark}
Let $c\in \R$ and $u\in C^{0}_{I_{c}}\left(\R,L^{2}(M,E)\right)$. Then, for all $s \in \R$, 
\begin{equation*}
    (L+\partial_{t})^{-1}u(s)=\int_{\{0\leq t \leq s-c\}}e^{-tL} u(s-t)dt. 
\end{equation*}
Therefore, we see that $ (L+\partial_{t})^{-1}u\in C^{0}_{I_{c}}\left(\R,L^{2}(M,E)\right)$ and, for $s\geq c$, we have
\begin{equation*}
    \|(L+\partial_{t})^{-1}u(s)\|\leq \int_{0}^{s-c}\|e^{-tL} u(s-t)\|dt \leq (s-c) \sup\{\|u(s')\|; \ c\leq s'\leq s\}. 
\end{equation*}
We then deduce that $(L+\partial_{t})^{-1}$ induces a continuous endomorphism of $C^{0}_{I_{c}}\left(\R,L^{2}(M,E)\right)$ for all $c 
\in \R$, and hence is a continuous operator from $C^{0}_{+}\left(\R,L^{2}(M,E)\right)$ to itself. 
\end{remark}

\begin{remark}\label{rem:Volterra.smoothness-inverse-heat-operator0}
 The space $ C^{\infty}_{+}(M\times  \R, E)$ carries a natural locally convex topology defined in a similar fashion as the topology of  $C^{0}_{+}\left(\R,L^{2}(M,E)\right)$ described 
 in Remark~\ref{rem:toplogy.C+}.  Furthermore, by using the  ellipticity of $L$ it can be further shown that with respect to this topology the operator $(L+\partial_{t})^{-1}$ induces a 
 continuous endomorphism of $C^{\infty}_{+}(M\times \R, E)$ (see also 
 Remark~\ref{rem:Volterra.smoothness-inverse-heat-operator} below on this point). 
\end{remark}

Let us denote by $E\boxtimes E^{*}$ the vector bundle over $M\times M\times \R$ whose fiber at $(x,y,t)\in M\times 
M\times \R$ is $\op{Hom}(E_{y},E_{x})$. We define the \emph{heat kernel} $k_{t}(x,y)$, $t>0$, as the smooth section of 
$E\boxtimes E^{*}$ over $M\times M\times (0,\infty)$ such that 
\begin{equation}
    e^{-tL}u(x)=\int_{M}k_{t}(x,y)u(y)|dy| \qquad \forall u \in L^{2}(M,E),
    \label{eq:Heat.heat-kernel-smooth-function}
\end{equation}where $|dy|$ is the Riemannian density defined by $g$ on $M$. That is, $k_{t}(x,y)|dy|$ is the Schwartz 
kernel of $e^{-tL}$. In addition, as $(L+\partial_{t})^{-1}$ is a continuous linear operator from  
$C^{0}_{+}\left(\R,L^{2}(M,E)\right)$ to itself, we may regard it as a continuous linear operator from $C^{\infty}_{c}(M\times 
\R, E)$ to $C^{0}(M\times \R, E)$. As such it has a kernel $k_{(L+\partial_{t})^{-1}}(x,s,y,t)\in C^{0}(M_{x}\times 
\R_{s},E)\hotimes \cD'(M_{y}\times \R_{t},E)$ such that
\begin{equation*}
    (L+\partial_{t})^{-1}u(x,s)=\acou{k_{(L+\partial_{t})^{-1}}(x,s,y,t)}{u(y,t)} \qquad \forall u \in C^{\infty}_{c}(M\times \R,E). 
\end{equation*}
We then observe that, at the level of Schwartz kernels, the definition~(\ref{eq:Heat.inverse-heat-equation}) means that
\begin{equation}
  k_{(L+\partial_{t})^{-1}}(x,s,y,t)= \left\{ 
  \begin{array}{ll}
    k_{s-t}(x,y)   & \text{for $s-t>0$},  \\
      0 &  \text{for $s-t<0$}. 
  \end{array}  \right.
     \label{eq:Heat.KQ0-heat-kernel}
\end{equation}It then follows that $(L+\partial_{t})^{-1}$ has the Volterra property in the sense of the following definition. 

\begin{definition}[\cite{Pi:COPDTV}]\label{def:Volterra-property}
A linear operator $Q:C^{\infty}_{c}(M\times \R,E)\rightarrow C^{0}(M\times \R,E)$ has the Volterra property when it satisfies the following properties:
\begin{enumerate}
\item[(i)] \emph{Time-Translation Invariance}. For all $u\in C^{\infty}_{c}(M\times \R,E)$ and $c\in \R$, 
\begin{equation*}
 Qu_c(x,t)=(Qu)(x,t+c) \qquad \text{for all $(x,t)\in M\times \R$},
\end{equation*}where $u_c(x,t)=u(x,t+c)$. 

\item[(ii)]  \emph{Causality Principle}. For all $u\in C^{\infty}_{c}(M\times \R,E)$ and $t_0\in \R$, 
\begin{equation*}
 \text{$u=0$ on $M\times (-\infty,t_0]$} \ \Longrightarrow \ \text{$Qu(x,t_0)=0$ for all $x\in M$}.
\end{equation*}
\end{enumerate}
\end{definition}

\begin{remark}
\label{rem:Volterra.kernel}
If we further assume that $Q$ is a continuous linear operator from $C^\infty_c(M\times \R, E)$ to $C^\infty(M\times \R, E)$, then the Volterra property implies that 
the Schwartz kernel of $Q$ takes the form,
\begin{equation*}
 k_Q(x,s,y,t)=K_Q(x,y,s-t),
\end{equation*}for some $K_Q(x,y,t)$ in $C^{\infty}(M,E)\hotimes\cD'(M\times \R,E)$ such that $K_Q(x,y,t)=0$ for $t<0$. 
The distribution $K_Q(x,y,t)$ is then called the \emph{Volterra kernel} of $Q$. We also observe that in this case $Q$ uniquely extends to a continuous linear operator from 
$C^\infty_+(M\times \R,E)$ to itself. 
\end{remark}


The Volterra \psido\ calculus aims at constructing a class of \psidos\ which is a natural receptacle for the inverse of the heat 
operator. The idea is to modify the classical \psido\ calculus in order to take into account the following properties: 
\begin{itemize}
    \item[(i)]  The aforementioned Volterra property.  

    \item[(ii)] The parabolic homogeneity of the heat operator $L+ \partial_{t}$, i.e., the homogeneity with respect 
             to the dilations,
\begin{equation*}
    \lambda.(\xi,\tau):=(\lambda\xi,\lambda^{2}\tau) \qquad \forall (\xi,\tau)\in \R^{n+1} \ \forall \lambda \in 
    \R^{*}.
\end{equation*}             
\end{itemize}
In what follows, for $G\in \cS'(\R^{n+1})$  and $\lambda\neq 0$, we denote by $G_{\lambda}$ the distribution in  
    $\cS'(\R^{n+1})$ defined by   
    \begin{equation}
        \acou{G_{\lambda}( \xi,\tau)}{u(\xi,\tau)} :=    |\lambda|^{-(n+2)} 
            \acou{G(\xi,\tau)} {u(\lambda^{-1}\xi, \lambda^{-2}\tau)} \quad \forall u \in \cS(\R^{n+1}). 
    \end{equation}
In addition, we denote by $\C_{-}$ the complex halfplane $\{\Im \tau <0\}$ with closure $\overline{\C}_{-}$. 

\begin{definition}%
    A distribution $ G\in \cS'(\R^{n+1})$ is (parabolic) 
homogeneous of degree $m$, $m\in \Z$, when
\begin{equation}
\label{eq:homogeneous.Voterra.symbol}
    G_{\lambda}=\lambda^m G \qquad \forall \lambda \in \R\setminus 0.
\end{equation}
\end{definition}
 
We mention the following version of Paley-Wiener-Schwartz Theorem. 
 
\begin{lemma}[{\cite[Prop.~1.9]{BGS:HECRM}}] \label{lem:volterra.volterra-extension}
Let $q(\xi,\tau)\in C^\infty((\R^{n}\times\R)\setminus0)$ be a parabolic homogeneous 
function of degree $m$, $m\in \Z$, such that 
\begin{itemize}
    \item[(i)] $q(\xi,\tau)$ extends to a continuous function on $(\R^{n}\times\overline{\C}_{-})\setminus0$ in 
    such way to be holomorphic with respect to the variable $\tau$ on $\R^{n}\times\C_{-}$.
\end{itemize} 
\noindent Then  there is a unique $G\in \cS'(\R^{n+1})$ agreeing with $q$ on $\R^{n+1}\setminus 0$ and such that
\begin{itemize}
    \item[(ii)] $G$ is homogeneous of degree $m$.

    \item[(iii)] The inverse Fourier transform $\check G(x,t)$ vanishes for $t<0$. 
\end{itemize}
\end{lemma}
\begin{remark}[See~\cite{BGS:HECRM}]\label{eq:Heat.homogeneity-inverse-Fourier}
    The homogeneity of $G$ implies that 
    \begin{equation*}
        \check{G}_{\lambda}=|\lambda|^{-(n+2)}\lambda^{-m}\check{G} \qquad \forall \lambda \in \R^{*}.
    \end{equation*}
    In particular, we see that $\check{G}$ is \emph{positively} homogeneous of degree $-(m+n+2)$. 
\end{remark}

Let $U$ be an open subset of $\R^{n}$. We define Volterra symbols and Volterra \psidos\ on 
$U\times\R^{n+1}\setminus 0$ as follows. 

\begin{definition}
\label{def.volterra.asymptotic.expansion}
    $S_{\op v}^m(U\times\R^{n+1})$, $m\in\Z$,  consists of smooth functions $q(x,\xi,\tau)$ on 
    $U\times\R^n\times\R$ with an asymptotic expansion  $q(x,\xi,\tau) \sim \sum_{j\geq 0} q_{m-j}(x,\xi,\tau)$, where 
    \begin{itemize}
        \item[-] $q_{l}(x,\xi,\tau)\in C^{\infty}(U\times[(\R^n\times\R)\setminus0])$ is a homogeneous Volterra symbol of degree~$l$, 
    i.e., $q_{l}$ is  parabolic homogeneous of degree $l$ and  satisfies the 
    property (i) in Lemma~\ref{lem:volterra.volterra-extension} with respect to the variables $\xi$ and $\tau$. \smallskip  
    
        \item[-] The asymptotic expansion is meant in the sense that, for all  compact sets $K\subset U$, integers $N$ and $k$ and multi-orders $\alpha$ 
	and $\beta$, there is a constant $C_{NK\alpha\beta 
    k}>0$ such that,  for  all $(x,\xi,\tau)\in K\times \R^n\times\R$ with $|\xi|+|\tau|^{\frac12}\geq 1$, we have
            \begin{equation}
                \biggl|\partial^{\alpha}_{x}\partial^{\beta}_{\xi} \partial^k_{\tau}\biggl(q-\sum_{j< N} 
            q_{m-j}\biggr)(x,\xi,\tau) \biggr| 
                \leq C_{NK\alpha\beta k} (|\xi|+|\tau|^{1/2})^{m-N-|\beta|-2k}.
                          \label{eq:volterra.asymptotic-symbols}
            \end{equation}
    \end{itemize}
\end{definition}

In what follows, for a symbol $q(x,\xi,\tau)\in S^m_{\op v}(U\times\R)$ we shall denote by $q(x,D_{x},D_{t})$ the operator from 
$C^{\infty}_{c}(U\times \R)$ to $C^{\infty}(U\times \R)$ defined by
\begin{equation*}
    q(x,D_{x},D_{t})u(x,t)= (2\pi)^{-(n+1)}\iint e^{i(x\cdot \xi +t\tau)}q(x,\xi,\tau)\hat{u}(\xi,\tau)d\xi d\tau \quad \forall u 
    \in C^{\infty}_{c}(U\times \R).
\end{equation*}

\begin{definition}\label{def:volterra.PsiDO}
    $\pvdo^m(U\times\R)$, $m\in\Z$,  consists of continuous linear operators 
    $Q$ from $C_{c}^\infty(U_{x}\times\R_{t})$ to $C^\infty(U_{x}\times\R_{t})$ such that
    \begin{itemize}
        \item[(i)] $Q$ has the Volterra property in the sense of Definition~\ref{def:Volterra-property}.  
    
        \item[(ii)] $Q$ can be put in the form,
	\begin{equation}
	    Q=q(x,D_{x},D_{t})+R,
	    \label{eq:Heat.Volterra-PsiDO-symbol-R}
	\end{equation}
	for some symbol $q(x,\xi,\tau)\in S^m_{\op v}(U\times\R)$ and some smoothing operator  $R$.
    \end{itemize}
\end{definition}

\begin{remark}
 For any operator $Q\in \pvdo^m(U\times\R)$ there is a unique Volterra kernel $K_{Q}(x,y,t)$ in 
 $C^{\infty}(U,\cD'(\UR))$ such that $K_{Q}(x,y,t)=0$ for $t<0$ and 
\begin{equation*}
    Qu(x,s)=\acou{K_{Q}(x,y,s-t)}{u(y,t)}\qquad \forall u \in C^{\infty}_{+}(\UR).
\end{equation*}In fact, if we put $Q$ in the form~(\ref{eq:Heat.Volterra-PsiDO-symbol-R}) and we denote by $k_{R}(x,s,y,t)$ the Schwartz kernel of the 
smoothing operator $R$ as defined in~(\ref{eq:Heat.heat-kernel-smooth-function}), then
\begin{equation*}
    K_{Q}(x,y,t)=\check{q}(x,x-y,t)+k_{R}(x,0,y,-t).
\end{equation*}
\end{remark}

\begin{example}\label{ex:heat-inverse.diff-op}
    Let $P$ be a differential operator of order $2$ on $U$ with principal symbol $p_{2}(x,\xi)$. Then 
    the operator $P+\partial_{t}$ is a Volterra \psido\ of order $2$ with principal symbol 
    $p_{2}(x,\xi)+i\tau$. In particular, if $p_{2}(x,\xi)>0$ for all $(x,\xi)\in U\times (\Rno)$, then 
    $p_{2}(x,\xi)+i\tau\neq 0$ for all $(x,\xi,\tau)\in U\times[(\R^n\times \overline{\C_{-}}\setminus 0)]$. 
\end{example}

The following definition provides us with further examples of Volterra \psidos. 

\begin{definition}\label{def:volterra.homogeneous-PsiDO} 
Let $q_{m}(x,\xi,\tau) \in C^\infty(U\times(\R^{n+1}\setminus 0))$ be a homogeneous Volterra 
symbol of order $m$ and let $G_{m}(x,\xi,\tau)\in C^\infty\left(U,\cS'(\R^{n+1})\right)$ denote its unique homogeneous extension given by 
    Lemma~\ref{lem:volterra.volterra-extension}. Then 
    \begin{itemize}
        \item[-]  $\check q_{m}(x,y,t)$ is the inverse Fourier transform of $G_{m}(x,\xi,\tau)$ 
	w.r.t.~the last $n+1$ variables.  
    
        \item[-]  The operator $q_{m}(x,D_{x},D_{t}):C^{\infty}_{c}(\UR)\rightarrow C^{\infty}(\UR)$ is defined by
        \begin{equation}
            q_{m}(x,D_{x},D_{t})u(x,s):=\acou{\check{q}_{m}(x,x-y,s-t)}{u(y,t)} \qquad \forall u\in C^{\infty}_{+}(U\times \R).
            \label{eq:Heat.homogeneous-Volterra-PsiDOs}
        \end{equation}
    \end{itemize}
\end{definition}

\begin{remark}
   It follows from the proof of~\cite[Prop.~1.9]{BGS:HECRM} that the 
   homogeneous extension $G_{m}(x,\xi,\tau)$ depends smoothly on $x$, i.e., it belongs to
   $C^\infty\left(U,\cS'(\R^{n+1})\right)$. 
\end{remark}

\begin{lemma}
The operator $q_{m}(x,D_{x},D_{t})$ is a Volterra \psido\ of order $m$ with symbol $q\sim q_{m}$.    
\end{lemma}
\begin{proof}[Sketch of Proof]
Set $Q=q_{m}(x,D_{x},D_{t})$. Since $\check{q}_{m}(x,y,t)$ belongs to $C^\infty\left(U,\cS'(\R^{n+1})\right)$, it follows from~(\ref{eq:Heat.homogeneous-Volterra-PsiDOs}) that 
 the operator $q_{m}(x,D_{x},D_{t})$ is continuous and satisfies the Volterra property. Denote by  $G_{m}(x,\xi,\tau)$  the unique homogeneous extension of 
 $q_{m}(x,\xi,\tau)$ given by Lemma~\ref{lem:volterra.volterra-extension}. In addition, let $\varphi \in C^{\infty}_{c}(\R^{n+1})$ be such that 
 $\varphi(\xi,\tau)=1$ near $(\xi,\tau)=(0,0)$. Then the symbol $\tilde{q}_{m}(x,\xi,\tau):=\left(1-\varphi(\xi,\tau)\right)q_{m}(x,\xi,\tau)$ lies 
    in $S_{\op 
    v}^m(U\times\R^{n+1})$ and we have
 \begin{equation*}
     K_{Q}(x,y,t)=(\tilde{q}_{m})^{\vee}(x,y,t)+(\varphi G_{m})^{\vee}(x,y,t),
 \end{equation*}Observe that $(\varphi G_{m})^{\vee}(x,y,t)$ is smooth since this is the inverse Fourier transform of a 
 compactly supported function. Thus $Q$ agrees with $\tilde{q}_{m}(x,D_{x},D_{t})$ up to a smoothing operator, and 
 hence is a Volterra \psido\ of order $m$. Furthermore, it has symbol $\tilde{q}_{m}\sim q_{m}$. The proof is complete.
\end{proof}

We gather the main properties of Volterra \psidos\ in the following statement. 

\begin{proposition}[\cite{Gr:AEHE, Pi:COPDTV, BGS:HECRM}]
\label{prop:Volterra-properties} 
The following properties hold.
  \begin{enumerate}
      
      \item \emph{Pseudolocality}. For any $Q\in \pvdo^{m}(\UR)$, the Volterra kernel $K_{Q}(x,y,t)$ is smooth on the open subset $\{(x,y,t)\in 
      M\times M\times \R; \ x\neq y \ \text{or} \ t \neq 0\}$.\smallskip 
       
        \item \emph{Proper Support}. For any $Q\in \pvdo^{m}(\UR)$ there exists $Q'\in \pvdo^{m}(\UR)$ such that 
        $Q'$ is properly 
        supported and $Q-Q'$ is a smoothing operator.\smallskip 
       
      \item  \emph{Composition}. Let $Q_{j}\in \pvdo^{m_{j}}(\UR)$, $j=1,2$, have symbol $q_{j}$ and assume that $Q_1$ or $Q_2$ is properly supported.
Then $Q_{1}Q_{2}$ lies in $\pvdo^{m_{1}+m_{2}}(\UR)$
    and has symbol  $q_{1}\#q_{2} \sim \sum \frac{1}{\alpha!} 
\partial_{\xi}^{\alpha}q_{1} D_{\xi}^\alpha q_{2}$.\smallskip 
  
      \item   \emph{Parametrices}.  Any $Q\in \pvdo^{m}(U\times\R)$ admits a parametrix 
    in $\pvdo^{-m}(U\times\R)$ if and only if its principal symbol is nowhere vanishing on 
    $U\times[(\R^n\times \overline{\C_{-}})\setminus 0]$.\smallskip 
  
      \item  \emph{Diffeomorphism Invariance}. Let $\phi$ be 
    a diffeomorphism from $U$ onto an open 
subset $V$ of $\R^n$. Then for any $Q \in \pvdo^{m}(U\times\R)$ the operator 
$(\phi\oplus \op{id}_{\R})_{*}Q$  is contained in $\pvdo^{m}(V\times\R)$. 
  \end{enumerate}
\end{proposition}

\begin{remark}
   Most properties of Volterra \psidos\ can be proved in the same way as with classical \psidos\ or by observing that Volterra 
   \psidos\ are \psidos\ of type $(\frac{1}{2},0)$ in the sense of~\cite{Ho:ALPDO3}. One important exception is the asymptotic completeness, i.e., given 
   homogeneous Volterra symbols $q_{m-j}$ of degree $m-j$, $j=0,1,\ldots$, there is a Volterra \psido\ with 
   symbol $q\sim \sum q_{m-j}$. This property is a crucial ingredient in the parametrix construction in 
   Proposition~\ref{prop:Volterra-properties} (see~\cite{Po:JAM} for a discussion on this point). 
\end{remark}

As usual with \psidos, the asymptotic expansion~(\ref{eq:volterra.asymptotic-symbols}) for the symbol of a given \psido\ can be translated in terms of an asymptotic expansion for the Schwartz kernel in terms of distributions that are smoother and smoother. For Volterra \psidos\ we obtain the following result.

\begin{proposition}[\cite{Gr:AEHE, Pi:COPDTV, BGS:HECRM}]\label{prop:Heat.asymptotic-kernels}
 Let $Q\in \pvdo^{m}(\UR)$ have symbol $q\sim \sum_{j\geq 0}q_{m-j}$. Then, for all $N\in 
 \N_{0}$, there is $J\in \N$ such that
 \begin{equation}
     K_{Q}(x,y,t)= \sum_{j\leq J}\check{q}_{m-j}(x,x-y,t) \ \bmod C^{N}(U\times\UR). 
\label{eq:Heat.asymptotic-kernels}
 \end{equation}
\end{proposition}
\begin{proof}[Sketch of Proof] As Volterra \psidos\ are \psidos\ of type $(\frac{1}{2},0)$, the kernel of a Volterra 
    \psido\ of order $\leq -(n+2+2N)$ is $C^{N}$ (see~\cite{Ho:ALPDO3}). Let us choose $J$ so that $m-J\leq -(n+1+2N)$, then  
    $Q-\sum_{j\leq J} q_{m-j}(x,D_{x},D_{t})$ is a Volterra \psidos\ with symbol $q^{J}\sim \sum_{j\geq 
    J+1}q_{m-j}$, and hence it has order $m-J-1\leq -(n+2+2N)$. Therefore, its kernel is $C^{N}$. This proves the result.  
\end{proof}

The invariance property in Proposition~\ref{prop:Volterra-properties} enables us to define Volterra \psidos\ on products of manifolds with $\R$ and acting on sections of vector bundles. 
In particular, we can define Volterra \psidos\ on the manifold $M\times \R$ and acting on the sections of our vector bundle $E$ (seen as a vector bundle over $M\times \R$). 
All the aforementioned  properties of Volterra \psidos\  hold \emph{verbatim} in this 
context. We shall denote by $\pvdo^{m}(M\times \R,E)$ the space of Volterra \psidos\ of order $m$ on $M\times \R$ acting on sections of $E$.  

\begin{remark}
 By Remark~\ref{rem:Volterra.kernel} Volterra \psidos\ on $M\times \R$ uniquely extend to continuous operators $C^\infty_+(M\times \R, E)$ to itself. Therefore, the composition of such 
 Volterra \psidos\ always make sense as a continuous operator from $C^\infty_+(M\times \R, E)$ to itself. This enables us to drop the proper support assumption in part (3) of 
 Proposition~\ref{prop:Volterra-properties}.  
 \end{remark}

If $Q\in \pvdo^{m}(M\times \R,E)$, then there is a unique $K_{Q}(x,y,t)\in C^{\infty}(M\times \R)\hotimes \cD'(M,E)$ 
such that $K_{Q}(x,y,t)=0$ for $t<0$ and 
\begin{equation*}
    Qu(x,s)=\acou{K_{Q}(x,y,s-t)}{u(y,t)} \qquad \forall u \in C^{\infty}_{+}(M\times \R,E).
\end{equation*}We shall refer to $K_{Q}(x,y,t)$ as the \emph{Volterra kernel} of $Q$. Proposition~\ref{prop:Volterra-properties} ensures us that $K_{Q}(x,y,t)$ is smooth for $t\neq 0$. 
Therefore, on $M\times M\times \R^{*}$ we 
may regard  $K_{Q}(x,y,t)$ as a smooth section of $E\boxtimes E^{*}$ over $M\times M\times \R^{*}$ such that
\begin{equation*}
    \acou{K_{Q}(x,y,t)}{u(y,t)} =\int_{M\times \R}K_{Q}(x,y,t)u(y,t)|dy|dt \qquad \forall u \in 
    C^{\infty}_{+}(M\times \R^{*},E),
\end{equation*}where in the l.h.s.~$K_{Q}(x,y,t)$ is an element of 
$C^{\infty}(M\times \R)\hotimes \cD'(M,E)$ and in the r.h.s.~it is a smooth section of $E\boxtimes E^{*}$. 

It follows from Example~\ref{ex:heat-inverse.diff-op} and Proposition~\ref{prop:Volterra-properties} that the heat operator $L+\partial_{t}$ admits a  parametrix in 
$\pvdo^{-2}(M\times \R,E)$. Comparing such a parametrix with the inverse $(L+\partial_{t})^{-1}$ defined 
by~(\ref{eq:volterra.inverse-heat-operator}) and using~(\ref{eq:Heat.KQ0-heat-kernel}) we arrive at the following result.

\begin{proposition}[\cite{Gr:AEHE, Pi:COPDTV}, {\cite[pp.~363-362]{BGS:HECRM}}] \label{thm:Heat.inverse-heat-operator-PsiDO} The operator 
    $(L+\partial_{t})^{-1}$ defined by~(\ref{eq:volterra.inverse-heat-operator})  is a Volterra \psido\ of 
    order~$-2$. Moreover, we have
     \begin{equation}
         k_{t}(x,y)=K_{(L+\partial_{t})^{-1}}(x,y,t) \qquad \forall t>0.
    \label{eq:Heat.K-heat-inverse-heat-kernel}
\end{equation}    
 \end{proposition}
 
\begin{remark}\label{rem:Volterra.smoothness-inverse-heat-operator}
 The fact that $(L+\partial_{t})^{-1}$ is Volterra \psido\  implies that it induces a continuous linear operator from $C^\infty_+(M\times \R, E)$ to itself 
 (\emph{cf}.~Remark~\ref{rem:Volterra.smoothness-inverse-heat-operator0}). 
\end{remark}

 Proposition~\ref{thm:Heat.inverse-heat-operator-PsiDO} provides us with a representation of the heat kernel 
 as the (Volterra) kernel of a Volterra \psido. Combining it 
 with~(\ref{eq:Heat.asymptotic-kernels}) enables us to describe the asymptotic behaviour of $k_{t}(x,x)$ as $t \rightarrow 0^{+}$ 
 (see~\cite{Gr:AEHE, BGS:HECRM} and next section).  More generally, we have the following result.

\begin{proposition}\label{thm:Heat.P-inverse-heat-operator-PsiDO} Let $P:C^{\infty}(M,E) \rightarrow C^{\infty}(M,E)$ be a differential operator of order $m$. For $t>0$ 
denote by $h_{t}(x,y)$ the kernel of $Pe^{-tL}$ defined as in~(\ref{eq:Heat.heat-kernel-smooth-function}). Then, $P(L+\partial_{t})^{-1}$ is a Volterra \psido\ of order $m-2$, and we have
     \begin{equation*}
    h_{t}(x,y)=K_{P(L+\partial_{t})^{-1}}(x,y,t) \qquad \forall t>0. 
\end{equation*}    
 \end{proposition}
\begin{proof}
   As the order of $P$ as a Volterra \psido\ is $m$, it follows from Proposition~\ref{prop:Volterra-properties} that $P(L+\partial_{t})^{-1}$ is a Volterra \psido\ of order $m-2$. Moreover, we have
   \[h_{t}(x,y)=P_{x}k_{t}(x,y)=P_{x}K_{(L+\partial_{s})^{-1}}(x,y,t)=K_{P(L+\partial_{s})^{-1}}(x,y,t).\]
   The proof is complete.  
\end{proof}

\section{Equivariant Heat Kernel Asymptotics}
\label{sec:Equivariant-heat-kernel-asymptotics}
In this section, we explain how to use the representation of the heat kernel given in the previous section to derive 
equivariant heat kernel asymptotics. We shall keep using the notation of the previous section. In particular, we let 
$(M^{n},g)$ be a Riemannian manifold. In addition, we let $G$ be a group of isometric diffeomorphisms and let $E$ be a 
$G$-equivariant Hermitian vector bundle over $M$. Given $\phi \in G$, we denote by $\phi^{E}$ the unitary vector bundle isomorphism from 
$E$ onto $\phi^{*}E$ induced by the action of $E$. This defines a unitary operator $U_{\phi}:L^{2}(M,E)\rightarrow L^{2}(M,E)$ by
\begin{equation}
\label{eq:unitary.operator.on.L^2}
    U_{\phi}u(x)=\phi^{E}\left(\phi^{-1}(x)\right)\!u\left(\phi^{-1}(x)\right) \qquad \forall u \in L^{2}(M,E). 
\end{equation}
The goal of this section is to derive a short-time asymptotics equivariant traces $\Tr \left [ 
Pe^{-tL}U_{\phi}\right]$, where $\phi$ ranges over $G$ and $P$ ranges over differential operator acting on the sections of $E$.

Let $\phi \in G$. For $t>0$ let $h_{t}(x,y)$ be the kernel of $Pe^{-tL}$ as defined in~(\ref{eq:Heat.heat-kernel-smooth-function}). We observe that the Schwartz kernel of 
$Pe^{-tL}U_{\phi}$ is $h_{t}(x,\phi(y))\phi^{E}(x)$, and so we have
\begin{equation}
    \Tr \left [ Pe^{-tL}U_{\phi}\right] = \int_{M} \tr_{E} \left [h_{t}(x,\phi(x)) \phi^{E}(x)\right]|dx|= 
    \int_{M} \tr_{E}\left [ \phi^{E}(x)h_{t}(x,\phi(x))\right]|dx|. 
    \label{eq:Heat.equivariant-trace-formula}
\end{equation}We are thus led to understand the short-time behavior of $h_{t}(x,\phi(x))$. Since by Proposition~\ref{thm:Heat.P-inverse-heat-operator-PsiDO} 
we can represent $h_{t}(x,y)$ as the kernel of a Volterra \psido, we shall more generally study the short-time behavior 
of $K_{Q}(x,\phi(x),t)$, where $Q \in \pvdo^{m}(M\times \R, E)$, $m \in \Z$.  

In what follows, we denote by $M^{\phi}$ the fixed-point 
set of $\phi$, and  for $a =0,\ldots, n$, we let $M_a^{\phi}$ be the subset of $M^{\phi}$ consisting of fixed-points $x$ at which $\phi'(x)-1$ has rank $n-a$, i.e., 
the eigenvalue $1$ of $\phi'(x)$ has multiplicity $a$. 
Therefore, we have the disjoint-sum decomposition,
\begin{equation*}
    M^{\phi}=\bigsqcup_{0\leq a \leq n}M_a^{\phi}. 
\end{equation*}
In addition, we pick some $\epsilon_{0}\in (0,\rho_{0})$, where $\rho_{0}$ is the injectivity radius of $(M,g)$.  

Let $x_{0}$ be a point in some component $M_a^{\phi}$.  Denote by $B_{x_{0}}(\epsilon_0)$ the ball of radius $\epsilon_0$ around the origin in 
$T_{x_{0}}M$. Then $\exp_{x_{0}}$ induces a diffeomorphism from $B_{x_{0}}(\epsilon_0)$ onto an open neighborhood 
$U_{\epsilon_0}$ of $x_{0}$ in $M$. Moreover, as $\phi$ is an isometry, for all $X\in B_{x_{0}}(\epsilon_0)$, we have
\begin{equation}
    \phi\left(\exp_{x_{0}}(X)\right)=\exp_{\phi(x_{0})}(\phi'(x_{0})X)=\exp_{x_{0}}(\phi'(x_{0})X).
    \label{eq:Heat.action-phi-Nphi}
\end{equation}Thus under $\exp_{x_{0}|B_{x_{0}}(\epsilon_0)}$ the diffeomorphism $\phi$ corresponds to $\phi'(x_{0})$, and 
hence $M^{\phi}\cap U_{\epsilon_0}$ is identified with $B_{x_{0}}^{\phi}(\epsilon_0):=B_{x_{0}}(\epsilon_0)\cap \ker (\phi'(x_{0})-1)$. 
Incidentally, the tangent bundle $TM^{\phi}_{|M^{\phi}\cap U_{\epsilon_0}}$ is identified with $B_{x_{0}}^{\phi}(\epsilon_0)\times 
\ker (\phi'(x_{0})-1)$ and the normal bundle  $\left(TM^{\phi}\right)^{\bot}_{|M^{\phi}\cap U_{\epsilon_0}}$ is identified with $B_{x_{0}}^{\phi}(\epsilon_0)\times 
\ker (\phi'(x_{0})-1)^{\bot}$. Note also that when $a=0$ this shows that $x_{0}$ is an isolated 
fixed-point. It follows from this that  each component $M_a^{\phi}$ is a (closed) submanifold of dimension $a$ of $M$ and over $M_a^{\phi}$ 
the set $\cN^{\phi}:=\sqcup_{x\in 
M^{\phi}}\ker (\phi'(x)-1)^{\bot}$ can be organized as a smooth vector bundle. We denote by $\pi:\cN^{\phi}\rightarrow 
M^{\phi}$ the corresponding canonical map. We shall refer to $\cN^{\phi}$ as the normal bundle of $M^{\phi}$. 
Note that  $\phi'$ induces (over each component $M_a^{\phi}$) an isometric vector bundle isomorphism of $\cN^{\phi}$ onto 
itself. 

As is well known, using the normal bundle $\cN^{\phi}$ we can construct a tubular neighborhood of $M^{\phi}$ as follows.  Let 
$\cN^{\phi}(\epsilon_0)$ be the ball bundle of $\cN^{\phi}$ of radius $\epsilon_0$ around the zero-section. Then the map
$\cN^{\phi}(\epsilon_0)\ni X \rightarrow \exp_{\pi(x)}(X)$ is a homeomorphism from $\cN^{\phi}(\epsilon_0)$ onto an open tubular neighborhood $V_{\epsilon_0}$ of 
$M^{\phi}$ in $M$. Moreover, over each submanifold $M_a^{\phi}$, $a=0,1,\ldots,n$, this maps induces a 
diffeomorphism from $\cN^{\phi}(\epsilon_0)_{|M_a^{\phi}}$ onto its image. Let us fix some $\epsilon\in (0,\epsilon_{0})$ and let $(x,t)\in M^{\phi}\times (0,\infty)$. Observe that, in view of~(\ref{eq:Heat.action-phi-Nphi}), for 
all $v \in N_{x}^{\phi}(\epsilon)$, we have
\begin{equation*}
 K_{Q}\left(\exp_{x}v,\exp_{x}(\phi'(x)v),t\right)=K_{Q}\left(\exp_{x}v, \phi(\exp_{x}v),t\right). 
\end{equation*}For $x \in M^{\phi}$ and $t>0$ set
\begin{equation}
\label{eq:definition.I_Q.Volterra.kernel}
    I_{Q}(x,t):=\phi^{E}(x)^{-1}\int_{N_{x}^{\phi}(\epsilon)} \phi^{E}\left(\exp_{x}v\right) K_{Q}\left(\exp_{x}v, 
    \exp_{x}(\phi'(x)v),t\right)|dv| . 
\end{equation}
This defines a smooth section of  $\End E$ over $M^{\phi}\times (0,\infty)$, since $\phi^{E}(x)\in \End E_{x}$ for all 
$x \in M^{\phi}$. 

In what follows, we shall say that a function $f(t)$ is $\op{O}(t^{\infty})$ as $t\rightarrow 0^{+}$ when $f(t)$ is $\op{O}(t^{N})$ 
for all $N\in \N$. 

\begin{lemma}\label{lem:Heat-localization}
     As $t\rightarrow 0^{+}$, we have
\begin{equation*}
    \int_{M}\tr_{E}\left[\phi^{E}(x)K_{Q}(x,\phi(x),t)\right]|dx|=\int_{M^{\phi}}\tr_{E}\left[\phi^{E}(x)I_{Q}(x,t)\right]|dx| +\op{O}(t^{\infty}).
\end{equation*}
\end{lemma}
\begin{proof}
If we regard $K_{Q}(x,y,t)$ as a distributional section of $E\boxtimes E^{*}$ over $M\times M\times \R$, then Proposition~\ref{prop:Volterra-properties} tells us that 
$K_{Q}(x,y,t)$ is smooth on $\{(x,y,t)\in M\times M\times \R; \ x\neq y\}$. Incidentally,
$K_{Q}\left(x,\phi(x),t\right)$ is smooth on $(M\setminus M^{\phi})\times \R$. Let $N\in \N$.  Since $K_{Q}(x,y,t)=0$ for $t<0$, 
we see that $\partial^{N}_{t}K_{Q}(x,\phi(x),0)=0$ for all $x\in M\setminus  M^{\phi}$. The Taylor formula at $t=0$ then 
implies that, uniformly on compact subsets of $M\setminus M^{\phi}$, we have
\begin{equation*}
 K_{Q}\left(x,\phi(x),t\right)=\op{O}(t^{N})  \qquad \text{as $t\rightarrow 0^{+}$}.  
\end{equation*}
As $M$ is compact and $V_{\epsilon}$ is an open neighborhood of $M^{\phi}$, the complement $M\setminus V_{\epsilon}$ is a 
compact subset of $M\setminus M^{\phi}$. Thus,
  \begin{multline*}
   \int_{M}\tr_{E}\left[\phi^{E}(x)K_{Q}(x,\phi(x),t)\right]  |dx|     = 
   \int_{V_{\epsilon}}\tr_{E}\left[\phi^{E}(x)K_{Q}(x,\phi(x),t)\right] |dx|  +\op{O}(t^{N})  \\
       = \int_{M^{\phi}}\left( \int_{\cN^{\phi}_{x}(\epsilon)} 
      \tr_{E}\left[\phi^{E}(\exp_{x}(v))K_{Q}\left(\exp_{x}(v),\phi(\exp_{x}(v)),t\right)\right] |dv|\right)|dx| + \op{O}(t^{N})\\
        = \int_{M^{\phi}}\tr_{E}\left[\phi^{E}(x)I_{Q}(x,t)\right] |dx| +\op{O}(t^{N}).
 \end{multline*}This proves the lemma. 
\end{proof}
 
Thanks to Lemma~\ref{lem:Heat-localization} we are reduced to study the short-time behavior of $I_{Q}(x,t)$. Note this is a purely local 
issue and $I_{Q}(x,t)$ depends on $\epsilon$ only up to $\op{O}(t^{\infty})$ near $t=0$. Therefore, upon choosing 
$\epsilon_{0}$ small enough so that there is a local trivialization of $E$ over the tubular neighborhood 
$V_{\epsilon_{0}}$, we may  assume that $E$ is a trivial vector bundle. 

Given a fixed-point $x_{0}$ in a submanifold 
component $M_a^{\phi}$, consider some local coordinates $x=(x^{1},\ldots, x^{a})$ around 
$x_{0}$. Setting $b=n-a$, we may further assume that 
over the range of the domain of the local coordinates there is an orthonormal frame $e_{1}(x),\ldots, e_{b}(x)$ of 
$\cN^{\phi}$. This defines fiber coordinates $v=(v^{1},\ldots, v^{b})$. Composing with the map $\cN^{\phi}(\epsilon_{0})\ni 
(x,v)\rightarrow \exp_{x}v$ we then get local coordinates $x^{1},\ldots, x^{a},v^{1},\ldots ,v^{b}$ for $M$ near 
the fixed-point $x_{0}$. We shall refer to this type of coordinates as \emph{tubular coordinates}. 
Let $q(x,v;\xi,\nu;\tau)\sim \sum_{j\geq 0}q_{m-j}(x,v;\xi,\nu;\tau)$ be the symbol $Q$ in these tubular coordinates.  
The asymptotic expansion $\sim$ here follows from Definition~\ref{def.volterra.asymptotic.expansion}.
We denote by $K_{Q}(x,v;y,w;t)$ the kernel of $Q$ in these coordinates. In the local coordinates 
$x^{1},\ldots, x^{a}$ we have 
\begin{equation}\label{IQIntKer}
    I_{Q}(x,t)=\int_{|v|<\epsilon}\phi^{E}(x,0)^{-1}\phi^{E}(x,v)K_{Q}(x,v;x,\phi'(x)v;t)dv,
\end{equation}where $\phi^{E}(x,v)$ is $\phi^{E}$ in the tubular coordinates $(x,v)$.  

In what follows, we let $U$ be the open subset of $\R^a$ over which the coordinates $x=(x^{1},\ldots, x^{a})$ range. Moreover, we denote by $B(\epsilon_{0})$ (resp., 
$B(\epsilon)$) the open ball about the origin in $\R^{b}$ with  radius $\epsilon_{0}$ (resp., $\epsilon$). Note that  the range of $v=(v^{1},\ldots 
, v^{b})$ is $B(\epsilon_{0})$.  In addition, for $j=0,1,\ldots$ we set 
\begin{equation}
    q^{E}_{m-j}(x,v;\xi,\nu;\tau):=\phi^{E}(x,0)^{-1}\phi^{E}(x,v)q_{m-j}(x,v;\xi,\nu;\tau). 
    \label{eq:Equivariant.twisted-symbol}
\end{equation}

\begin{lemma}\label{lem:Heat.asymptotic-IQ-symbols}
    As $t\rightarrow 0^{+}$ and uniformly on compact subsets of $U$, we have
    \begin{equation*}
        I_{Q}(x,t)\sim \sum_{j\geq 0} \int_{|v|<\epsilon}(q^{E}_{m-j})^{\vee}\left(x,v;0,(1-\phi'(x))v;t\right)dv.
    \end{equation*}
\end{lemma}
\begin{proof}
Let $N \in \N_{0}$. By Proposition~\ref{prop:Heat.asymptotic-kernels} there is $J\in \N$ such that $K_{Q}-\sum_{j\leq 
J}\check{q}_{m-j}$ is $C^{N}$. Set 
\begin{equation*}
    R_{N}(x,v,t):=K_{Q}\left(x,v;x,\phi'(x)v;t\right)-\sum_{j\leq J}\check{q}_{m-j}\left(x,v;0,(1-\phi'(x))v;t\right).
\end{equation*}Then $R_{N}(x,v,t)$ is $C^{N}$  on $U\times B(\epsilon_{0})\times \R$. 
Moreover $R_{N}(x,v,t)=0$ for $t<0$, since $K_{Q}(x,v;y,w;t)$ and all the $\check{q}_{m-j}(x,v;y,w;t)$ vanish for 
$t<0$. This implies that $\partial^{j}R_{N}(x,v,0)=0$ for all $j\leq N$. Applying Taylor's formula at $t=0$ to 
$R_{N}(x,v,t)$ then shows that, as $t\rightarrow 0^{+}$ and uniformly on compact subsets of $U\times B(\epsilon_{0})$, 
the function 
$R_{N}(x,v,t)$ is $\op{O}(t^{N})$, that is, 
\begin{equation*}
 K_{Q}(x,v;x,\phi'(x)v;t)=\sum_{j\leq J}\check{q}_{m-j}\left(x,v;0,(1-\phi'(x))v;t\right)  + \op{O}(t^{N}).
\end{equation*}Therefore, uniformly on compact subsets of $U$, 
\begin{equation*}
    I_{Q}(x,t)= \sum_{j\leq J}  
  \int_{|v|<\epsilon}(q^{E}_{m-j})^{\vee}\left(x,v;0,(1-\phi'(x))v;t\right)dv +\op{O}(t^{N}).
\end{equation*}This proves the lemma. 
\end{proof}

\begin{lemma}\label{lem:Heat.asymptotic-Isymbol}
    As $t\rightarrow 0^{+}$ and uniformly on compact subsets of $U$, we have 
\begin{multline}
    \int_{|v|<\epsilon}(q^{E}_{m-j})^{\vee}\left(x,v;0,(1-\phi'(x))v;t\right)dv\\ \sim  \!\!\! 
  \sum_{\substack{|\alpha|+j+m+n \\ \textup{even}}}  \! t^{\frac{j-(m+a+2)+|\alpha|}{2}}
     \int_{\R^{b}} \frac{v^{\alpha}}{\alpha!}\left(\partial_{v}^{\alpha}q_{m-j}^{E}\right)^{\vee}\left(x,0;0,(1-\phi'(x))v;1\right)dv.
     \label{eq:Heat.int-checkqj-epsilon}
\end{multline}    
\end{lemma}
\begin{proof}
  Let $h(x,v,w,t)$ be the function on  $U\times B(\epsilon_{0})\times \left[(\R^{b}\times \R)\setminus 
    0\right]$ defined by 
    \begin{equation*}
        h(x,v,w,t):=(q^{E}_{m-j})^{\vee}\left(x,v;0,(1-\phi'(x))w;t\right)dv.
    \end{equation*}We observe that $h(x,v,w,t)$ is smooth on $U\times B(\epsilon_{0})\times (\R^{b}\setminus 0)\times 
    \R$ and vanishes for $t<0$. Moreover, the homogeneity of $\check{q}_{m-j}$ in the sense 
    of~(\ref{eq:homogeneous.Voterra.symbol}) implies that 
\begin{equation}
    h(x,v,\lambda w,\lambda^{2}t)=|\lambda|^{-(n+2)}\lambda^{j-m}h(x,v,w,t) \qquad \forall \lambda \in \R^{*}.
    \label{eq:Heat.homogeneity-h}
\end{equation}
   Setting $k=j-(m+n+2)$, we see that, for all $t>0$, we have
   \begin{equation}
     \int_{|v|<\epsilon}h(x,v,v,t)dv    = t^{\frac{b}{2}}\int_{B\left(\frac{\epsilon}{\sqrt{t}}\right)} 
     h(x,\sqrt{t}v,\sqrt{t}v,t)dv 
        = t^{\frac{k+b}{2}} \int_{B\left(\frac{\epsilon}{\sqrt{t}}\right)}h(x,\sqrt{t}v,v,1)dv. 
      \label{eq:Heat.int-h-rescaling}
   \end{equation}

    Let $N\in \N$. By Taylor's formula,
    \begin{equation}
       h(x,\sqrt{t}v,v,1)=\sum_{|\alpha|<N}\frac{t^{\frac{|\alpha|}{2}}}{\alpha!}\frac{v^{\alpha}}{\alpha !}\partial_{v}^{\alpha}h(x,0,v,1)+t^{\frac{N}{2}}R_{N}(x,\sqrt{t}v,v), 
       \label{eq:Heat.Taylor-expansion-h}
    \end{equation}where $R_{N}(x,v,w)$ is the function on $U\times B(\epsilon_{0})\times \R^{b}$ given by
    \begin{equation*}
        R_{N}(x,v,w)=\sum_{|\alpha|=N}\int_{0}^{1}(1-s)^{N-1}w^{\alpha}\partial_{v}^{\alpha}h(x,sv,w,1)ds.
    \end{equation*} 
   Let  $K$ be a compact subset of $U$. As $w^{\alpha}\partial_{v}^{\alpha}h(x,v,w,t)$ is smooth on $U\times B(\epsilon_{0})\times (\R^{b}\setminus 0)\times 
    \R$ and vanishes for $t<0$,  we see that $w^{\alpha}\partial_{t}^{l}\partial_{v}^{\alpha}h(x,v,w,0)=0$ for all $l\in 
    \N_{0}$. Therefore, using once more Taylor's formula at $t=0$ shows that, for all $l\in \N_{0}$, there is a constant $C_{K\alpha l}>0$ such that
    \begin{equation*}
        |w^{\alpha}\partial_{v}^{\alpha}h(x,v,w,t)|\leq C_{kl\alpha}|t|^{l} \qquad \forall 
        (x,v,w,t)\in K\times B(\epsilon)\times S^{b-1}\times (0,1).
    \end{equation*}
    In addition, the homogeneity of $h(x,v,w,t)$ implies that, when $w\neq 0$, we have
    \begin{equation*}
        w^{\alpha}\partial_{v}h(x,v,w,1)=w^{\alpha}|w|^{k}\partial_{v}h(x,v,|w|^{-1}w,|w|^{-2}).
    \end{equation*}Thus,
    \begin{equation*}
        \left |w^{\alpha}\partial_{v}^{\alpha}h(x,v,w,1)\right|\leq C_{kl\alpha } |w|^{k+|\alpha|-2l} \quad 
        \forall (x,v,w)\in K\times B(\epsilon) \times (\R^{b}\setminus 0).
    \end{equation*}
    
    The above estimate shows that $w^{\alpha}\partial_{v}h(x,v,w,1)$ has rapid decay in $w$ uniformly with respect to $x$ and $v$, as 
    $x$ ranges over $K$ and $v$ ranges over $B(\epsilon)$. Incidentally, both $w^{\alpha}\partial_{v}h(x,v,w,1)$ and 
    $R_{N}(x,v,w)$ are uniformly bounded on $K\times B(\epsilon)\times \R^{b}$. It then follows that there is a 
    constant $C_{KN}>0$ such that 
    \begin{equation*}
        |R_{N}(x,\sqrt{t}v,v)|\leq C_{KN} \qquad \forall (x,v)\in K\times B(\epsilon). 
    \end{equation*}
    Therefore, integrating both sides of~(\ref{eq:Heat.Taylor-expansion-h}) with respect to $v$ over $B\left(\frac{\epsilon}{\sqrt{t}}\right)$ 
    we see that, as $t\rightarrow 0^{+}$ and uniformly on $K$, we have
    \begin{equation*}
        \int_{B\left(\frac{\epsilon}{\sqrt{t}}\right)}h(x,\sqrt{t}v,v,1)dv 
       = \sum_{|\alpha|<N}t^{\frac{|\alpha|}{2}}\int_{B\left(\frac{\epsilon}{\sqrt{t}}\right)} 
        \frac{v^{\alpha}}{\alpha !}\partial_{v}^{\alpha}h(x,0,v,1)dv + \op{O}\biggl(t^{\frac{N-b}{2}}\biggl).
    \end{equation*}
    Together with~(\ref{eq:Heat.int-h-rescaling}) this proves that, as $t\rightarrow 0^{+}$ and uniformly on $K$, we have
    \begin{equation}
        \int_{|v|<\epsilon}h(x,v,v,t)dv \sim\sum t^{\frac{k+b+|\alpha|}{2}}\int_{B\left(\frac{\epsilon}{\sqrt{t}}\right)} 
        \frac{v^{\alpha}}{\alpha !}\partial_{v}^{\alpha}h(x,0,v,1)dv. 
        \label{eq:Heat.int-h-asymptotic-epsilon}
    \end{equation}
    We observe that $k+b= j-(m+a+2)$. Moreover, as mentioned above,  the function $w^{\alpha}\partial_{v}^{\alpha}h(x,0,w,1)$ has rapid decay uniformly with respect to $x$, as 
    $x$ ranges over $K$. Therefore, as $t\rightarrow 0^{+}$ and uniformly on $K$, we have
    \begin{equation}
        \int_{B\left(\frac{\epsilon}{\sqrt{t}}\right)} 
        v^{\alpha}\partial_{v}^{\alpha}h(x,0,v,1)dv = \int_{\R^{b}}v^{\alpha}\partial_{v}^{\alpha}h(x,0,v,1)dv+\op{O}(t^{\infty}).
        \label{eq:Heat.asymptotic-partial-h-epsilon}
    \end{equation}
    In addition, the homogeneity property~(\ref{eq:Heat.homogeneity-h}) for $\lambda=-1$ gives
    \begin{equation*}
       \int_{\R^{b}}v^{\alpha}\partial_{v}^{\alpha}h(x,0,v,1)dv  = 
       \int_{\R^{b}}(-v)^{\alpha}\partial_{v}^{\alpha}h(x,0,-v,(-1)^{2}1)dv 
       =(-1)^{|\alpha|+j-m}\int_{\R^{b}}v^{\alpha}\partial_{v}^{\alpha}h(x,0,v,1)dv.
    \end{equation*}Thus $\int_{\R^{b}}v^{\alpha}\partial_{v}^{\alpha}h(x,0,v,1)dv=0$ whenever 
    $|\alpha|+j-m$ is odd.  
    Combining this with~(\ref{eq:Heat.int-h-asymptotic-epsilon}) and (\ref{eq:Heat.asymptotic-partial-h-epsilon}) shows that, as $t\rightarrow 0^{+}$ and uniformly on $K$,  we have
     \begin{equation*}
        \int_{|v|<\epsilon}h(x,v,v,t)dv \sim  \!\!\! \!\!\! \sum_{\substack{|\alpha|+j-m \\ \text{even}}}  \! t^{\frac{j-(m+a+2)+|\alpha|}{2}} \int_{\R^{b}} 
        \frac{v^{\alpha}}{\alpha !}\partial_{v}^{\alpha}h(x,0,v,1)dv. 
    \end{equation*}This proves the lemma. 
\end{proof} 

Combining Lemmas~\ref{lem:Heat.asymptotic-IQ-symbols} and~\ref{lem:Heat.asymptotic-Isymbol} we see that, as $t\rightarrow 0^{+}$ and uniformly on compact subsets of $U$, we have
\begin{equation}
  I_{Q}(x,t) \sim \!\!\! \!\!\!  \sum_{\substack{|\alpha|+j-m \\ \text{even}}}  \! t^{\frac{j-(m+a+2)+|\alpha|}{2}} 
  \int_{\R^{b}} \frac{v^{\alpha}}{\alpha!}\left(\partial_{v}^{\alpha}q^{E}_{m-j}\right)^{\vee}\left(x,0;0,(1-\phi'(x))v;1\right)dv.  
  \label{eq:IQ}
\end{equation}
If  $|\alpha|+j-m$ is 
even, then $\frac{-(m+a+2)+j+|\alpha|}{2}$ is an integer and is greater than   $-\frac{m+a}{2}-1$, i.e.,  it is greater than 
or equal to $-\left[\frac{m+a}{2}\right]-1$. We actually have an equality when $j=|\alpha|=0$ and $m$ is even and 
when $|\alpha|+j=1$ and $m$ is odd. Thus, grouping together all the terms with same powers of $t$, we can rewrite the above 
asymptotic in the form,
\begin{equation}
\label{eq:I_Q.asymptoric.expansion.I_Q^j}
  I_{Q}(x,t)\sim   \sum_{j \geq 0}t^{-\left(\frac{a}{2}+\left[\frac{m}{2}\right]+1\right)+j}I_{Q}^{(j)}(x), 
\end{equation}where we have set
\begin{equation}
    I_{Q}^{(j)}(x):=\sum_{|\alpha|\leq m-2\left[\frac{m}{2}\right]+2j} \int 
    \frac{v^{\alpha}}{\alpha!}\left(\partial_{v}^{\alpha}q^{E}_{2\left[\frac{m}{2}\right]-2j+|\alpha|}\right)^{\vee}\left(x,0;0,(1-\phi'(x))v;1\right)dv.
    \label{eq:Heat.IQj}
\end{equation}Therefore, we arrive at the following result. 

\begin{proposition}\label{prop:Heat.asymptotic-IQ}
 Let $Q\in \pvdo^{m}(M\times \R,E)$, $m \in \Z$.  Uniformly on each fixed-point submanifold $M_a^{\phi}$, $a=0,2,...,n$, we have
    \begin{equation}
        I_{Q}(x,t)\sim  \sum_{j \geq 0} t^{-\left(\frac{a}{2}+\left[\frac{m}{2}\right]+1\right)+j}I_{Q}^{(j)}(x) \qquad 
        \text{as $t\rightarrow 0^{+}$}, 
        \label{eq:Heat.asymptotic-IQ}    \end{equation}where $I_{Q}^{(j)}(x)$ is the section of $\End E$ over $M^{\phi}$ defined 
    by~(\ref{eq:Heat.IQj}) in terms of the symbol of $Q$ in local tubular coordinates over which $E$ is trivialized.
\end{proposition}

\begin{remark}
 On $M_a^{\phi}$ the leading term in~(\ref{eq:Heat.asymptotic-IQ}) is 
 $t^{-\left(\frac{m+a}{2}+1\right)}I_{Q}^{(0)}(x)$, where $I_{Q}^{(0)}(x)$ depends only on the 
 principal symbol of $Q$. Namely,
 \begin{equation*}
     I_{Q}^{(0)}(x)=\int_{\R^{b}} (q_{m}^{E})^{\vee}(x,0;0,(1-\phi'(x))v;1)dv=|1-\phi'(x)|^{-1}\int_{\R^{b}} 
     \check{q}_{m}(x,0;0,v;1)dv. 
 \end{equation*}
\end{remark}

\begin{remark}
  The asymptotic~(\ref{eq:Heat.asymptotic-IQ}) is expressed in terms of the symbol of $Q$ in tubular coordinates. 
  However, we usually start with a symbol in some local coordinates before passing to tubular coordinates. We determine
  the symbol in the tubular coordinates by using the change of variable formula for symbols (as given, e.g., in~\cite{Ho:ALPDO3}) in the case of the change of variables  
  $\psi(x,v)=\exp_x(v)$. 
\end{remark}

We are now in a position to state and prove the main result of this section. 

 \begin{proposition}\label{TraceOfHeatKernelVB}
  Let $P:C^{\infty}(M,E)\rightarrow C^{\infty}(M,E)$ be a differential operator of order $m$. 
 \begin{enumerate}
     \item  Uniformly on each fixed-point submanifold $M^{\phi}_{a}$, $a=0,2,...,n$, we have 
    \begin{equation}
        I_{P(L+\partial_{t})^{-1}}(x,t)\sim  \sum_{j \geq 0} 
        t^{-\left(\frac{a}{2}+\left[\frac{m}{2}\right]\right)+j}I_{P(L+\partial_{t})^{-1}}^{(j)}(x) \qquad 
        \text{as $t\rightarrow 0^{+}$},
        \label{eq:Heat.asymptotic-IQE}
    \end{equation}where $I_{P(L+\partial_{t})^{-1}}^{(j)}(x)$ is the section of $\End E$ over $M^{\phi}$ defined 
    by~(\ref{eq:Heat.IQj}) in terms of the symbol of $P(L +\partial_{t})^{-1}$ in any tubular coordinates over 
    which $E$ is trivial. \smallskip

  \item  As $t\rightarrow 0^{+}$, we have
     \begin{align}
        \Tr \left [ 
    Pe^{-tL}U_{\phi}\right] & = \int_{M^{\phi}}\tr_{E}\left[ \phi^{E}(x)I_{P(L+\partial_{t})^{-1}}(x,t)\right]|dx|+\op{O}(t^{\infty})\\
     &\sim  \sum_{0\leq a\leq n} \sum_{j \geq 0} t^{-\left(\frac{a}{2}+\left[\frac{m}{2}\right]\right)+j} 
    \int_{M_a^{\phi}}\tr_{E}\left[ \phi^{E}(x)I_{P(L+\partial_{t})^{-1}}^{(j)}(x)\right] |dx|. 
    \label{eq:Trace.P.equivariant.heat.kernel.asymptotic.expansion}
     \end{align}
 \end{enumerate}
 \end{proposition}
\begin{proof}
 The first part is an immediate consequence of 
  Proposition~\ref{prop:Heat.asymptotic-IQ}, since $P(L+\partial_{t})^{-1}$ is a Volterra \psido\ of 
  order $m-2$.  Combining Proposition~\ref{thm:Heat.P-inverse-heat-operator-PsiDO} and~(\ref{eq:Heat.equivariant-trace-formula}) shows that 
  \begin{equation*}
       \Tr \left [ Pe^{-tL}U_{\phi}\right]= \int_{M} \tr_{E} \left[ \phi^{E}(x) 
       K_{P(L+\partial_{t})^{-1}}(x,\phi(x),t)\right].
  \end{equation*}The 2nd part then follows from Lemma~\ref{lem:Heat-localization} and the first part. The proof is complete. 
\end{proof}

\begin{remark}
    When $P=1$, the asymptotic~(\ref{eq:Trace.P.equivariant.heat.kernel.asymptotic.expansion}) was originally established by Gilkey~\cite{Gi:LNPAM}
    and Shih-Chi Lee~\cite{SCLee:PhD}. 
\end{remark}

\begin{remark}\label{rem:Heat.computation-Ij-inverse-heat-operator}
   The formula (\ref{eq:Heat.IQj}) expresses the coefficients  $I_{P (L+\partial_{t})^{-1}}^{(j)}(x)$ in terms of 
   the homogeneous components of the symbol of $P(L+\partial_{t})^{-1}$. Therefore, in order to compute them at a given point 
   $x_{0}\in M^{\phi}$, we may replace $P(L+\partial_{t})^{-1}$ by $PQ$, where $Q$ is Volterra \psido\ parametrix 
   of $L+\partial_{t}$ defined near $x_{0}$. In particular, we have
   \begin{equation*}
       I_{P(L+\partial_{t})^{-1}}(x_{0},t)=I_{PQ}(x_{0},t)+\op{O}(t^{\infty}). 
   \end{equation*}
\end{remark}

\section{The Local Equivariant Index Theorem}
\label{sec:proof-key-thm}
In this section, we shall use the Volterra pseudodifferential calculus and the equivariant heat kernel asymptotics from 
the previous section to give a new proof of the local equivariant index theorem  for Dirac operators.   

Let $(M^{n}, g)$ be an even dimensional compact spin oriented Riemannian manifold with the spinor bundle $\sS=\sS^{+}\oplus 
\sS^{-}$. We let $G$ be a subgroup of the connected component of the group  of smooth orientation-preserving isometries of $(M,g)$ preserving the spin 
structure. Then any $\phi\in G$ uniquely lifts to a unitary vector bundle isomorphism $\phi^{\sS}: \sS\rightarrow 
\phi_{*}\sS$ (see~\cite{BG:SODVM}). In addition, we let $E$ be a $G$-equivariant Hermitian vector bundle over $M$ of rank $p$ 
equipped with a $G$-equivariant Hermitian connection $\nabla^{E}$.
We also form the  $G$-equivariant 
$\Z_{2}$-graded Hermitian vector bundle $\sS\otimes E = ( \sS^{+}\otimes E) \oplus ( \sS^{-}\otimes E)$, 
and let $\sD_{\nabla^{E}}:C^{\infty}(M,\sS\otimes E)\rightarrow C^{\infty}(M,\sS\otimes E)$ be the Dirac operator associated to these data. Namely,
\begin{equation*}
    \sD_{\nabla^{E}}=\sD_{g}\otimes 1_{E}+(c\otimes 1_{E})\circ \nabla^{E},
\end{equation*}where $\sD_{g}:C^{\infty}(M,\sS)\rightarrow C^{\infty}(M,\sS)$ is the Dirac operator of $M$ and 
$c:T^{*}M\times \sS\rightarrow \sS$ is the Clifford action of $T^{*}M$ on $\sS$. Note that with respect to the splitting 
$\sS\otimes E = ( \sS^{+}\otimes E) \oplus ( \sS^{-}\otimes E)$ the Dirac operator $\sD_{\nabla^{E}}$ takes the form,
\begin{equation*}
    \sD_{\nabla^{E}}=
    \begin{pmatrix}
        0 & \sD_{\nabla^{E}}^{-} \\
        \sD_{\nabla^{E}}^{+} & 0
    \end{pmatrix}, \qquad \sD_{\nabla^{E}}^{\pm}:C^{\infty}(M,\sS^{\pm}\otimes E)\rightarrow C^{\infty}(M,\sS^{\mp}\otimes E).
\end{equation*}

Given $\phi\in G$, as in the previous section, we let $\phi^{E}$ be the unitary vector bundle isomorphism from $E$ onto $\phi_*E$ induced by the action of $\phi$ on $E$. 
Then $\phi^{\sS}\otimes \phi^{E}$ is a lift of $\phi$ to a unitary vector bundle isomorphism from $\sS\otimes E$ onto $\phi_{*}\sS\otimes 
\phi_{*}E=\phi_{*}(\sS\otimes E)$.   As in the previous section, this defines a unitary operator 
$U_{\phi}:L^{2}(M,\sS\otimes E)\rightarrow L^{2}(M,\sS\otimes E)$. Note that the $G$-equivariant setup implies that $U_{\phi}$ 
commutes with the Dirac operator $\sD_{\nabla^{E}}$. The equivariant index $\sD_{\nabla^{E}}:G\rightarrow \Z$ is defined by
\begin{equation*}
    \ind \sD_{\nabla^{E}}(\phi)= \Tr \left[U_{\phi|\ker \sD^{+}_{\nabla^{E}}}\right]- \Tr\left[U_{\phi|\ker \sD^{-}_{\nabla^{E}}}\right] \qquad 
    \forall \phi\in G. 
\end{equation*}When $\phi=\op{id}_{M}$ we recover the Fredholm index of $\sD_{\nabla^{E}}$. When $\phi$ has only isolated 
fixed-points, the equivariant index  of $\phi$ agrees with the Lefschetz number associated to the spin complex with coefficient in 
$E$ (see~\cite{AB:LFT1}). The equivariant index theorem of Atiyah-Segal-Singer~\cite{AS:IEO2, AS:IEO3} expresses the equivariant index in terms 
of universal curvature polynomials defined as follows. 

Let $\phi \in G$.  As in the previous section, we shall denote by $M^{\phi}$ the fixed-point set of $\phi$ and 
by $\cN^{\phi}$ the normal bundle of $M^{\phi}$. We observe that, as $\phi$ preserves the orientation, the fixed-point submanifolds $M_a^\phi$ have even dimensions. Moreover, we shall orient each fixed-point submanifold $M^{\phi}_{a}$, 
$a=0,2,\ldots,n$, as in~\cite[Prop.~6.14]{BGV:HKDO}, so that $\phi^{\sS}$ gives rise to a section of 
$\Lambda^{n-a}(\cN^{\phi})^{*}_{\left| M_{a}^{\phi}\right.}$ which is positive with respect to the orientation of 
$\cN^{\phi}$ defined by the orientations of $M$ and $M^{\phi}_{a}$. 

Let $R^{TM}$ be the curvature of $(M,g)$, regarded as a section of $\Lambda^{2}T^{*}M \otimes \End (TM)$. 
The $G$-equivariance of Levi-Civita connection $\nabla^{TM}$ implies that, its restriction to each fixed-point 
submanifold $M_{a}^{\phi}$, $a=0,2,\ldots,n$,  preserves the splitting $TM_{|M^{\phi}_{a}}=TM^{\phi}_{a}\oplus 
\cN^{\phi}$ over $M^{\phi}_{a}$. Therefore, it induces connections $\nabla^{TM^{\phi}}$ and 
$\nabla^{\cN^{\phi}}$ on $TM^{\phi}_{a}$ and $\cN^{\phi}$, respectively, so that we have
\begin{equation*}
    \nabla^{TM}_{|TM^{\phi}_{a}}=\nabla^{TM^{\phi}}\oplus \nabla^{\cN^{\phi}} \qquad \textup{on $M^{\phi}_{a}$}.
\end{equation*}Note that $\nabla^{TM^{\phi}}$ is the Levi-Civita connection of $TM^{\phi}_{a}$. Let $R^{TM^{\phi}}$ and $R^{\cN^{\phi}}$ 
be the respective curvatures of $\nabla^{TM^{\phi}}$ and 
$\nabla^{\cN^{\phi}}$.  We observe that
\begin{equation}
    R_{|\Lambda^{2}TM_{a}^{\phi}}=R^{TM^{\phi}}\oplus R^{\cN^{\phi}}. 
    \label{eq:LEIT.splitting-curvature}
\end{equation}
Define
\begin{equation}
    \hat{A}(R^{TM^{\phi}})= {\det}^{\frac{1}{2}} \left(\frac{R^{TM^{\phi}}/2}{\sinh(R^{TM^{\phi}}/2)}\right) 
    \quad \text{and} \quad  
    \nu_{\phi}\left(R^{\cN^{\phi}}\right)={\det}^{-\frac{1}{2}}\left(1- \phi^{\cN}e^{-R^{\cN^{\phi}}}\right),
    \label{eq:Conformal.characteristic-forms}
\end{equation}where ${\det}^{-\frac{1}{2}}\left(1- \phi^{\cN}e^{-R^{\cN^{\phi}}}\right)$ is defined in the same way as 
in~\cite[Section 6.3]{BGV:HKDO}. 

In addition, let $F^{E}$ be the curvature of the $G$-equivariant connection $\nabla^{E}$ and denote by $F^{E}_0 $ its 
restriction to $M_{a}^{\phi}$. Note that $F^{E}_0$ is a smooth section of 
$\left(\Lambda^{2}T^{*}M^{\phi}_{a}\right)\otimes \End(E)$. We then  define
\begin{equation*}
   \Ch_{\phi}\left(F^{E}\right):=\Tr \left[ \phi^{E}\exp\left(-F^{E}_0\right)\right]\in C^\infty\left(M_a^\phi,\Lambda^{2\bt}T^{*}M^{\phi}_{a}\right). 
\end{equation*}

We can now state the equivariant index theorem in the following form. 

\begin{theorem}[Equivariant Index Theorem~\cite{AS:IEO2, AS:IEO3}]\label{thm:EIT} For all $\phi\in G$, we have
\begin{equation}
    \ind \sD_{\nabla^E}(\phi)= (-i)^{\frac{n}{2}}\sum_{\substack{0\leq a \leq n\\ \textup{$a$ even}}} (2\pi)^{-\frac{a}{2}} 
            \int_{M^{\phi}_{a}}\hat{A}(R^{TM^{\phi}})\wedge \nu_{\phi}\left(R^{\cN^{\phi}}\right)\wedge 
            \Ch_{\phi}\left(F^{E}\right). 
            \label{eq:LEIT.EIT}
        \end{equation}
\end{theorem}

\begin{remark}
When $\phi=\op{id}_{M}$ we recover the index theorem of Atiyah-Singer~\cite{AS:IEO1, AS:IEO3} for Dirac operators. In case 
$\phi$ has only isolated fixed-points, the formula~(\ref{eq:LEIT.EIT}) reduces to the fixed-point formula of Atiyah-Bott~\cite{AB:LFT1, AB:LFT2}. 
\end{remark}

Let us recall how the equivariant index theorem can be proved by heat kernel techniques. In what follows we let
$\str_{\sS\otimes E}=\tr_{\sS^{+}\otimes E}-\tr_{\sS^{-}\otimes E}$ be the supertrace on the fibers of the $\Z_{2}$-graded 
bundle $\sS\otimes E=( \sS^{+}\otimes E) \oplus ( \sS^{-}\otimes E)$. We also denote by $\Str$ the corresponding 
supertrace  on $\cL^{1}\left( L^{2}(M,\sS\otimes E)\right)$. By the equivariant McKean-Singer 
formula~(see, e.g., \cite[Prop.~6.3]{BGV:HKDO}), for any given $\phi\in G$, we have
\begin{equation*}
     \ind \sD_{\nabla^{E}}(\phi)= \Str \left[e^{-t\sD_{\nabla^{E}}^{2}}U_{\phi}\right]\qquad \text{for all $t>0$}.
\end{equation*}
Therefore, the equivariant index theorem is a consequence of the following result. 

\begin{theorem}[Local Equivariant Index Theorem~\cite{DP:T, Gi:LNPAM, Ka:PhD}]\label{thm:LEIT.local-equiv.-index-thm}
    For all $\phi\in G$, we have
\begin{equation}
   \lim_{t\rightarrow 0^{+}} \Str \left[e^{-t\sD_{\nabla^{E}}^{2}}U_{\phi}\right] = (-i)^{\frac{n}{2}}
   \sum_{\substack{0\leq a\leq n\\ \textup{$a$ even}}} (2\pi)^{-\frac{a}{2}} 
            \int_{M^{\phi}_{a}}\hat{A}(R^{TM^{\phi}})\wedge \nu_{\phi}\left(R^{\cN^{\phi}}\right)\wedge 
            \Ch_{\phi}\left(F^{E}\right). 
            \label{eq:LEIT.local-equiv.-index-thm}
        \end{equation}
\end{theorem}

\begin{remark}
The local equivariant index theorem for the Dirac complex with coefficients in a vector bundle is originally due to 
Gilkey~\cite{Gi:LNPAM}, who also obtained versions of this theorem for several other elliptic complexes. Versions for the 
signature complex were also given by Donnelly-Patodi~\cite{DP:T} and Kawasaki~\cite{Ka:PhD} around the same time. All 
these proofs partly rely on Riemanian invariant theory.  
Purely analytical proofs were subsequently given by Bismut~\cite{Bi:ASITPA2}, 
Berline-Vergne~\cite{BV:BSMF, BGV:HKDO} and Lafferty-Yu-Zhang~\cite{LYZ:TAMS}.  In addition, Liu-Ma~\cite{LM:DMJ} proved a 
version of the local equivariant index theorem for families of Dirac operators.  
\end{remark}

In what follows, given a differential form $\omega$ on $M$ (resp., $M_a^\phi$) we shall denote by $|\omega|^{(n)}$ (resp., $|\omega|^{(a,0)}$) its Berezin integral, i.e., 
its inner product with the volume form of $M$ (resp., the induced volume form of $M_a^\phi$).  We note that if $\omega$ is a differential form on $M_a^\phi$, then 
\begin{equation}
 \int_{M_a^\phi} |\omega(x)|^{(a,0)}|dx|=\int_{M_a^\phi} \omega. 
 \label{eq:LEIT.Berezin-integration}
\end{equation}
Bearing this in mind, let $\phi \in G$. Applying the 2nd part of Proposition~\ref{TraceOfHeatKernelVB} to 
$L=\sD_{\nabla^{E}}^{2}$ and $P=\gamma$, where $\gamma=\op{id}_{\sS^{+}\otimes 
E}-\op{id}_{\sS^{-}\otimes E}$ is the grading operator, we see that, as $t\rightarrow 0^{+}$, we have
\begin{equation*}
     \Str \left [ e^{-t\sD_{\nabla^{E}}^{2}}U_{\phi}\right]  = \int_{M^{\phi}}\str_{\sS\otimes E}\left[ \phi^{\sS\otimes 
     E}(x)I_{(\sD_{\nabla^{E}}^{2}+\partial_{t})^{-1}}(x,t)\right]|dx|+\op{O}(t^{\infty}). 
\end{equation*}
Therefore, we see that the local equivariant index theorem is a consequence of the following result. 

\begin{theorem}\label{thm:LEIT.local-equiv.-index-thm-pointwise}
  Let $\phi \in G$. Then, as $t\rightarrow 0^{+}$ and uniformly on each fixed-point submanifold $M_a^\phi$, we have
   \begin{multline}
      \str_{\sS\otimes E} \left[ \phi^{\sS\otimes E}(x_{0})I_{(\sD_{\nabla^{E}}^{2}+\partial_{t})^{-1}}(x_{0},t)\right]=\\ (-i)^{\frac{n}{2}} (2\pi)^{-\frac{a}{2}}
      \biggl|\hat{A}(R^{TM^{\phi}})\wedge \nu_{\phi}\left(R^{\cN^{\phi}}\right)\wedge 
            \Ch_{\phi}\left(F^{E}\right)(x_{0})\biggr|^{(a,0)} + \op{O}(t).
      \label{eq:LEIT.local-equiv.-index-thm-pointwise}
  \end{multline} 
\end{theorem}

We shall now prove Theorem~\ref{thm:LEIT.local-equiv.-index-thm-pointwise}. By the 1st part Proposition~\ref{TraceOfHeatKernelVB}, as $t\rightarrow 0^{+}$ and uniformly on each fixed-point submanifold $M_{a}^{\phi}$, we have
\begin{equation}
\label{eq:asympt.I.parametrix}
  I_{(\sD_{\nabla^{E}}^{2}+\partial_{t})^{-1}}(x,t)\sim  \sum_{j \geq 0} 
        t^{-\frac{a}{2}+j}I_{(\sD_{\nabla^{E}}^{2}+\partial_{t})^{-1}}^{(j)}(x).   
\end{equation}
Comparing the asymptotics~(\ref{eq:LEIT.local-equiv.-index-thm-pointwise}) and~(\ref{eq:asympt.I.parametrix}) we deduce that in order to prove 
Theorem~\ref{thm:LEIT.local-equiv.-index-thm-pointwise} it is enough to show that the asymptotics~(\ref{eq:LEIT.local-equiv.-index-thm-pointwise}) hold pointwise at any fixed-point of $\phi$.
 
Let $x_{0}\in M_{a}^{\phi}$, $a=0,2,\ldots,n$. By Remark~\ref{rem:Heat.computation-Ij-inverse-heat-operator}, given any Volterra \psido\ parametrix $Q$ defined near $x_{0}$, we have
\begin{equation}
    I_{(\sD_{\nabla^{E}}^{2}+\partial_{t})^{-1}}(x_{0},t)=I_{Q}(x_{0},t) +\op{O}(t^{\infty}).
    \label{eq:LIT.IQ-approximation}
\end{equation}As a result we may replace the Dirac operator $\sD_{\nabla^{E}}$ by any differential operator that agrees with 
$\sD_{\nabla^{E}}$ in a given local chart near $x_{0}$. In other words, this enables us to localize the problem and replace 
$\sD_{\nabla^{E}}$ by a Dirac operator on $\R^{n}$ and acting on the trivial bundle with fiber $\sS_{n}\otimes \C^{p}$,  where 
$\sS_{n}$ is the spinor space of $\R^{n}$ and $p$ is the rank of $E$. 

To wit let $e_{1},\ldots,e_{n}$ be an oriented orthonormal basis of $T_{x_{0}}M$ such that 
$e_{1},\ldots, e_{a}$ span $T_{x_{0}}M^{\phi}$ and $e_{a+1},\ldots, e_{n}$ span $\cN^{\phi}_{x_{0}}$. This provides us 
with normal coordinates $(x^{1},\ldots, x^{n})\rightarrow \exp_{x_{0}}\left(x^{1}e_{1}+\cdots +x^{n}e_{n}\right)$.  Moreover, 
using parallel transport enables us to construct a synchronous local oriented tangent frame $e_{1}(x),\ldots,e_{n}(x)$ such that $e_{1}(x), 
\ldots, e_{a}(x)$ form an oriented frame of $TM_{a}^{\phi}$ and $e_{a+1}(x),\ldots, e_{n}(x)$ form an (oriented) frame $\cN^{\phi}$ 
(when both frames are restricted to $M^\phi$). This  gives rise to trivializations of the tangent and spinor bundles. 
We also trivialize $E$ near $x_{0}$. Using these coordinates and 
trivializations, we let $\sD$ be a Dirac operator on $\R^{n}$ acting on the 
trivial bundle with fiber $\sS_{n}\otimes \C^{p}$ associated to a metric on $\R^{n}$ and a connection on $\C^{p}$ that
 agree near $x=0$  with the metric $g$ and connection $\nabla^{E}$, respectively. Incidentally, 
$\sD_{\nabla^{E}}$ agrees with $\sD$ near $x=0$. Note that $e_{j}(x)=\partial_{j}$ at $x=0$. Moreover, the coefficients $g_{ij}(x)$ of the metric and the 
coefficients $\omega_{ikl}:=\acou{\nabla_i^{TM}e_k}{e_l}$ of the Levi-Civita connection satisfy 
\begin{equation}
     g_{ij}(x)=\delta_{ij}+\op{O}(|x|^{2}), \qquad  \omega_{ikl}(x)= -\frac12 
R_{ijkl}x^j 
     +\op{O}(|x|^{2}),   
    \label{eq:AS-asymptotic-geometric-data}
\end{equation}
where  $R_{ijkl}:=\acou{R^{TM}(0)(\partial_{i},\partial_{j})\partial_{k}}{\partial_{l}}$ are the coefficients of the 
curvature tensor at $x=0$  (see, e.g,~\cite[Chap.~1]{BGV:HKDO}).   
Moreover, in order to simplify notations we shall denote by $\phi'$ the endomorphism $\phi'(0)$ of $\R^{n}$. We shall use a similar 
notation for $\phi^{\cN}(0)$, $\phi^{\sS}(0)$, and $\phi^{E}(0)$. In particular, $\phi^{\cN}$ is the element of $\op{SO}(n-a)$ such that 
\begin{equation*}
    \phi'= 
    \begin{pmatrix}
        1 & 0 \\
        0 & \phi^{\cN}
    \end{pmatrix}.
\end{equation*}

Let $Q\in \pvdo^{-2}(\R^{n}\times \R, \sS_{n}\otimes \C^{p})$ be a parametrix for $\sD^{2}+\partial_{t}$. 
Using~(\ref{eq:LIT.IQ-approximation})  we obtain
\begin{equation}
\str_{\sS\otimes E}\left[ \phi^{\sS\otimes E}(x_{0})I_{(\sD_{\nabla^{E}}^2+\partial_{t})^{-1}} (x_{0},t)\right]  = 
(\str_{\sS_{n}\otimes \C^{p}})\left[ (\phi^{\sS}\otimes \phi^{E})I_{Q} 
(0,t)\right] +\op{O}(t^{\infty}).  
\label{eq:ActionOnBundle}
 \end{equation}

The supertrace of an endomorphism on spinors is computed as follows. Let $\Lambda(n)=\Lambda^{\bt}_{\C}\R^{n}$ be the complexified exterior algebra of $\R^{n}$.    
We shall use the following gradings on $\Lambda(n)$, 
\begin{equation*}
    \Lambda(n)=\bigoplus_{1\leq j \leq n}\Lambda^{j}(n)=\bigoplus_{\substack{1\leq k \leq a\\
    1\leq l\leq n-a}}\Lambda^{k,l}(n),
\end{equation*}where $\Lambda^{j}(n)$ is the space of forms of degree $j$ and $\Lambda^{k,l}(n)$ is the span of forms 
$dx^{i_{1}}\wedge \cdots \wedge dx^{i_{k+l}}$ with $1\leq i_{1}<\cdots <i_{k}\leq a$ and $a+1\leq i_{k+1}<\cdots 
<i_{k+l}\leq n$. Given a form $\omega \in \Lambda(n)$ we shall denote by $\omega^{(j)}$ (resp., $\omega^{(k,l)}$) its 
component in $\Lambda^{j}(n)$ (resp., $\Lambda^{k,l}(n)$). 
Let $\Cl(n)$ be the complexified Clifford algebra of $\R^{n}$ (seen as a subalgebra of $\End \Lambda(n)$) and 
let $c:\Lambda(n)\rightarrow \Cl(n)$ be the linear isomorphism given by the Clifford action of $\Lambda(n)$ on itself. Composing $c$ with the 
spinor representation $\Cl(n)\rightarrow \End \sS_{n}$ (which is an algebra isomorphism since $n$ is  even), we get a 
linear isomorphism $\Lambda(n)\rightarrow \End \sS_{n}$, which we shall also denote by $c$. We denote by $\sigma: \End \sS_{n}\rightarrow \Lambda(n)$ 
its inverse. We note that, although $c$ and $\sigma$ are not isomorphisms of algebras, for $\omega_{j}\in 
\Lambda^{k_{j},l_{j}}(n)$, $j=1,2$, we have
\begin{equation}
    \sigma\left[ c(\omega_{1})c(\omega_{2})\right]=\omega_{1} \wedge \omega_{2} \quad \bmod
  \bigoplus_{(k,l)\in \cK} \Lambda^{k,l}(n),
  \label{eq:SymbolQuantization}
\end{equation}where $\cK$ consists of all pairs $(k,l)$ such that, either $k\leq k_{1}+k_{2}-2$ and $l\leq l_{1}+l_{2}$, or 
$k\leq k_{1}+k_{2}$ and $l\leq l_{1}+l_{2}-2$.

\begin{lemma}[{\cite[Thm.~1.8]{Ge:POSASIT}}; see also~{\cite[Prop.~3.21]{BGV:HKDO}}]\label{lm:ActionOnSpinorBundle}
Let $\mathfrak{a}\in\End\sS_{n}$. Then 
    \begin{equation*}
    \str_{\sS_{n}}[\mathfrak{a}]=(-2i)^{\frac{n}{2}}\left| \sigma(\mathfrak{a})\right|^{(n)}. 
\end{equation*}
\end{lemma}

In the case of $\phi^{\sS}$ we have the following result. 
\begin{lemma}[\cite{BGV:HKDO, LYZ:TAMS}]\label{lem:LEIT.symbol-phiS}
    The form $\sigma\left[ \phi^{\sS}\right]$ belongs to $\Lambda^{(0,2\bt)}(n)$ and we have
    \begin{equation}
       \sigma\left[ \phi^{\sS}\right]^{(0,n-a)}=2^{-\frac{n-a}{2}} {\det}^{\frac{1}{2}}(1-\phi^{\cN})dx^{a+1}\wedge 
       \cdots \wedge dx^{n}.
       \label{eq:LEIT.top-symbol-phiS}
    \end{equation}
\end{lemma}
\begin{proof}
    As $\phi^{\cN}$ is an element of $\op{SO}(n-a)$, there is an oriented orthonormal basis $\{v_{a+1},\ldots, v_{n}\}$ of  $\{0\}^a\times \R^{n-a}$ such for 
$j=\frac{a}{2}+1,\ldots,\frac{n}{2}$ the subspace $\op{Span}\{v_{2j-1},v_{2j}\}$ is invariant under $\phi^{\cN}$ and 
the matrix of $\phi^{\cN}$ with respect to the basis $\{v_{2j-1},v_{2j}\}$ is a rotation matrix of the form, 
\begin{equation*}
    \begin{pmatrix}
        \cos \theta_{j} & -\sin \theta_{j} \\
        \sin \theta_{j} & \cos \theta_{j}
    \end{pmatrix}= \exp  
    \begin{pmatrix}
        0 & -\theta_{j} \\
       \theta_{j} & 0
    \end{pmatrix}, \qquad 0<\theta_{j}\leq \pi.
\end{equation*}
Using~\cite[Eqs.~(3.4)--(3.5)]{BGV:HKDO} we then see that 
\begin{equation}
    \phi^{\sS}= \prod_{\frac{a}{2}<j\leq \frac{n}{2}}\left( \cos\left(\frac{\theta_{j}}{2}\right) + 
    \sin\left(\frac{\theta_{j}}{2}\right)c(v^{2j-1})c(v^{2j})\right),
    \label{eq:LEIT.phiS}
\end{equation}where $\{v^{a+1},\ldots,v^{n}\}$ is the basis of $\Lambda^{0,1}(n)$ that is dual to $\{v_{a+1},\ldots, v_{n}\}$. 
It  then follows that $\sigma(\phi^{\sS})$ is an element of $\Lambda^{0,2\bt}(n)$ and we have
\begin{equation*}
    \sigma[\phi^{\sS}]^{(0,n-a)}=\prod_{\frac{a}{2}<j\leq \frac{n}{2}} \sin\left(\frac{\theta_{j}}{2}\right)v^{a+1}\wedge \cdots 
    \wedge v^{n}= 2^{-\frac{n-a}{2}} {\det}^{\frac{1}{2}}(1-\phi^{\cN})dx^{a+1}\wedge \cdots \wedge dx^{n},
\end{equation*}where we have used the equality $4\sin^{2}\frac{\theta}{2}=
\begin{vmatrix}
    1-\cos \theta & \sin \theta \\
    -\sin \theta & 1-\cos \theta
\end{vmatrix}$. The proof is complete. 
\end{proof}

We shall determine the short-time behavior of $\sigma \left[  \phi^{\sS} \tr_{\C^{p}}\left[\phi^{E}I_{Q}(0, 
t)\right]\right]$ in~(\ref{eq:ActionOnBundle}) by means of elementary considerations on Getzler orders of 
Volterra \psidos\ as in~\cite{Po:CMP}. This notion of order is intimately related to the rescaling of 
Getzler~\cite{Ge:SPLASIT}, which is motivated by the following assignment of degrees:
\begin{equation}
\deg \partial_{j} =\deg c(dx^j)=1, \qquad  \deg \partial_{t}=2, \qquad  \deg x^j=-1.
  \label{eq:AS-Getzler-order}
\end{equation}As observed in \cite{Po:CMP} this defines another  filtration on Volterra \psidos\ as follows.  

Let $Q\in  \Psi_{\op{v}}^{m'}(\R^{n}\times\R, \sS_n\otimes \C^{p})$ have symbol $q(x,\xi,\tau) \sim \sum_{r\geq m'} q_{m'-r}(x, \xi, \tau)$. 
Taking components in each subspace 
$\Lambda^j T_{\C}^{*}\R^{n}$ and using Taylor expansions at $x=0$ we get asymptotic expansions of symbols,  
\begin{equation} 
    \sigma[q(x, \xi, \tau)] \sim  \sum_{j,r} \sigma[q_{m'-r}(x, \xi, \tau)]^{(j)} \\
 \sim  \sum_{j,r,\alpha} 
    \frac{x^{\alpha}}{\alpha!}  \sigma[\partial_{x}^{\alpha}\partial_v^{\beta}q_{m'-r}(0,\xi, \tau)]^{(j)},
    \label{eq:AS.Getzler-asymptotic}
\end{equation}The last asymptotic is meant in the following sense:  for $j=0,\ldots, n$ and all $N\in \N$, as 
$x\rightarrow 0$ and $|\xi|+|\tau|^{\frac{1}{2}}\rightarrow \infty$,  we have
\begin{equation}
   \sigma[q(x, \xi, \tau)]^{(j)}-\sum_{r+|\alpha|=N+j}\frac{x^{\alpha}}{\alpha!}  
    \sigma[\partial_{x}^{\alpha}\partial_v^{\beta}q_{m'-r}(0,\xi, \tau)]^{(j)} =\op{O}\left( 
    \|\xi,\tau\|^{m'}\left(|x|+\|\xi,\tau\|^{-1}\right)^{N}\right),
    \label{eq:AS.Getzler-asymptotic-def}
\end{equation}where we have set $\|\xi,\tau\|:=|\xi|+|\tau|^{\frac{1}{2}}$, and there are similar asymptotics for all  
$\partial_{\xi}^{\beta}\partial_{\tau}^{k}$-derivatives (upon replacing the exponent $m'$ by $m'-|\beta|-k$). Moreover, the degree assignment~(\ref{eq:AS-Getzler-order}) leads us to define (Getzler-)rescaling operators $\delta_{\lambda}^{*}$, $\lambda \in \R$, on Volterra symbols 
 with coefficients in $\Lambda(n)\otimes M_{p}(\C)$  by letting
 \begin{equation*}
     \delta_{\lambda}^{*}q(x,\xi,\tau):=\lambda^{j}q(\lambda^{-1} x,\lambda\xi,\lambda^2 \tau) \qquad \forall 
     q\in \Sv^{*}(\R^{n}\times \R^{n}\times \R)\otimes \Lambda^{j}(n)\otimes M_{p}(\C).
 \end{equation*}
Note that in~(\ref{eq:AS.Getzler-asymptotic}) each symbol $ x^{\alpha}
\partial_{x}^{\alpha}\sigma[q_{m'-r}(0,\xi, \tau)]^{(j)}$ is homogeneous of degree $\mu:=m'-r+j-|\alpha|$ with respect to this 
rescaling. We shall say that such a symbol is  \emph{Getzler homogeneous} of degree $\mu$. 
The asymptotic expansion~(\ref{eq:AS.Getzler-asymptotic}) then implies that, in the sense of~(\ref{eq:AS.Getzler-asymptotic-def}), we have
\begin{equation*}
     \sigma[q(x, \xi, \tau)] \sim \sum_{\mu\leq m}q_{(\mu)}(x, \xi, \tau), 
\end{equation*}where $q_{(\mu)}(x, \xi, \tau)$ is the Getzler-homogeneous symbol of degree $\mu$ given by
\begin{equation}
    q_{(\mu)}(x, \xi, \tau):= \!\!\!\! \sum_{m'-r+j-|\alpha|=\mu}  \!\!\!\! \frac{x^{\alpha}}{\alpha!}  
    \sigma[\partial_{x}^{\alpha}\partial_v^{\beta}q_{m'-r}(0,\xi, \tau)]^{(j)},
     \label{eq:index.asymptotic-symbol}
\end{equation}and $m$ is the greatest integer $\mu$ such that $q_{(\mu)}\neq 0$. 

Alternatively, in terms of the rescaling operators $\delta_{\lambda}^{*}$, for all $(x,\xi,\tau)\in \R^{n}\times \R^{n}\times \R$, $(\xi,\tau)\neq 0$, we have 
\begin{equation*}
    \delta_{\lambda}^{*}\sigma[q(x, \xi, \tau)] \sim \sum_{\mu\leq m}\lambda^{\mu}q_{(\mu)}(x, \xi, \tau) \qquad 
    \text{as $\lambda \rightarrow 0$}.
\end{equation*}We observe that the homogeneous symbols $q_{(\mu)}(x,\xi,\tau)$ are uniquely determined by the above 
asymptotic. In particular, the leading Getzler-homogeneous symbol $q_{(m)}(x,\xi,\tau)$ is uniquely determined by 
\begin{equation}
      \delta_{\lambda}^{*}\sigma[q(x, \xi, \tau)]=\lambda^{m} q_{(m)}(x,\xi,\tau)+\op{O}(\lambda^{m-1}).
      \label{eq:Ge.leading-deltal}
\end{equation}

\begin{remark}
 Let $\cS(\R^{n}\times \R)$  be the Fr\'echet space of Schwartz-class functions on $\R^{n}\times \R$. 
As each symbol $q_{(\mu)}(x,\xi,\tau)$ is a polynomial with respect to the variable $x$, it 
defines an element $q_{(\mu)}(x,D_{x},D_{t})$ of $\cL\left( \cS(\R^{n}\times \R, \C^{p})\right)\otimes \Lambda(n)$ by
\begin{equation}
q_{(\mu)}(x,D_{x},D_{t})u(x,s)= \acou{\check{q} (x,x-y,s-t)}{u(y,t)} \qquad \forall u \in  \cS(\R^{n}\times \R, \C^{p}). 
\label{eq:Volterra.getzler-homogeneous-Volterra-PsiDO}
\end{equation}
Note that by using the action of $\Lambda(n)$ on itself by left-multiplication we may regard 
$q_{(\mu)}(x,D_{x},D_{t})$ as an actual operator of $\cS(\R^{n}\times \R, \Lambda(n)\otimes\C^{p})$ to itself. 
\end{remark}

\begin{definition}[\cite{Po:CMP}]\label{def:getzler.model-operator}
Bearing in mind the notation~(\ref{eq:index.asymptotic-symbol}) and~(\ref{eq:Volterra.getzler-homogeneous-Volterra-PsiDO}), we make the following definitions: 
\begin{enumerate}
\item The integer $m$ is called the Getzler order of $Q$,
  
\item The symbol $q_{(m)}$ is called the principal Getzler-homogeneous symbol  of $Q$,
  
\item The operator $Q_{(m)}=q_{(m)}(x,D_{x}, D_{t})$ is called the model operator of $Q$.
\end{enumerate}
\end{definition}

\begin{remark}
    As the symbol $ \sigma[\partial_{x}^{\alpha}\partial_v^{\beta}q_{m'-r}(0,\xi, \tau)]^{(j)}$ is 
    Getzler-homogeneous of degree $m'-r+j-|\alpha|\leq m'+n$, we see that the Getzler order of $Q$ is always~$\leq 
    m+n$. 
\end{remark}

\begin{example}
Let $A=A_{i}dx^{i}$ be the connection 1-form on $\C^{p}$. Then it follows from~(\ref{eq:AS-asymptotic-geometric-data}) 
that the covariant derivative $ \nabla_{i}= \partial_{i}+ \frac14 \omega_{ikl}(x)c(e^k)c(e^l)+A_{i}$ on $\sS_{n}\otimes 
\C^{p}$ has Getzler order 1 and its model operator is
    \begin{equation}
        \nabla_{i (1)} =\partial_{i}-\frac14 R_{ij}x^j , \qquad \text{where}\ 
        R_{ij}=\sum_{k<l} R_{ijkl}^{TM}(0) dx^k \wedge dx^l. 
        \label{eq:AS.model-spin-connection}
    \end{equation}
\end{example}

\begin{remark}
In what follows, we shall often look at symbols or operators up to terms that have lower Getzler orders. For this reason, 
it is convenient to use the notation $\op{O}_{G}(m)$ to denote any remainder term (symbol or operator) of Getzler 
order $\leq m$. 
Note that in view of~(\ref{eq:Ge.leading-deltal}) a Volterra symbol $q\in \Sv^{*}(\R^{n}\times \R^{n}\times \R, 
M_{p}(\C))\otimes \Lambda(n)$ is $\op{O}_{G}(m)$ if and only if, for all $(x,\xi,\tau)\in \R^{n}\times \R^{n}\times \R$, $(\xi,\tau)\neq 0$, 
\begin{equation}
   \delta_{\lambda}^{*}q(x,\xi,\tau)= \op{O}(\lambda^{m}) \qquad   \text{as $\lambda \rightarrow 0$}.
   \label{eq:Ge.rescaling-OG(m)}
\end{equation}Moreover, two Volterra symbols have Getzler order $m$ and same leading Getzler-homogeneous symbol if and 
only if their difference is $\op{O}_{G}(m-1)$. 
 \end{remark}

\begin{lemma}[\cite{Po:CMP}]\label{lem:index.top-total-order-symbol-composition}
For  $j=1,2$ let $Q_{j}\in \Psi^{*}_{\op{v}}(\R^{n}\times\R, \sS_n\otimes \C^{p})$ have Getzler order $m_{j}$ 
and model 
operator $Q_{(m_{j})}$ and assume that either $Q_{1}$ or $Q_{2}$ is properly supported. Then 
\begin{equation}
    \sigma\left[Q_{1}Q_{2}\right]= Q_{(m_{1})} Q_{(m_{2})} +\op{O}_{G}(m_{1}+m_{2}-1).
\end{equation}
\end{lemma}

\begin{example}\label{ex:D^2.formula}
By Lichnerowicz's formula, near $x=0$, we have
    \begin{equation}
 \sD^2=  \sD^{2}_{\nabla^{E}}= -g^{ij}(\nabla_{i}^{\sS\otimes E}\nabla_{j}^{\sS\otimes E} -\Gamma_{ij}^k 
     \nabla_{k}^{\sS\otimes E})  + \frac{\kappa}4 +\frac{1}{2}c(e^{i})c(e^{j})F^{E}(e^{i},e^{j}),
     \label{eq:AS.lichnerowicz-bis}
\end{equation}
where the $\Gamma_{ij}^k$ are the Christoffel symbols of the metric, $\kappa$ is the scalar curvature, and $\{e^{j}\}$ 
is the coframe dual to the synchronous frame $\{e_{i}\}$. Therefore, 
 combining Lemma~\ref{lem:index.top-total-order-symbol-composition} with~(\ref{eq:AS-asymptotic-geometric-data}) 
 and~(\ref{eq:AS.model-spin-connection}) shows that $\sD_{\nabla^{E}}^{2}$ has Getzler order $2$ and its model operator is 
\begin{equation}
\sD^2_{(2)}=H_{R}+F^{E}(0), 
\end{equation}where we have set 
\begin{equation*}
   H_{R}=- \sum_{i=1}^n (\partial_{i}-\sum_{j=1}^n\frac14 R_{ij}x^j)^{2} \qquad \text{and} \qquad F^{E}(0)=\sum_{i<j} 
   F^{E}(0)(\partial_{i},\partial_{j})dx^{i}\wedge dx^{j}. 
\end{equation*}
\end{example}

In what follows, it would be convenient to introduce the variables $x'=(x^{1},\ldots, x^{a})$ and $x''=(x^{a+1},\ldots, 
x^{n})$, so that $x=(x',x'')$. When using these variables we shall denote by $q(x',x'';\xi',\xi'';\tau)$ and 
$K_{Q}(x',x'';y',y'';t)$ the respective symbol and kernel of any given "operator'' $Q\in \pvdo^{*}(\R^{n}\times 
\R, \C^{p})\otimes\Lambda(n)$.  We then define 
 \begin{equation}
I_{Q}(x',t):=\int_{\R^{n-a}} K_{Q}\left(x', x''; 0, (1-\phi^{\cN})x''; t\right)dx'', \qquad x'\in \R^{a}.
\label{IQonRn}
\end{equation}

 \begin{lemma}\label{lem:AS.approximation-asymptotic-kernel-j}
     Let $Q\in \Psi_{\op{v}}^{*}(\R^{n}\times\R, \sS_{n}\otimes \C^{p})$ have Getzler order $m$ and model operator 
     $Q_{(m)}$. In addition, let $j$ be an integer~$\leq n$. 
 \begin{enumerate}
  \item If $m-j$ is an odd integer, then
  \begin{equation*}
     \sigma[I_{Q}(0,t)]^{(j)}= \op{O}(t^{\frac{j-m-a-1}2}) \qquad \text{as $t\rightarrow 0^{+}$}. 
  \end{equation*}
 
   \item If $m-j$ is an even integer, then
   \begin{equation}
   \label{eq:I_Q.rel.its Gezler.sym}
   \sigma[I_{Q}(0,t)]^{(j)}= t^{\frac{j-m-a}2-1} 
         I_{Q_{(m)}}(0,1)^{(j)} + \op{O}(t^{\frac{j-m-a}2}) \qquad \text{as $t\rightarrow 0^{+}$}.  
   \end{equation}
 \end{enumerate} 
 \end{lemma}
\begin{proof}
 Let $q(x, \xi, \tau)\sim \sum_{k\leq m'} q_{k}(x, \xi, \tau)$ be the symbol of $Q$ and denote by $q_{(m)}(x,\xi, \tau)$ its principal 
 Getzler homogeneous symbol. Recall that Proposition~\ref{prop:Heat.asymptotic-IQ} provides us with an asymptotic for $I_{Q}(x,t)$ in terms of the symbol of $Q$ in 
 tubular coordinates. We shall use the tubular coordinates $(x',v)\in \R^{a}\times \R^{n-a}$ given by the change of 
 variable,
\begin{equation*}
   x=\psi(x',v):=\exp_{x'}\left(v_{1}e_{a+1}(x')+\cdots+ v_{b}e_{n}(x')\right), \qquad  (x',v)\in \R^{a}\times \R^{n-a},
\end{equation*}where on the far right-hand side we have identified $x'$ with $(x',0)\in\R^{n}$. 
 Note that, as the original coordinates are normal coordinates, for all $v \in \R^{n-a}$, we have
 \begin{equation}
     \psi(0,v)=\exp_{x}\left(v_{1}\partial_{a+1}+\cdots+ 
      v_{n-a}\partial_{n}\right)=(0,v) .
      \label{eq:ProfProp.tubular-coordinates-normal}
 \end{equation}
 Furthermore, in  what follows, upon identifying $\R^{n}$ and $\R^{a}\times \R^{n-a}$, it will be convenient to regard functions on $\R^{n}\times 
 \R^{n}\times \R$ as functions on $\R^{a}\times \R^{n-a}\times \R^{a}\times \R^{n-a}\times \R$.

 Let $\tilde{q}(x', v;\xi',\nu;\tau)\sim  \sum_{k\leq m'} \tilde{q}_{k}(x', v;\xi',\nu;\tau)$ be the symbol of $Q$ in the tubular 
 coordinates, i.e., $\tilde{q}(x', v;\xi',\nu;\tau)$ is the symbol of $\psi^{*}Q$. As in the tubular coordinates, the 
 derivative  
 $\phi'$ is constant along the fibers of $\cN^{\phi}$, we see that 
 $\phi^{\sS}$ too is fiberwise constant. Incidentally, in the notation of~(\ref{eq:Equivariant.twisted-symbol}) the symbols $\tilde{q}^{\sS}_{k}$ and 
 $\tilde{q}_{k}$ agree for all $k\leq m'$. Bearing this mind, Proposition~\ref{prop:Heat.asymptotic-IQ} shows that, as $t\rightarrow 0^{+}$, 
\begin{equation}
     \sigma[I_{Q}(0,t)]^{(j)}\sim  \sum_{\substack{\textup{$|\alpha|-k$ even} \\ k\leq m'}} t^{\frac{|\alpha|-(k+a+2)}{2}}  
     \int_{\R^{n-a}}\frac{v^{\alpha}}{\alpha !}\left(\partial_v^{\alpha}\sigma[\tilde{q}_{k}]^{(j)}\right)^{\vee}(0, 0;0, 
     (1-\phi^{\cN}(0))v;1)dv.
     \label{eq:sigmaIQ}
\end{equation}
Using~(\ref{eq:ProfProp.tubular-coordinates-normal}), the change of variable formula for symbols~(\cite[Thm.~18.1.17]{Ho:ALPDO3}) gives
\begin{equation*}
    \tilde{q}_{k}(0,v;\xi',\nu;\tau) =  \sum_{\substack{l-|\beta|+|\gamma|=k\\ 2|\gamma|\leq  |\beta|}} 
     a_{\alpha\beta}(0,v)\xi^{\gamma}D_{\xi}^{\beta}q_{k}\left(0,v;\xi',\nu;\tau\right) 
 \end{equation*}where the $a_{\beta\gamma}(x',v)$ are some smooth functions such that $a_{\beta\gamma}(x)=1$ when 
 $\beta=\gamma=0$. 
 Thus, 
\begin{equation}
 \sigma[I_{Q}(0,t)]^{(j)}\sim   \sum_{\substack{\textup{$|\alpha|-l+|\beta|-|\gamma|$ even} \\ l\leq m', \ 2|\gamma|\leq  
 |\beta|}} t^{\frac{|\alpha|-l+|\beta|-|\gamma|-(a+2)}{2}}  I^{(j)}_{l\alpha\beta\gamma}, 
 \label{eq:ProofProp.asymptotic-sigmaIQj}
\end{equation}where we have set
\begin{equation*}
 I^{(j)}_{l\alpha\beta\gamma}:= \int_{\R^{n-a}}a_{\alpha\beta}(0,v) \frac{v^{\alpha}}{\alpha!} 
 \left(\partial_v^{\alpha}\sigma[\xi^{\gamma}D_{\xi}^{\beta}q_{l}]^{(j)}\right)^{\vee}\left(0,0;0, 
 (1-\phi^{\cN}(0))v;1\right)dv.\nonumber
\end{equation*}

Note that the symbol $v^{\alpha}\partial_v^{\alpha}\sigma[q_{l}]^{(j)}(0,0;\xi',\nu;\tau)$ 
is Getzler homogeneous of degree $l+j-|\alpha|$. Therefore, it must be zero if $l+j-|\alpha|>m$, since otherwise  $Q$ 
would have Getzler order~$>m$. This implies that in~(\ref{eq:ProofProp.asymptotic-sigmaIQj}) all the coefficients $I^{(j)}_{l\alpha\beta\gamma}$ with 
$l+j-|\alpha|>m$ must be zero. Furthermore, the condition $2|\gamma|\leq |\beta|$ and implies that 
$|\gamma|-|\beta|\leq -\frac{1}{2}|\beta|$, and hence $|\gamma|-|\beta|\leq - 1$ unless $\beta=\gamma=0$. Therefore, 
 if $l+j-|\alpha|\leq m$ and $2|\gamma|\leq |\beta|$, then $t^{\frac{1}{2}(|\alpha|-l+|\beta|-|\gamma|-(a+2))}$ is 
 $\op{O}(t^{\frac{1}{2}(j-m-(a+2))})$ and even is $\op{o}(t^{\frac{1}{2}(j-m-(a+2))})$ if we further have $l+j-|\alpha|< m$ or 
 $(\beta,\gamma)\neq (0,0)$. Observe that the asymptotic~(\ref{eq:ProofProp.asymptotic-sigmaIQj}) contains only integer powers of $t$ (non-negative or negative). Therefore, from 
 the above observations  we deduce that if $m-j$ is odd, then  all the (non-zero) terms in~(\ref{eq:ProofProp.asymptotic-sigmaIQj}) are 
 $\op{O}(t^{\frac{1}{2}(j-m-(a+1))})$, and hence
\begin{equation*}
         \sigma[I_{Q}(0,t)]^{(j)} =\op{O}\left(t^{\frac{1}{2}(j-m-(a+1))}\right). 
\end{equation*}
Likewise, if $m-j$ is even, then all the terms in~(\ref{eq:ProofProp.asymptotic-sigmaIQj}) with $l+j-|\alpha|\neq  m$ or with $l-|\alpha|=m-j $ and 
 $(\beta,\gamma)\neq (0,0)$ are $\op{O}(t^{\frac{1}{2}(j-m-a)})$. Thus, 
 \begin{equation}
\sigma[I_{Q}(0,t)]^{(j)}= t^{\frac{j-(m+a+2)}2}
        \sum_{l-|\alpha|=m-j}  I^{(j)}_{l\alpha00} + \op{O}(t^{\frac{j-(m+a)}2}).
\label{middle} 
\end{equation}

To complete the proof it remains to identify the coefficient of $t^{\frac{j-(m+a+2)}2}$ in~(\ref{middle}) with 
$I_{Q_{(m)}}(0,1)^{(j)}$. To this end observe that the formula~(\ref{eq:index.asymptotic-symbol}) for $q_{(m)}$ at $x'=0$ gives 
\begin{equation*}
    q_{(m)}(0,v;\xi,\nu;\tau)^{(j)}=\sum_{k+j-|\alpha|=m}  \frac{v^{\alpha}}{\alpha!}  
     \partial_{v}^{\alpha}\left(\sigma[q_{k}]^{(j)}\right)(0,0;\xi,\nu;\tau).
\end{equation*}Thus, 
 \begin{align*} 
 I_{Q_{(m)}}(0,1)^{(j)}= & \sum_{k-|\alpha|=m-j} 
 \int_{\R^{n-a}}\frac{v^{\alpha}}{\alpha!}\left(\partial_v^{\alpha}\sigma[q_{m-j+|\alpha|}]^{(j)}\right)^{\vee}\left(0,0;0,(1-\phi^{\cN}(0))v;1\right)dv\\ 
 = & \sum_{l-|\alpha|=m-j}  I^{(j)}_{l\alpha00}. 
 \end{align*} 
This completes the proof.        
\end{proof}

We are now in a position to prove the key lemma of the proof of 
Theorem~\ref{thm:LEIT.local-equiv.-index-thm-pointwise}.  

 \begin{lemma}\label{lem:AS.approximation-asymptotic-kernel}
     Let $Q\in \Psi_{\op{v}}^{*}(\R^{n}\times\R, \sS_{n}\otimes \C^{p})$ have Getzler order $m$ and model operator 
     $Q_{(m)}$.  
 \begin{enumerate}
  \item If $m$ is an odd integer, then
  \begin{equation}
  \label{eq:Str.phi.I_Q.local.odd}
     \str_{\sS_{n}\otimes \C^{p}}\left[\phi^{\sS\otimes E}I_{Q}(0,t)\right]= \op{O}(t^{-\frac{m+1}2}) \qquad \text{as $t\rightarrow 0^{+}$}. 
  \end{equation}
 
   \item If $m$ is an even integer, then, as $t\rightarrow 0^{+}$, we have
   \begin{multline} 
    (\str_{\sS_{n}\otimes \C^{p}})\left[\phi^{\sS\otimes E}I_{Q}(0,t)\right]= \\
        (-i)^{\frac{n}{2}}t^{-(\frac{m}2+1)} 
        2^{\frac{a}{2}}  {\det}^{\frac12}\left(1-\phi^{\cN}\right) 
        \left|\tr_{\C^{p}}[\phi^{E}I_{Q_{(m)}}(0,1)]\right|^{(a,0)}  + \op{O}(t^{-\frac{m}2}).
         \label{eq:AS-convergence-symbol-KQ}
   \end{multline}
 \end{enumerate} 
 \end{lemma}
\begin{proof}
For $t>0$ set $\fa(t)= (1_{\sS_{n}}\otimes \tr_{\C^{p}})[\phi^{E}I_{Q}(0,t)]$. Then by Lemma~\ref{lm:ActionOnSpinorBundle} we have
\begin{equation}
\label{eq:str.phi.I_Q.sigma}
  (\str_{\sS_{n}\otimes \C^{p}})\left[\phi^{\sS\otimes E}I_{Q}(0,t)\right]=(-2i)^{\frac{n}{2}}\left| 
  \sigma\left[\phi^{\sS}\fa(t)\right]\right|^{(n)} \qquad \text{for all $t>0$}.
\end{equation}
As Lemma~\ref{lem:LEIT.symbol-phiS} ensures us that $\sigma\left[ \phi^{\sS}\right]$ is an element of $\Lambda^{(0,2\bt)}(n)$, 
using~(\ref{eq:SymbolQuantization}) we deduce that
\begin{equation*}
  \sigma\left[ \phi^{\sS}\fa(t)\right]^{(n)}= \sigma\left[\phi^{\sS}\right]^{(0,n-a)}\wedge  
  \sigma\left[\fa(t)\right]^{(a,0)} + \sum_{1\leq \ell \leq 
  \frac{1}{2}(n-a)}\varphi_{\ell}\left(\sigma\left[\fa(t)\right]^{(a,2\ell)} \right),
\end{equation*}where $\varphi_{\ell}:\Lambda^{(0,2\ell)}(n)\rightarrow \Lambda^{n}(n)$ is a linear map which does not depend on $\fa(t)$. 
Moreover, for $\ell=0,1,\ldots,\frac{1}{2}(n-a)$, Lemma~\ref{lem:AS.approximation-asymptotic-kernel-j} ensures that, as $t\rightarrow 0^{+}$, we have
\begin{itemize}
    \item  $\sigma\left[\fa(t)\right]^{(a,2\ell)}=\op{O}\left(t^{\ell -\frac{m+1}{2}}\right)$ when $m$ is odd.

    \item  $\sigma\left[\fa(t)\right]^{(a,2\ell)}=t^{\ell -\frac{m}{2}-1}\tr_{\C^{p}}\left[\phi^{E}I_{Q_{(m)}}(0,1)^{(a,2\ell)}\right] +  
    \op{O}\left(t^{\ell -\frac{m}{2}}\right)$ when $m$ is even. 
\end{itemize}
Therefore, when $m$ is odd, we obtain 
\begin{equation*}
     \sigma\left[ \phi^{\sS}\fa(t)\right]^{(n)}=\op{O}\left(t^{-\frac{m+1}{2}}\right) \qquad \text{as $t\rightarrow 0^{+}$}. 
\end{equation*}
Combining this with~(\ref{eq:str.phi.I_Q.sigma}) we get the asymptotic~(\ref{eq:Str.phi.I_Q.local.odd}) when $m$ is odd. When $m$ is even,
we see that, as $t\rightarrow 0^{+}$, we have 
\begin{equation*}
    \sigma\left[ \phi^{\sS}\fa(t)\right]^{(n)}=  t^{-(\frac{m}2+1)}\sigma\left[\phi^{\sS}\right]^{(0,n-a)}\wedge 
    \tr_{\C^{p}}[\phi^{E}I_{Q_{(m)}}(0,1)]^{(a,0)}+  \op{O}(t^{-\frac{m}2}).
\end{equation*}
Combining this with~(\ref{eq:str.phi.I_Q.sigma}) and the formula~(\ref{eq:LEIT.top-symbol-phiS}) for $\sigma\left[\phi^{\sS}\right]^{(0,n-a)}$ yields the 
asymptotic~(\ref{eq:AS-convergence-symbol-KQ}) when $m$ is even. The proof is complete. 
\end{proof}

As in~\cite{LYZ:TAMS} it is convenient to introduce the following curvature matrices:
\begin{equation*}
    R':=(R_{ij})_{1\leq  i,j\leq a} \qquad \text{and} \qquad  R'':=(R_{a+i,a+j})_{1\leq  i,j\leq n-a}.
\end{equation*}
Note that~(\ref{eq:LEIT.splitting-curvature}) implies that the component in $\Lambda^{\bt,0}(n)$ of $R'$ (resp., $R''$) is $R^{TM^{\phi}}(0)$ (resp., $R^{\cN^{\phi}}(0)$). 

\begin{lemma}\label{lem:GetzlerOrderParametrix} Let $Q\in \pdo^{-2}(\R^{n}\times \R, \sS_{n}\otimes \C^{p})$ be a parametrix for 
    $\sD^{2}+\partial_{t}$. Then 
\begin{enumerate}
    \item  $Q$ has Getzler order~$-2$ and its model operator is
\begin{equation*}
        Q_{(-2)}=(H_{R}+\partial_{t})^{-1}\wedge \exp\left(-tF^{E}(0)\right).
\end{equation*}
    
    \item  For all $t>0$, we have 
    \begin{gather}
        I_{Q_{(-2)}}(0,t)= I_{(H_{R}+\partial_{t})^{-1}}(0,t)\wedge \exp\left(-tF^{E}(0)\right),\label{eq:ProofProp.IHR0}\\
        I_{(H_{R}+\partial_{t})^{-1}}(0,t)=\frac{(4\pi t)^{-\frac{a}{2}}}{
        {\det}^{\frac12}\left(1-\phi^{\cN}\right)}   
        {\det}^{\frac12}\left(\frac{tR'/2}{\sinh(tR'/2)}\right) {\det}^{-\frac12} \left(1-\phi^{\cN}e^{-tR''}\right) .
        \label{eq:ProofProp.IHR}
    \end{gather}
\end{enumerate}   
\end{lemma}
\begin{proof}
The first part is contained in~\cite[Lemma 5]{Po:CMP}.  This immediately gives the formula~(\ref{eq:ProofProp.IHR0}). 
The formula for $ I_{(H_{R}+\partial_{t})^{-1}}(0,t)$ is obtained exactly like in~\cite[p.\ 459]{LM:DMJ}. For reader's 
convenience we mention the main details of this computation. 

The kernel of $(H_{R}+\partial_{t})^{-1}$ can be determined from the arguments of~\cite{Ge:SPLASIT}. More 
precisely, let $A\in \mathfrak{so}_{n}(\R)$ and set $B=A^{t}A$. Consider the harmonic oscillators,
\begin{equation*}
    H_{A}:=-\sum_{1\leq i\leq n}(\partial_{i}+\ii A_{ij}x^{j})^2 \quad \text{and} \quad H_{B}:=-\sum_{1\leq i\leq n}\partial_{i}^{2} 
    +\frac{1}{4}\acou{Bx}{x}.
\end{equation*}In particular substituting $A=\frac{1}{2}\ii R$ in the formula for $H_{A}$ gives $H_{R}$. In addition, define
\begin{equation*}
    X:= \ii \sum_{i,j} A_{ij}x^{i}\partial_{j} = \ii \sum_{i<j}A_{ij}(x^{i}\partial_{j}-x^{j}\partial_{i}).
\end{equation*}Note that $H_{A}=H_{B}+X$. Observe also that, as $X$ is a linear combination of the infinitesimal rotations 
$x^{i}\partial_{j}-x^{j}\partial_{i}$, the $O(n)$-invariance of $H_{B}$ implies that $[H_{B},X]=0$. Thus, 
\begin{equation}
    e^{-tH_{A}}=e^{-tX}e^{-tH_{B}} \qquad \forall t\geq 0.
    \label{eq:ProofProp.commutationHB-X}
\end{equation}

The heat kernel of $H_{B}$ is determined by using Melher's formula in its version of~\cite{Ge:SPLASIT}. We get
\begin{equation}
    K_{(H_{B}+\partial_{t})^{-1}}(x,y,t)= (4\pi t)^{-\frac{n}{2}} {\det}^{\frac{1}{2}}\left( \frac{t\sqrt{B}}{\sinh 
    (t\sqrt{B})}\right) \exp \left( -\frac{1}{4t}\Theta_{B}(x,y,t)\right), \quad t>0, 
    \label{eq:ProofProp.MelherHB}
 \end{equation}where we have set
\begin{equation*}  \Theta_{B}(x,y,t):= \biggl\langle\frac{t\sqrt{B}}{\tanh (t\sqrt{B})}x,x \biggr \rangle + 
    \biggl\langle\frac{t\sqrt{B}}{\tanh (t\sqrt{B})}y,y \biggr \rangle 
    -2\biggl\langle\frac{t\sqrt{B}}{\sinh (t\sqrt{B})}x,y \biggr \rangle , \nonumber
\end{equation*}Here $\sqrt{B}$ is any square root of $B$ (e.g., $\sqrt{B}=\ii A$). Note that the right-hand side of~(\ref{eq:ProofProp.MelherHB}) is 
actually an analytic function of $(\sqrt{B})^{2}$. 

We also observe that for $t\in \R$ the matrix $e^{-t\sqrt{-1}A}$ is an element of $\op{O}(n)$, since in a suitable 
orthonormal basis it can 
be written as a block diagonal of $2\times 2$ rotation matrices with purely imaginary angles. Moreover, the family of 
operators $u\rightarrow u(e^{-t\sqrt{-1}A})$, $t\in \R$, is a one-parameter group of operators on $L^{2}(\R^{n})$ with 
infinitesimal generator $X$, so it agrees with $e^{-tX}$ for $t>0$. 
Combining this with~(\ref{eq:ProofProp.commutationHB-X}) and~(\ref{eq:ProofProp.MelherHB}) then gives 
\begin{gather*}
    K_{(H_{A}+\partial_{t})^{-1}}(x,y,t)= (4\pi t)^{-\frac{n}{2}} {\det}^{\frac{1}{2}}\left( \frac{t\sqrt{B}}{\sinh 
   (t \sqrt{B})}\right) \exp \left( -\frac{1}{4t}\Theta_{A}(x,y,t)\right),\\
     \Theta_{A}(x,y,t):= \biggl\langle\frac{t\sqrt{B}}{\tanh (t\sqrt{B})}x,x \biggr \rangle + 
    \biggl\langle\frac{t\sqrt{B}}{\tanh (t\sqrt{B})}y,y \biggr \rangle 
    -2\biggl\langle\frac{t\sqrt{B}}{\sinh (t\sqrt{B})}e^{- t\ii A}x,y \biggr \rangle, 
\end{gather*}where we have used the fact that $e^{- t\ii A}$ is an orthogonal matrix. 
Substituting $A=\frac{1}{2}\ii R$ and  $\sqrt{B}=\frac{1}{2}R$ then gives the kernel of 
$(H_{R}+\partial_{t})^{-1}$. We obtain
\begin{equation}
    K_{(H_{R}+\partial_{t})^{-1}}(x,y,t)= (4\pi t)^{-\frac{n}{2}} {\det}^{\frac{1}{2}}\left( \frac{tR/2}{\sinh (
   tR/2)}\right) \exp \left( -\frac{1}{4t}\Theta_{R}(x,y,t)\right), \quad t>0, 
  \label{eq:ProofProp.MelherHR}
  \end{equation}where we have set
\begin{equation*}
     \Theta_{R}(x,y,t):= \biggl\langle\frac{tR/2}{\tanh (tR/2)}x,x \biggr \rangle + 
    \biggl\langle\frac{tR/2}{\tanh (tR/2)}y,y \biggr \rangle 
    -2\biggl\langle\frac{tR/2}{\sinh (tR/2)}e^{tR/2}x,y \biggr \rangle. \nonumber
\end{equation*}

We are ready to compute $I_{(H_{R}+\partial_{t})^{-1}}(0,t)$. From~(\ref{IQonRn}) and~(\ref{eq:ProofProp.MelherHR}) we get 
\begin{equation}
   I_{(H_{R}+\partial_{t})^{-1}}(0,t)=(4\pi t)^{-\frac{n}{2}} {\det}^{\frac{1}{2}}\left( \frac{ tR/2}{\sinh 
   (tR/2)}\right) \int_{\R^{n-a}} \exp \left( -\frac{1}{4t}\Theta(v,t)\right)dv, 
   \label{eq:ProofProp.IHR-Theta}
\end{equation}where $\Theta(v,t):= \Theta_{R}(v,\phi^{\cN} v,t)$. Set $\scA=\frac{1}{2}tR''$. As
$[\phi^{\cN}, \scA]=0$, we see that
\begin{align*}
\Theta(v,t) & =  \biggl\langle\frac{\scA}{\tanh \scA}v,v \biggr \rangle + 
    \biggl\langle\frac{\scA}{\tanh \scA}\phi^{\cN} v, \phi^{\cN} v \biggr \rangle 
    -2\biggl\langle\frac{\scA}{\sinh \scA}e^{\scA}v,\phi^{\cN} v \biggr \rangle \\
    &  = 2 \biggl\langle\frac{\scA}{\sinh \scA}\left(\cosh \scA-  \left(\phi^{\cN }\right)^{-1}e^{\scA}\right)v, v \biggr \rangle .
\end{align*}Note that
\begin{align*}
    \left(\cosh \scA-  \left(\phi^{\cN }\right)^{-1}e^{\scA}\right)+\left(\cosh \scA- 
     \left(\phi^{\cN }\right)^{-1}e^{\scA}\right)^{T} & = e^{\scA}+e^{-\scA}- \left(\phi^{\cN }\right)^{-1}e^{\scA}-\phi^{\cN} e^{-\scA}\\ & =
     e^{\scA}\left(1- \left(\phi^{\cN }\right)^{-1}\right)(1-\phi^{\cN} e^{-2\scA}).
\end{align*}Therefore, using the formula for the integral of a Gaussian function and its extension to 
Gaussian functions associated to form-valued symmetric matrices, we get
\begin{multline*}
  \int_{\R^{n-a}} \exp \left( -\frac{1}{4t}\Theta(v,t)\right)dv =    \int_{\R^{n-a}} \exp \left( -\frac{1}{4t} 
  \biggl\langle\frac{\scA}{\sinh \scA}e^{\scA}\left(1- \left(\phi^{\cN }\right)^{-1}\right)(1-\phi^{\cN} e^{-2\scA}) v, v \biggr \rangle  \right)dv\\
    =  (4\pi)^{\frac{n-a}{2}}{\det}^{-\frac{1}{2}} \left(\frac{\scA}{\sinh \scA}\right) {\det}^{-\frac{1}{2}} 
    \left[e^{\scA}\left(1- \left(\phi^{\cN }\right)^{-1}\right)\right]{\det}^{-\frac{1}{2}}(1-\phi^{\cN} e^{-2\scA}).
\end{multline*}
We observe that ${\det}^{-\frac{1}{2}} \left[e^{\scA}\left(1- \left(\phi^{\cN }\right)^{-1}\right)\right]={\det}^{-\frac{1}{2}} (1-\phi^{\cN})$, and so
using~(\ref{eq:ProofProp.IHR-Theta}) we get 
\begin{equation*}
            I_{(H_{R}+\partial_{t})^{-1}}(0,t)={(4\pi)^{-\frac{a}{2}}}{
        {\det}^{-\frac12}\left(1-\phi^{\cN}\right)}   
        {\det}^{\frac12}\left(\frac{tR'/2}{\sinh(tR'/2)}\right) {\det}^{-\frac12} \left(1-\phi^{\cN} e^{-tR''}\right).
\end{equation*}This proves~(\ref{eq:ProofProp.IHR}) and completes the proof. 
\end{proof}

Let us go back to the proof of Theorem~\ref{thm:LEIT.local-equiv.-index-thm-pointwise}. Let $Q\in \pdo^{-2}(\R^{n}\times \R, \sS_{n}\otimes \C^{p})$ be a 
parametrix for $\sD^{2}+\partial_{t}$. Thanks to 
Lemma~\ref{lem:GetzlerOrderParametrix} we know that $Q$ has Getzler order $-2$. Therefore, 
using~(\ref{eq:ActionOnBundle}) and (\ref{eq:AS-convergence-symbol-KQ}) we see that, as $t\rightarrow 0^{+}$, 
\begin{align}
\str_{\sS\otimes E} \left[ \phi^{\sS\otimes E}(x_{0})I_{(\sD^{2}_{\nabla^{E}}+\partial_{t})^{-1}} (x_{0},t)\right] & =    
(\str_{\sS_{n}\otimes \C^{p}}) \left[ (\phi^{\sS}\otimes \phi^{E}) I_{Q} (0,t)\right]  
+\op{O}(t^{\infty}) \nonumber \\
& = (-i)^{\frac{n}{2}}2^{\frac{a}{2}}{\det}^{\frac12}\left(1-\phi^{\cN}\right)\left| \tr_{\C^{p}}[ \phi^{E}
I_{Q_{(-2)}}(0,1)]\right|^{(a,0)} +\op{O}(t).
\label{eq:LEIT.reduction-model}
\end{align}
As noted above, the components in $\Lambda^{\bt,0}(n)$ of the curvatures $R'$ and $R''$ are $R^{TM^{\phi}}(0)$ and $R^{\cN^{\phi}}(0)$, 
respectively. Likewise, the component in $\Lambda^{\bt,0}(n)$ of the curvature $F^{E}(0)$ is $F^{E}_{0}(0)$. 
Therefore, using~(\ref{eq:ProofProp.IHR0})--(\ref{eq:ProofProp.IHR}) we see that the component in $\Lambda^{\bt,0}(n)$ of $I_{Q_{(-2)}}(0,1)$ is 
\begin{equation}
  I_{Q_{(-2)}}(0,1)^{(\bt,0)}=       \frac{(4\pi)^{-\frac{a}{2}}}{
        {\det}^{\frac12}\left(1-\phi^{\cN}\right)}  \hat{A}(R^{TM^{\phi}}(0))\wedge 
        \nu_{\phi}\left(R^{\cN^{\phi}}(0)\right)\wedge \Ch_{\phi}\left(F^{E}(0)\right).
        \label{eq:LEIT.horizontal-component}
\end{equation}
Combining this with~(\ref{eq:LEIT.reduction-model}) we deduce that, as $t\rightarrow 0^{+}$, we have
\begin{multline*}
  \str_{\sS\otimes E}\left[ \phi^{\sS}(x_{0})I_{(\sD^{2}_{\nabla^E}+\partial_{t})^{-1}} (x_{0},t)\right] \\= (-i)^{\frac{n}{2}} 
  (2\pi)^{-\frac{a}{2}} \biggl|\hat{A}(R^{TM^{\phi}}(0))\wedge 
  \nu_{\phi}\left(R^{\cN^{\phi}}(0)\right) \wedge \Ch_{\phi}\left(F^{E}(0)\right)\biggr|^{(a,0)}+ \op{O}(t). 
\end{multline*}This gives the asymptotic~(\ref{eq:LEIT.local-equiv.-index-thm-pointwise}) at the point $x_0$. This 
completes the proof of Theorem~\ref{thm:LEIT.local-equiv.-index-thm-pointwise} and the local equivariant index theorem. 

\begin{remark}
 As Theorem~\ref{thm:LEIT.local-equiv.-index-thm-pointwise} is a purely local statement, it also allows us to obtain the local equivariant index theorem for 
 Dirac operators acting on sections of any $G$-equivariant Clifford module $E$, where $G$ is any compact group of 
 orientation-preserving isometries. This only amounts to replace $\Ch_{\phi}(F^{E})$ by $\Ch_{\phi}(F^{E/\sS})$, where 
 $F^{E/\sS}$ is the twisted curvature in the sense of~\cite{BGV:HKDO}. In particular, this enables us to recover the 
 local equivariant index theorem for the de Rham and signature complexes with coefficients in any $G$-equivariant 
 Hermitian vector bundle. 
\end{remark}

It should be stressed out that the above proof of the local equivariant index theorem actually gives a more general 
result. The key lemma in the proof is Lemma~\ref{lem:AS.approximation-asymptotic-kernel}, which was specialized to $Q=(\sD^{2}+\partial_{t})^{-1}$. This lemma 
actually holds for general Volterra \psidos. As a result, this allows us to obtain a version of Theorem~\ref{thm:LEIT.local-equiv.-index-thm-pointwise} for Volterra \psidos\ as 
follows. 

Following~\cite{Po:CMP} we shall call \emph{synchronous normal coordinates} centered at a point $x_{0}\in M$ the data 
of normal coordinates centered at $x_{0}$ and trivialization of the spinor bundle via synchronous tangent frame 
associated to an oriented orthonormal basis $e_{1},\ldots,e_{n}$ of $T_{x_{0}}M$. If $x_{0}\in M^{\phi}_{a}$ for a 
given $\phi \in G$, we shall call such a basis \emph{admissible} if $e_{1},\ldots,e_{a}$ is an oriented orthonormal 
basis of $T_{x_{0}}M_{a}^{\phi}$ (which implies that $e_{a+1},\ldots,e_{n}$ is an oriented orthonormal basis of 
$\cN_{x_{0}}^{\phi}$). 

\begin{definition}
Let $Q\in \pvdo^{\bt}(M\times \R,\sS\otimes E)$.  
\begin{enumerate}
    \item  We shall say that $Q$ has Getzler order $m$ at a given point $x_{0}\in M$, when, for any synchronous normal 
    coordinates centered at $x_{0}$ over which $E$ is trivialized, the operator $Q$ agrees up near $x=0$ with an operator 
    $\tilde{Q}\in \pvdo^{\bt}(\R^{n}\times \R, \sS_{n}\otimes \C^{p})$ 
that has  Getzler order~$m$.

    \item  Given a subset $S\subset M$, we shall say that $Q$ has Getzler order $m$ along $S$ when it has Getzler 
    order~$\leq m$ at every point  of $S$. 
\end{enumerate}
\end{definition}

\begin{example}
    The operator  $(\sD_{\nabla^{E}}^{2}+\partial_{t})^{-1}$ has Getzler order~$-2$ at every point of $M$.
\end{example}

\begin{remark}
    It can be shown that the condition in (1) holds for some synchronous normal coordinates centered at $x_{0}$ and 
    some trivialization of $E$ over these coordinest, then it holds for any such data. That is, the notion of Getzler 
    of order at a given point of $M$ is independent of the choice of the synchronous normal coordinates and trivialization of 
    $E$ near that point. 
 \end{remark}

\begin{theorem}\label{thm:LEIT.Volterra-LEIT}
 Given  $\phi\in G$, let $Q\in \pvdo^{\bt}(M\times \R,\sS\otimes E)$ have Getzler order~$m$ along the fixed-point set 
 $M^{\phi}$. 
    \begin{enumerate}
        \item  If $m$ is odd, then, uniformly on $M^{\phi}$, 
        \begin{equation}
            \str_{\sS\otimes E}\left[ \phi^{\sS\otimes E}(x)I_{Q}(x,t)\right]=\op{O}\left( t^{-\frac{m+1}{2}}\right) \qquad 
            \text{as $t\rightarrow 0^{+}$}.
            \label{eq:LEIT.Volterra-LEIT-odd}
        \end{equation}
    
        \item  If $m$ is even, then, as $t\rightarrow 0^{+}$ and uniformly on each fixed-point submanifold 
        $M_{a}^{\phi}$, we have
        \begin{equation}
            \str_{\sS\otimes E}\left[ \phi^{\sS\otimes E}(x)I_{Q}(x,t)\right]=    \gamma_{\phi}(Q)(x)t^{-\frac{m}{2}-1}             + 
            \op{O}\left( t^{-\frac{m}{2}}\right),
             \label{eq:LEIT.Volterra-LEIT-even}
        \end{equation}where $\gamma_{\phi}(Q)(x)$ is a function on $M^{\phi}_{a}$ such that, for all $x_{0}\in M_{a}^{\phi}$, in 
        any synchronous normal coordinates centered at $x_{0}$ over which $E$ is trivialized, we have
        \begin{equation*}
            \gamma_{\phi}(Q)(0)= (-i)^{\frac{n}{2}} 2^{\frac{a}{2}}{\det}^{\frac12}\left(1-\phi^{\cN}(0)\right)  \left| \tr_{E}\left[ 
            \phi^{E}(0)I_{Q_{(m)}}(0,1)\right]\right|^{(a,0)},
        \end{equation*} where $Q_{(m)}$ is the model operator of $Q$.  
          \end{enumerate}
\end{theorem}
\begin{proof}
 The proof follows the outline of the proof of Theorem~\ref{thm:LEIT.local-equiv.-index-thm-pointwise}. More precisely, as in 
 the proof of Theorem~\ref{thm:LEIT.local-equiv.-index-thm-pointwise} it is enough to prove the 
 asymptotics~(\ref{eq:LEIT.Volterra-LEIT-odd})--(\ref{eq:LEIT.Volterra-LEIT-even}) pointwise, since we are already 
 provided by a \emph{uniform} short-time asymptotic for $I_{Q}(x,t)$ thanks to 
 Proposition~\ref{prop:Heat.asymptotic-IQ}. However, these asymptotics are nothing but the contents of 
 Lemma~\ref{lem:AS.approximation-asymptotic-kernel}. The proof is complete. 
\end{proof}

\begin{remark}
    With some additional work it can be shown that, if $Q\in \pvdo^{\bt}(M\times \R, \sS\otimes E)$ have Getzler 
    order~$m$ along $M^{\phi}$, then, for all $a=0,2,\ldots,n$, there is a 
  unique section $\Upsilon_{\phi}(Q)(x)$ of $\left( \Lambda^{\bt}_{\C}T^{a}M^{\phi}_{a}\right) \otimes \End (E)$ over $M_{a}^{\phi}$ such that, 
  for every point $x_{0}\in M_{a}^{\phi}$, in any admissible synchronous normal coordinates centered at $x_{0}$ over which $E$ is 
  trivialized, we have 
  \begin{equation*}
    \Upsilon_{\phi}(Q)(0)= (-i)^{\frac{n}{2}} 2^{\frac{a}{2}}{\det}^{\frac12}\left(1-\phi^{\cN}(0)\right)  
    I_{Q_{(m)}}(0,1)^{(\bt,0)}. 
  \end{equation*}We then have 
  \begin{equation*}
      \gamma_{\phi}(Q)(x)= \left| \tr_{E}\left[ 
            \phi^{E}(x) \Upsilon_{\phi}(Q)(x)\right]\right|^{(a,0)} \qquad \text{for all $x\in M^{\phi}$}. 
  \end{equation*}
\end{remark}

An interesting consequence of Theorem~\ref{thm:LEIT.Volterra-LEIT} is that it provides us with the following  \emph{differentiable} version of the local equivariant index 
theorem.

\begin{proposition}\label{thm:Diff.local.equiv.index.thm} 
Given $\phi\in G$, let $P:C^{\infty}(M,\sS\otimes E)\rightarrow C^{\infty}(M,\sS\otimes E)$ be a differentiable operator 
of Getzler order~$m$ along the fixed-point set $M^{\phi}$. 
\begin{enumerate}
    \item   If $m$ is odd, then
    \begin{equation}
        \Str\left[ Pe^{-t\sD_{\nabla^{E}}^{2}}U_{\phi}\right]= \op{O}\left( t^{-\frac{m-1}{2}}\right) \qquad 
            \text{as $t\rightarrow 0^{+}$}.
            \label{eq:LEIT.DLEIT-odd}
    \end{equation}

    \item  If $m$ is even, then, as $t\rightarrow 0^{+}$, we have 
    \begin{equation}
        \Str\left[ Pe^{-t\sD_{\nabla^{E}}^{2}}U_{\phi}\right]=   t^{-\frac{m}{2}}
        \int_{M^{\phi}} \gamma_{\phi}(P;\sD_{\nabla^{E}})(x)|dx| + \op{O}\left( t^{-\frac{m}{2}+1}\right) , 
                    \label{eq:LEIT.DLEIT-even}
    \end{equation}where $\gamma_{\phi}(P;\sD_{\nabla^{E}})(x)$ is a function on $M^{\phi}$ such that, for all $x_{0}\in 
    M_{a}^{\phi}$, $a=0,2,\ldots,n$, in 
        any synchronous normal coordinates centered at $x_{0}$ over which $E$ is trivialized, we have
        \begin{equation*}
            \gamma_{\phi}(P;\sD_{\nabla^{E}})(x)(0)= (-i)^{\frac{n}{2}} 2^{\frac{a}{2}}{\det}^{\frac12}\left(1-\phi^{\cN}(0)\right)  \left| \tr_{E}\left[ 
            \phi^{E}(0)I_{P_{(m)}(H_{R}+\partial_{t})^{-1}}(0,1)\wedge \exp(-F^{E}_{0}(0))\right]\right|^{(a,0)},
        \end{equation*} where $P_{(m)}$ is the model operator of $P$.  
\end{enumerate}
 \end{proposition}
  \begin{proof}
      By Proposition~\ref{TraceOfHeatKernelVB}, as $t\rightarrow 0^{+}$, we have 
      \begin{equation*}
            \Str\left[ Pe^{-t\sD_{\nabla^{E}}^{2}}U_{\phi}\right]= \int_{M^{\phi}}\str_{\sS\otimes 
            E}\left[\phi^{\sS\otimes E}(x)I_{P(\sD_{\nabla^{E}}^{2}+\partial_{t})^{-1}}(x,t)\right]|dx| +\op{O}(t^{\infty}).
      \end{equation*}
As the operator $Q=P(\sD_{\nabla^{E}}^{2}+\partial_{t})^{-1}$ has Getzler order $m-2$, using Theorem~\ref{thm:LEIT.Volterra-LEIT} we see 
that, when $m$ is odd the asymptotics~(\ref{eq:LEIT.DLEIT-odd}) holds. Moreover, when $m$ is even, as $t\rightarrow 0^{+}$, we obtain 
\begin{equation*}
    \Str\left[ Pe^{-t\sD_{\nabla^{E}}^{2}}U_{\phi}\right]= 
    t^{-\frac{m}{2}}\int_{M^{\phi}}\gamma_{\phi}(P;\sD_{\nabla^{E}})(x)dx| + \op{O}\left( t^{-\frac{m}{2}+1}\right), 
\end{equation*}where we have set $\gamma_{\phi}(P;\sD_{\nabla^{E}})(x)=\gamma(Q)(x)$. 

Let $x_{0}\in M^{\phi}_{a}$, $a=0,2,\ldots,n$, and let us consider admissible normal coordinates centered at $x_{0}$. 
Then the operator $Q=P(\sD_{\nabla^{E}}^{2}+\partial_{t})^{-1}$ has model operator 
\begin{equation*}
    Q_{(m-2)}=P_{(m)}(H_{R}+\partial_{t})^{-1}\wedge \exp(-tF^{E}(0)).
\end{equation*}
Therefore, setting $\mu= (-i)^{\frac{n}{2}} 2^{\frac{a}{2}}{\det}^{\frac12}\left(1-\phi^{\cN}(0)\right)$, we get
\begin{align*}
 \gamma_{\phi}(Q)(0)& =  \mu \left| \tr_{E}\left[\phi^{E}(0)I_{Q_{(m-2)}}(0,1)\right]\right|^{(a,0)}\\ 
 &=  \mu \left| 
 \tr_{E}\left[\phi^{E}(0)I_{P_{(m)}(H_{R}+\partial_{t})^{-1}}(0,1)\wedge 
 \exp\left(-tF_{0}(0)\right)\right]\right|^{(a,0)}.
\end{align*}This proves~(\ref{eq:LEIT.DLEIT-even}) and completes the proof. 
\end{proof}

\begin{remark}\label{rmk:LEIT.odd}
The considerations of this section on Getzler order and model operators of operators in $\pvdo^{\bt}(\R^{n}\times \R, 
\sS_{n}\otimes \C^{p})$ are restricted to even dimension $n$ in order to use the isomorphism $\End \sS_{n}\simeq 
\Cl_{\C}(n)$. As pointed out in~\cite{Po:CMP} the notions of Getzler order and model operators make sense for elements of  $\pvdo^{\bt}(\R^{n}\times \R, 
\C^{p})\otimes \Cl_{\C}(n)$ independently of the parity of $n$. In particular,  Lemma~\ref{lem:AS.approximation-asymptotic-kernel-j} holds \emph{verbatim} in 
this context.  When $n$ is odd, $\Cl_{\C}(n)$ is a $2$-cover to $\End \sS_{n}$, but, as observed in~\cite{BF:AEF1}, the 
trace on $\Cl_{\C}(n)^{-}=c\left( \Lambda^{2\bt+1}(n)\right)$ behaves essentially like the supertrace on $\End \sS_{n}$ 
in even dimension. We refer to~\cite{Po:Odd} for odd-dimensional analogues of 
Lemma~\ref{lem:AS.approximation-asymptotic-kernel} and Theorem~\ref{thm:LEIT.Volterra-LEIT} 
and some of their applications. 
\end{remark}

\begin{remark}\label{rmk:LEIT-family-settings}
    It is not difficult to various family settings the considerations of this sections on Getzler orders and model 
    operators on Volterra \psidos\ (see, e.g., \cite{YWang:AJM, YWang:JKT14}). In particular, this provides us with 
    proofs of the local equivariant family index theorem of Liu-Ma~\cite{LM:DMJ} and the infinitesimal equivariant index 
    theorem (a.k.a.~Kirillov Formula) of Berline-Vergne~\cite{BV:EIKCF} (see also~\cite{Bi:ILFHEP}). We refer to~\cite{YWang:AJM} 
    for a proof of the former result and to~\cite{YWang:JKT14} of the latter. We note that the proofs 
    of~\cite{YWang:AJM, YWang:JKT14} rely on Lemma~9.5 of ~\cite{PW:Preprint} which are not correct. This can fixed 
    by using the approriate extensions of Lemma~\ref{lem:AS.approximation-asymptotic-kernel} and Theorem~\ref{thm:LEIT.Volterra-LEIT} to the respective family settings 
    at stake in~\cite{YWang:AJM} and~\cite{YWang:JKT14}. 
\end{remark}

\section{Connes-Chern Character and CM Cocycle}\label{sec:spectral-triples}
In this section, we briefly recall the framework for the local index formula and explain how to extend it to the setup 
of spectral triples over locally convex algebras. In what follows we assume the notation, definitions and results of  the 
prequel~\cite{PW:NCGCGI.PartI} (which we shall refer as Part~I).

\subsection{The Connes-Chern character} The role of manifolds in noncommutative geometry is played by spectral triples. 
More specifically, a \emph{spectral triple} is a triple $(\cA,\cH,D)$, where
\begin{itemize}
    \item $\cH=\cH^{+}\oplus \cH$ is a $\Z_{2}$-graded Hilbert space.
    
    \item $\cA$ is a $*$-algebra represented by bounded operators on $\cH$ preserving its $\Z_{2}$-grading. 

    \item $D$ is a selfadjoint unbounded operator  on $\cH$ such that 
   \begin{itemize}
          \item $D$ maps $\dom (D)\cap \cH^{\pm}$ to $\cH^{\mp}$. 
          \item The resolvent $(D+i)^{-1}$ is a compact operator.
           \item $a \dom (D) \subset \dom (D)$ and $[D, a]$ is bounded for all $a \in \cA$. 
\end{itemize}
\end{itemize}
In particular, with respect to the splitting $\mathcal{H}=\mathcal{H}^+\oplus \mathcal{H}^-$ the 
    operator $D$ takes the form,
    \begin{equation*}
        D=
        \begin{pmatrix}
            0 & D^{-} \\
            D^{+} & 0
        \end{pmatrix}, \qquad D^{\pm}:\dom(D)\cap \cH^{\pm}\rightarrow \cH^{\mp}.
    \end{equation*}

The paradigm of a spectral triple is given by a Dirac spectral triple, 
\begin{equation*}
( C^{\infty}(M), L^{2}_{g}(M,\sS), \sD_{g}),
\end{equation*}
where $(M^{n},g)$ is a compact spin Riemannian manifold ($n$ even) and $\sD_{g}$ is its Dirac operator 
acting on the spinor bundle $\sS=\sS^{+}\oplus \sS^{-}$. Given any Hermitian  vector bundle $E$ over $M$, the datum of 
any Hermitian  connection on $E$ enables us to form the coupled operator $\sD_{\nabla^{E}}$ acting on the sections of $\sS\otimes E$. Its Fredholm index depends only on the 
$K$-theory class of $E$ and is given by the local index formula for Atiyah-Singer~\cite{AS:IEO1, AS:IEO3}, i.e., the 
equivariant index formula~(\ref{thm:EIT}) in the case $\phi=\op{id}_{M}$. 

Likewise, given any spectral triple $(\cA,\cH,D)$ and a Hermitian finitely generated projective module $\cE$ over 
$\cA$, the datum of any Hermitian connection $\nabla^{\cE}$ on $\cE$ gives rise to an unbounded operator 
$D_{\nabla^{\cE}}$ on the Hilbert 
space $\cH(\cE)=\cH\otimes_{\cA} \cE$. Furthermore, with respect to the splitting $\cH(\cE)=\cH(\cE^{+})\oplus \cH(\cE^{-})$, where 
$\cH(\cE^{\pm}):=\cH^{\pm}\otimes_{\cA} \cE$, the coupled operator takes the form,
 \begin{equation*}
     D_{\nabla^{\cE}}=
        \begin{pmatrix}
            0 & D^{-}_{\nabla^{\cE}} \\
            D^{+}_{\nabla^{\cE}} & 0
        \end{pmatrix}, \qquad D^{\pm}_{\nabla^{\cE}}:\dom(D^{\pm})\otimes_{\cA}\cE\rightarrow \cH^{\mp}(\cE).
    \end{equation*}
The operator $D_{\nabla^{\cE}}$ is selfadjoint and Fredholm, and its Fredholm index is then defined by
\begin{equation*}
    \ind D_{\nabla^{\cE}}=\ind D_{\nabla^{\cE}}^{+}=\dim \ker D_{\nabla^{\cE}}^{+}-\dim \ker D_{\nabla^{\cE}}^{-}. 
\end{equation*}
This index depends only on the class of $\cE$ in the $K$-theory $K_{0}(\cA)$, and hence gives rise to an additive index 
map $\ind_{D}:K_{0}(\cA)\rightarrow \Z$. 

Let us further assume that the spectral triple is \emph{$p$-summable} for some $p\geq 1$, i..e, the operator $|D|^{-p}$ 
is trace-class. In this case, 
Connes~\cite{Co:NCDG} constructed an even periodic cyclic cohomology class $\Ch(D)\in \op{HP}^{0}(\cA)$, called 
\emph{Connes-Chern character}, whose pairing 
with $K$-theory computes the index map. This construction is described in Part~I in the more general setting of twisted 
spectral triples. For reader's convenience, we recall this construction in the special case of ordinary spectral 
triples. 

\begin{proposition}[Connes~\cite{Co:NCDG}]
    Assume that $(\cA,\cH,D)$ is $p$-summable and $D$ is invertible.
    \begin{enumerate}
        \item  For any $q\geq \frac{1}{2}(p-1)$ the following formula defines a $2q$-cyclic cocycle on the algebra 
        $\cA$:
        \begin{equation*}
            \tau_{2q}^{D}(a^{0},\ldots, a^{2q}):= \frac{1}{2}\frac{q!}{(2q)!}\Str \left\{D^{-1}[D, a^0]\cdots D^{-1}[D, a^{2q}]\right\}, \quad a^{j}\in \cA,
\end{equation*}where $\Str=\Tr_{\cH^{+}}-\Tr_{\cH^{-}}$ is the supertrace on $\cL^{1}(\cH)$.
    
        \item The class of the cocycle $\tau_{2q}^{D}$ in the periodic cyclic cohomology $\op{HP}^{0}(\cA)$ is independent of the 
        value of $q$. 
    \end{enumerate}
\end{proposition}

The case where $D$ is non-invertible is dealt with by passing to the invertible double 
$(\tilde{\cA},\tilde{\cH},\tilde{D})$. Here $\tilde{\cH}={\cH}\oplus \cH$ is equipped with the 
$\Z_{2}$-grading $\tilde{\gamma}=\gamma\oplus \gamma$ and $\tilde{\cA}=\cA\oplus \C$ is the unitalization of $\cA$ 
represented in $\tilde{\cH}$ by $(a,\lambda)\rightarrow \begin{pmatrix}
    a+\lambda & 0 \\
    0 & \lambda
\end{pmatrix}$. The operator $\tilde{D}$ is given by
\begin{equation}
\label{eq:define.invertible.D}
    \tilde{D}= 
    \begin{pmatrix}
        D & \Pi_{0} \\
        \Pi_{0} & -D
    \end{pmatrix},
\end{equation}
where $\Pi_{0}$ is the orthogonal projection onto $\ker D$. This provides us with a $p$-summable spectral triple $(\tilde{\cA},\tilde{\cH},\tilde{D})$ and it can be checked that 
the operator $\tilde{D}$ is invertible. We thus have well defined cyclic cocycles $\tau_{2q}^{\tilde{D}}$, 
$q\geq \frac{1}{2}(p-1)$, on $\tilde{A}$. 

\begin{proposition}[Connes~\cite{Co:NCDG}]
 Assume that $(\cA,\cH,D)$ is $p$-summable. For $q\geq   \frac{1}{2}(p-1)$, denote by $\overline{\tau}_{2q}^{D}$ 
 the restriction to $\cA$ of the cyclic cocycle $\tau_{2q}^{\tilde{D}}$. 
 \begin{enumerate}
     \item  The cochain $\overline{\tau}_{2q}^{D}$ is a cyclic cocycle on $\cA$ whose class in  $\op{HP}^{0}(\cA)$ is independent of the 
        value of $q$. 
 
     \item  When $D$ is invertible, the cyclic cocycles $\tau_{2q}^{D}$ and $\overline{\tau}^{D}_{2q}$ are cohomologous in 
     $\op{HC}^{2q}(\cA)$, and hence define the same class in 
     $\op{HP}^{0}(\cA)$. 
 \end{enumerate}
\end{proposition}

\begin{remark}
The homotopy invariance of the cohomology class of $\overline{\tau}_{2q}^{D}$ shows that, for $q\geq \frac{1}{2}(p+1)$, 
 the class of $\overline{\tau}_{2q}^{D}$ in $\op{HC}^{2k}(\cA)$ 
remains unchanged if in the formula~(\ref{eq:define.invertible.D}) for $\tilde{D}$ we 
replace the projection $\Pi_{0}$ by any operators of the form $t_{1}+t_{2}\Pi_{0}$, with $t_{j}\geq 0$, $t_{1}+t_{2}>0$ (\emph{cf}.~\cite{Co:NCDG}). 
\end{remark}

\begin{definition}[Connes~\cite{Co:NCDG}]
Assume that $(\cA,\cH,D)$ is $p$-summable. The Connes-Chern character of $(\cA,\cH,D)$, denoted by $\Ch(D)$, is defined 
as follows:
\begin{itemize}
    \item  If $D$ is invertible, then $\Ch(D)$ is the common class in $\op{HP}^{0}(\cA)$ of the 
    cyclic cocycles $\tau_{2q}^{D}$ and $\overline{\tau}_{2q}^{D}$, with $q\geq \frac{1}{2}(p-1)$. 

    \item  If $D$ is not invertible, then $\Ch(D)$ is the common class in $\op{HP}^{0}(\cA)$ of the
    cyclic cocycles $\overline{\tau}_{2q}^{D}$, $q\geq \frac{1}{2}(p-1)$. 
\end{itemize}
\end{definition}

The above definition is motivated by the following index formula. 

\begin{proposition}[Connes~\cite{Co:NCDG}]
For any Hermitian finitely generated projective module $\cE$ over $\cA$ and any Hermitian connection $\nabla^{\cE}$ on 
$\cE$, we have
\begin{equation}
\label{eq:index.coupled.D.pairing.Conne-Chern.character}
 \ind D_{\nabla^{\cE}}= \acou{\Ch(D)}{\Ch(\cE)},
\end{equation}where $\acou{\cdot}{\cdot}$ is the duality pairing between cyclic cohomology and cyclic homology and 
$\Ch(\cE)$ is the Connes-Chern character of $\cE$ in the periodic cyclic homology $\op{HP}^{0}(\cA)$ (\emph{cf.}~Part~I).  
\end{proposition}

\begin{remark}
   All the above results continue to hold if in the definition of the cocycles $\tau_{2q}^{D}$ 
    and $\overline{\tau}_{2q}^{D}$ we replace the operator $D$ by any operator $D|D|^{-t}$ with $t\in [0,1]$. In 
    particular, for $t=1$ we obtain the sign operator $F=D|D|^{-1}$, which is often convenient to use for defining the Connes-Chern 
    character.
\end{remark}

\subsection{The CM cocycle}
The cocycles $\tau_{2q}^{D}$ and $\overline{\tau}_{2q}^{D}$ used in the definition of the Connes-Chern character are difficult to compute in practice, even in the case of a Dirac spectral 
triple (see~\cite{Co:LQ, BF:APDOIT}). Therefore, it was sought for an alternative representative of the Connes-Chern 
character which would be easier to compute. Such a representative is provided by the CM cocycle~\cite{CM:LIFNCG}, which 
is constructed  as follows. 

In what follows we shall assume that $(\cA,\cH,D)$ is $p^{+}$-summable for some $p\geq 1$, that is,
\begin{equation*}
    \mu_{j}(D^{-1})=\op{O}(j^{-\frac{1}{p}}) \qquad \text{as $j\rightarrow \infty$},
\end{equation*}where $ \mu_{j}(D^{-1})$ is the $(j+1)$-th eigenvalue of $|D^{-1}|=|D|^{-1}$ counted with multiplicity. 
This implies that $(\cA,\cH,D)$ is $p$-summable for all $q>p$. In addition, we set
\begin{equation}
\label{eq:def.D^0_A.op.class}
  \cD_{D}^{0}(\cA)=\cA+[D,\cA],
\end{equation}
where $\gamma=1_{\cH^+}-1_{\cH^-}$ is the grading operator. Note that the very definition of a spectral triple implies that the operators in $\cD_{D}^{0}(\cA)$ are bounded. 

\begin{definition}
    We shall say that $(\cA,\cH,D)$ is hypo-regular when, for any $X\in \cD_{D}^{0}(\cA)$, all the operators $D^{m}XD^{-m}$, $m\in \N_{0}$, are bounded.
\end{definition}

From now on we further assume that  $(\cA,\cH,D)$ is hypo-regular.  Set $\cH^{\infty}=\bigcap_{m\geq 0} \dom (D^{-m})$ and equip $\cH^{\infty}$ with the Fr\'echet space topology 
 defined by the seminorms $\xi\rightarrow \acou{(1+D^{2m})\xi}{\xi}$, $m\geq 0$. Then the hyporegularity condition 
 implies that any $X\in \cD_{D}^{0}(\cA)$ induces a continuous linear operator from $\cH^{\infty}$ to itself. Note also 
 that $\cH^{\infty}$ is a dense subspace of $\cH$. In addition, for $m\geq 0$, we denote by $\cD_{D}^{m}(\cA)$ 
the class of unbounded operators on $\cH$, where $ \cD_{D}^{0}(\cA)$ is defined as in~(\ref{eq:def.D^0_A.op.class}), and
  \begin{gather*}
  \cD^{1}(\cA)= D\cD_{D}^{0}(\cA)+\cD_{D}^{0}(\cA)D,\\
     \cD^{m}_{D}(\cA) = \cD^{1}(\cA)\cD_{D}^{m-1}(\cA)+ \cD^{2}(\cA)\cD_{D}^{m-2}(\cA) +\cdots+ \cD_{D}^{m-1}(\cA)\cD^{1}(\cA), \quad m\geq 2.
  \end{gather*}

Alternatively, we may regard each class $ \cD^{m}_{D}(\cA)$ as a subspace of $\cL(\cH^{\infty})$. 
A spanning set of $ \cD^{m}_{D}(\cA)$ is then obtained as follows. For $m\in \N_{0}$, define
\begin{equation*}
    \cP(m)=\left\{(\alpha,\beta)\in \N_{0}^{m+2}\times \{0,1\}^{m+1}; |\alpha|=\alpha_{0}+\cdots +\alpha_{m+1}=m\right\}. 
\end{equation*}
Then $\cD^{m}_{D}(\cA)$ is spanned by operators of the form,
\begin{equation}
\label{eq:definition.P_alphabeta}
    P_{\alpha,\beta}(a^{0},\ldots,a^{l})= D^{\alpha_{0}}(a^{0})^{\beta_{0}}[D,a^{0}]^{1-\beta_{0}}\cdots 
    D^{\alpha_{m'}}(a^{l})^{\beta_{l}}[D,a^{0}]^{1-\beta_{l}}D^{\alpha_{l+1}}, 
\end{equation}where the pair $(\alpha,\beta)$ ranges over $\cP(m)$ and 
$a^{0},\ldots,a^{l}$ range over $\cA$. Setting 
$X^{j}=(a^{j})^{\beta_{j}}[D,a^{j}]^{1-\beta_{j}}$ such an operator can be rewritten as
\begin{equation*}
    (D^{\alpha_{0}}X^{0}D^{-\alpha_{0}})(D^{\alpha_{0}+\alpha_{1}}X^{1}D^{-(\alpha_{0}+\alpha_{1})})\cdots (D^{\alpha_{0}+\cdots 
    +\alpha_{m}}X^{m}D^{-(\alpha_{0}+\cdots +\alpha_{m})})D^{m},
\end{equation*}This and the hypo-regularity condition imply that the operator $P_{\alpha,\beta}(a^{0},\ldots,a^{m})D^{-m}$ is bounded. 
Combining this with the 
$p^{+}$-summability of $D$ we then deduce that, for all $X\in \cD^{m}_{D}(\cA)$, the operator $X|D|^{-z-m}$ is 
trace-class for $\Re z>p$, and so by taking its super-trace we then obtain a function $z\rightarrow 
\Str\left[X|D|^{-z-m}\right]$ which is holomorphic on the half-plane $\Re z>p$. 

\begin{definition}
    We say that $(\cA,\cH,D)$ has discrete and simple dimension spectrum when $(\cA,\cH,D)$  is hypo-regular and there is a discrete subset $\Sigma\subset 
    \C$ such that, for all $m\in \N_{0}$ and $X\in \cD^{m}_{D}(\cA)$, the function $\Str \left[ 
    X|D|^{-z-m}\right]$ has an analytic continuation to $\C\setminus \Sigma$ with at worst simple pole singularities on $\Sigma$. The dimension spectrum of $(\cA,\cH,D)$ 
    is then defined as the smallest such set. 
\end{definition}

\begin{remark}
   The $p^{+}$-summability of $D$ implies that the dimension spectrum of $(\cA,\cH,D)$ is contained in the half-space $\Re z\leq p$. 
\end{remark}

In what follows, given $X\in \cD_{D}^{0}(\cA)$, for $j=0,1,2,\ldots$ we denote by $X^{[j]}$ the $j$-th iterated commutator of 
$D^{2}$ with $X$, i.e., 
\begin{equation*}
 X^{[0]}=X, \qquad    X^{[j]}= \overbrace{[D^{2},[D^{2},\ldots [D^{2},}^{\textrm{$j$ times}}X]\cdots ]] , \quad j\geq 1.
\end{equation*}
Note that $X^{[j]}$ is an element of $\cD_{D}^{2j}(\cA)$. 
 
\begin{definition}\label{eq:CM.regularity}
    We say that $(\cA,\cH,D)$ is regular when, for any $X\in \cD_{D}^{0}(\cA)$, all the operators $X^{[m]}D^{-m}$, $m\in 
    \N_{0}$, are bounded.  
\end{definition}

As shown by Connes-Moscovici~\cite{CM:LIFNCG}, assuming that $(\cA,\cH,D)$ is regular and has discrete and simple dimension spectrum enables us to construct a class 
of \psidos\ and an analogue of Guillemin-Wodzicki's noncommutative residue trace as follows. 

In the following we denote by $\cB$ the class of unbounded operators on $\cH$ that are linearly combination of 
operators of the form $X^{[m]}|D|^{-m}$ with $X\in \cD_{D}^{0}(\cA)$ and $m\in \N_{0}$. The regularity assumption implies that all the 
operators in $\cB$ are bounded. Moreover, for $r\in \R$ we denote by $\OP^{r}$ the class of unbounded operators $T$ on 
$\cH$ such that $\dom(T)\supset \cH_{(\infty)}$ and $T|D|^{-r}$ is bounded. 

\begin{definition}
 $\Psi_{D}^{q}(\cA)$, $q\in \C$, consists of unbounded operators $P$ on $\cH$ such that the domain of $P$ contains 
$\cH^{\infty}$ and there is an asymptotic expansion of the form,
\begin{equation*}
    P\sim \sum_{j \geq 0} b_{j}|D|^{q-j}, \qquad b_{j} \in \cB,
\end{equation*}in the sense that
\begin{equation*}
    \biggl( P -\sum_{j<N}b_{j}|D|^{q-j}\biggr) \in \OP^{\Re q -N} \qquad \forall N\in \N.
        \label{eq:ST.PsiDOs2}
\end{equation*}
\end{definition}

In particular, the above definition implies that any operator in $\Psi_{D}^{q}(\cA)$ induces a continuous linear operator from $\cH^{\infty}$ to 
itself. Moreover, the operators in $\Psi^{q}_{D}(\cA)$  with $\Re q\leq 0$ extend to bounded operators of $\cH$ to itself. 
Those operators are compact (resp., trace-class) when $\Re q<0$ (resp., $\Re q<-p$).  
Furthermore, it can be shown (see~\cite{CM:LIFNCG, Hi:RITCM}) that if $X \in \cD_{D}^{0}(\cA)$, then, for all $j\in \N$ and $z\in \C$,  
\begin{equation*}
    |D|^{2z}X^{[j]}\sim  \sum_{k\geq 0} \binom{z}{j}X^{[j+k]}|D|^{2z-2k}, 
\end{equation*}and hence $|D|^{2z}X^{[j]}$ is contained in $\Psi^{2z+j}_{D}(\cA)$. This implies that 
\begin{equation*}
    \Psi_{D}^{q_{1}}(\cA)\Psi_{D}^{q_{2}}(\cA) \subset \Psi_{D}^{q_{1}+q_{2}}(\cA) \qquad \forall q_{j}\in \C.
\end{equation*}

In addition, the existence of a discrete and simple dimension spectrum implies that, for any $P \in \Psi^{q}_{D}(\cA)$,  the 
function $z \rightarrow \Str \left[P|D|^{-2z}\right]$ has a 
meromorphic extension to the entire complex plane with at worst simple pole singularities. Note that the poles are 
contained in the half-plane $\Re z\leq \frac{1}{2}(p+\Re q)$. We then set
\begin{equation}
    \bint P:= \Res_{z=0} \Str \left[P|D|^{-2z}\right] \qquad \forall P \in \Psi^{*}_{D}(\cA). 
    \label{eq:ST.residual-trace}
\end{equation}As it turns out (\emph{cf.}~\cite{CM:LIFNCG}), this defines a trace on $\Psi^{\bt}_{D}(\cA)$, i.e., 
\begin{equation*}
    \bint P_{1}P_{2}=\bint P_{2}P_{1} \qquad \text{for all $P_{j}\in \Psi_{D}^{q_{j}}(\cA)$}. 
\end{equation*}
This residual functional is the analogue of the noncommutative residue trace of Guillemin~\cite{Gu:NPWF} and 
Wodzicki~\cite{Wo:LISA}. Note also that it vanishes on all operators $P\in \Psi^{q}_{D}(\cA)$ with $\Re q<-p$, 
and so this is a \emph{local} functional. 

In what follows for $q\geq 1$ and $\alpha\in \N_{0}^{2q}$ we set
\begin{equation}
\label{eq:constrant.in.CM.cocycle}
            c_{q,\alpha}=(-1)^{|\alpha|}
    \frac{(q-1)!(\alpha_{1}+\cdots+\alpha_{2q})!}{\alpha!(\alpha_{1}+1)\cdots (\alpha_{2q}+2q)}. 
\end{equation}

\begin{theorem}[Connes-Moscovici~\cite{CM:LIFNCG}]
\label{thm:CM.cocyle.residue.formula}
Assume that $(\cA,\cH,D)$ is $p^{+}$-summable, regular and has a 
    discrete and simple dimension spectrum. Then
    \begin{enumerate}
        \item  The following formulas define an even periodic cyclic cocycle $\varphi^{\CM}=(\varphi_{2q}^{\CM})_{q\geq 0}$ on $\cA$:
         \begin{gather}
         \label{eq:CM.cocycle.residue.0}
          \varphi_{0}^{\CM}(a^{0})=\underset{z=0}{\Res}\left\{\Gamma(z)\Str\left[a^{0}|D|^{-2z}\right]\right\}+\Str\left[ a^{0}\Pi_{0}\right], \quad 
          a^{0}\in \cA,   \\
            \varphi_{2q}^{\CM}(a^{0},\ldots,a^{2q})=\sum_{\alpha\in \N_{0}^{q}}c_{q, \alpha}\bint \gamma a^{0}[D,a^{1}]^{[\alpha_{1}]}\cdots  
       [D,a^{2q}]^{[\alpha_{2q}]}|D|^{-2(|\alpha|+q)}, \quad a^{j}\in \cA, \ q\geq 1.
          \label{eq:CM.cocycle.residue.>0}
        \end{gather}
    
            \item  The cocycle  $\varphi^{\CM}$ represents the Connes-Chern character $\Ch(D)$ in $\op{HP}^{0}(\cA)$.
    
        \item For any Hermitian finitely projective module $\cE$ over $\cA$ and any Hermitian connection $\nabla^{\cE}$ 
         over $\cE$, we have
         \begin{equation*}
             \ind D_{\nabla^{\cE}}=\acou{\varphi^{\CM}}{\Ch(\cE)}.
         \end{equation*}
    \end{enumerate}
  \end{theorem}
  
\begin{remark}
The formulas~(\ref{eq:CM.cocycle.residue.0})--(\ref{eq:CM.cocycle.residue.>0}) provide us with the local index formula in noncommutative geometry~\cite{CM:LIFNCG}.  
The cocycle  $\varphi^{\CM}$  is called the CM cocycle of the spectral triple $(\cA,\cH,D)$. 
\end{remark}

\begin{remark}
 Connes-Moscovici~\cite{CM:LIFNCG} proved Theorem~\ref{thm:CM.cocyle.residue.formula} under the additional assumption that the functions 
 $\Gamma(z)\Str[ X|D|^{-z}]$, $X\in \cD_{D}^{\bt}(\cA)$, have rapid decay along vertical lines in the complex plane. 
 This technical assumption is removed in~\cite{Hi:RITCM}. 
\end{remark}

\begin{example}
A Dirac spectral triple $( C^{\infty}(M), L^{2}_{g}(M,\sS), \sD_{g})$ is $n^{+}$-summable with $n=\dim M$. 
It is also regular and has a discrete dimension spectrum contained in $\{k\in \N; \ k \leq n\}$. 
Each space $\Psi^{q}_{D}(\cA)$, $q \in \C$, is contained in the space of classical \psidos\ of order $q$. In addition, 
the residual trace $\bint$ agrees with the noncommutative residue trace of Guillemin and Wodzicki. Finally 
(see~{\cite[Remark~II.1]{CM:LIFNCG}}, \cite{Po:CMP}) the CM cocycle $\varphi^{\CM}=(\varphi_{2q}^{\CM})_{q\geq 0}$ is given by
            \begin{equation*}
                \varphi_{2q}^{\CM}(f^{0},\ldots,f^{2q})= \frac{(2i\pi)^{-\frac{n}{2}}}{(2q)!}\int_{M} f^{0}df^{1}\wedge \cdots \wedge 
                 df^{2q}\wedge \hat{A}\left(R^{TM}\right), \quad f^{j}\in C^{\infty}(M).
            \end{equation*}Combining this with~(\ref{eq:index.coupled.D.pairing.Conne-Chern.character}) enables us to recover the local  index 
            formula  of Atiyah-Singer~\cite{AS:IEO1, AS:IEO3}. 
\end{example}
  
\subsection{The CM cocycle of a smooth spectral triple} 
We shall now explain how to specialize the framework of the local index formula in noncommutative geometry to spectral 
triples over locally convex algebras, more precisely, for the \emph{smooth spectral triples} considered in Part~I. 

In what follows by a  locally convex $*$-algebra  we shall mean a $*$-algebra $\cA$ equipped with a locally convex 
topology with respect to which its product and involution are continuous maps. A spectral triple $(\cA,\cH, D)$ over 
such an algebra is called \emph{smooth} when the representation of $\cA$ and the map $a\rightarrow [D,a]$ are 
continuous linear maps from $\cA$ to $\cL(\cH)$. If we further assume $p$-summability, then the Connes-Chern character 
of $(\cA,\cH,D)$ descends to a cohomology class
\begin{equation*}
    \bCh(D)\in \op{\mathbf{HP}}^{0}(\cA),
\end{equation*}where $ \op{\mathbf{HP}}^{\bt}(\cA)$ is the periodic cyclic cohomology of \emph{continuous} cochains on 
$\cA$ (\emph{cf.}~Part~I). More precisely, the cyclic cocycles $\overline{\tau}_{2q}^{D}$, $q \geq \frac{1}{2}(p-1)$, 
are continuous and define the same class in $\op{\mathbf{HP}}^{0}(\cA)$. It then is natural to ask under which conditions the CM cocycle may represent the Connes-Chern character in 
$\op{\mathbf{HP}}^{0}(\cA)$. 

In what follows we let $(\cA,\cH,D)$ be a smooth spectral triple which is $p^{+}$-summable and hypo-regular.  

\begin{definition}
    We say that $(\cA,\cH,D)$ is uniformly hypo-regular when, for all $m\in \N_{0}$, the linear maps 
    $a\rightarrow D^{m}aD^{-m}$ and $a\rightarrow D^{m}[D,a]D^{-m}$ are continuous from $\cA$ to $\cL(\cH)$. 
\end{definition}

We observe that the hypo-regularity assumption enables us to endow each space $\cD^{m}_{D}(\cA)$, $m\geq 0$, with the norm,
\begin{equation*}
    \|X\|_{(m)}=\|X(1+D^{2})^{-m/2}\|, \qquad X\in \cD^{m}_{D}(\cA).
\end{equation*}
This gives rise to a natural normed topology on each space $\cD^{m}_{D}(\cA)$, $m\geq 0$. 

In what follows, given any open $\Omega\subset \C$, 
we shall denote by $\Hol(\Omega)$ the Fr\'echet space of holomorphic functions on $\Omega$. 

\begin{remark}
\label{rem:Str.P|D|^-z-m.holo}
If $(\cA,\cH,D)$ is uniformly hypo-regular,  then, for any pair $(\alpha,\beta)\in \cP(m)$, $m\in \N_{0}$, the map $(a^{0},\ldots,a^{m})\rightarrow P_{\alpha,\beta}(a^{0},\ldots,a^{m})$ 
is a continuous $(m+1)$-linear map from $\cA^{m+1}$ to $\cD^{m}_{D}(\cA)$, where 
the operator $P_{\alpha,\beta}(a^{0},\ldots,a^{m})$ is given by~(\ref{eq:definition.P_alphabeta}). It then follows that we obtain a continuous $(m+1)$-linear map,
\begin{equation*}
 \cA^{m+1}\ni (a^{0},\ldots,a^{m})\longrightarrow \Str\left[P_{\alpha,\beta}(a^{0},\ldots,a^{m})|D|^{-z-m}\right]\in \op{Hol}(\Re z>p).
\end{equation*}
\end{remark}

The above remark  leads us to the following notion of uniform dimension spectrum. 

\begin{definition}\label{def:CM.uniform-dimension-spectrum}
    We say that $(\cA,\cH,D)$ has a simple and discrete uniform dimension spectrum when it is uniformly hypo-regular 
    and has a simple and discrete dimension spectrum in such way that, for any $m\in \N_{0}$ and pair $(\alpha,\beta)\in \cP(m)$, 
    the following conditions are satisfied: 
\begin{enumerate}
\item[(i)] The $(m+1)$-linear map $(a^{0},\ldots,a^{m})\rightarrow \Str\left[P_{\alpha,\beta}(a^{0},\ldots,a^{m})|D|^{-z-m}\right]$ is continuous from 
$\cA^{m+1}$ to $\Hol(\C\setminus \Sigma)$.

\item[(ii)] Any $\sigma \in \Sigma$ has an open neighbourhood $\Omega \subset \C$ such that $\Sigma \cap \Omega=\{\sigma\}$ and the 
$(m+1)$-linear map $(a^{0},\ldots,a^{m})\rightarrow (z-\sigma)\Str\left[P_{\alpha,\beta}(a^{0},\ldots,a^{m})|D|^{-z-m}\right]$ is continuous from 
$\cA^{m+1}$ to $\Hol(\Omega)$. 
\end{enumerate}
\end{definition}

\begin{definition}
 We say that $(\cA,\cH,D)$ is uniformly regular when it is regular and, for all $m\in \N_{0}$, the linear maps 
 $a\rightarrow a^{[m]}D^{-m}$ and $a\rightarrow [D,a]^{[m]}D^{-m}$ are continuous from $\cA$ to $\cL(\cH)$.   
\end{definition}

We are now in a position to answer the question about the representation of the Connes-Chern character $\bCh(D)\in \bHP^0(\cA)$ 
by means of the CM cocycle. 

\begin{proposition}[\cite{Po:SmoothCM}]
\label{prop:CM.cocycle.HP.conditions}
    Suppose that $(\cA,\cH,D)$ is smooth, $p^{+}$-summable, uniformly regular and has a discrete and simple uniform dimension 
    spectrum. Then 
    \begin{enumerate}
        \item  The components~(\ref{eq:CM.cocycle.residue.0})--(\ref{eq:CM.cocycle.residue.>0}) of the CM cocycle $\varphi^{\CM}$ are continuous cochains on $\cA$.
    
        \item  The class of $\varphi^{\CM}$ in $\op{\mathbf{HP}}^{0}(\cA)$ agrees with the Connes-Chern character 
        $\bCh(D)$. 
    \end{enumerate}
\end{proposition}
 
\section{CM Cocycle and Heat-Trace Asymptotics}
\label{sec:CM.cocycle.HT.Asym}
 In this section, we shall now re-interpret the CM cocycle and its representation of the Connes-Chern character in terms of 
heat-trace asymptotics. As we shall see, this is especially convenient for smooth spectral triples over barelled locally 
convex algebras. 

In what follows we let $(\cA,\cH,D)$ be a hypo-regular $p^{+}$-summable spectral triple. We denote by $\cL^{1}(\cH)$ the Banach ideal of 
trace-class operators on $\cH$ equipped with the norm $\|T\|_{1}:=\Tr |T|$, $T\in \cL^{1}(\cH)$. In addition, we denote by $\Pi_{0}$ the 
orthogonal projection onto $\ker D$. Note that $\Pi_{0}$ is a finite-rank operator whose range is contained in 
$\cH^{\infty}$.  

\begin{lemma}
\label{lem:estimation.involving.heat.operator}
    Assume that $(\cA,\cH,D)$ is hypo-regular, and let $m\in \N_{0}$. 
    \begin{enumerate}
        \item  For all $X\in \cD^{m}_{D}(\cA)$, the operators $Xe^{-tD^{2}}$, $t>0$, form a continuous family of 
        trace-class operators. 
    
        \item  For all $q>p$,  there is a constant $C_{mq}>0$ such that
        \begin{equation}
        \label{eq:X.heat.operator.estimate}
            \|Xe^{-tD^{2}}\|_{1}\leq C_{mq}\|X\|_{(m)}t^{-\frac{m+q}{2}}\qquad \text{for all $t>0$ and $X\in 
            \cD^{m}_{D}(\cA)$}.
        \end{equation}
    
        \item  Let $\lambda_{0}$ be the smallest eigenvalue of $D^{2}$. For all $q>p$, there are a time $t_{q}>0$ and a constant $C_{mq}>0$, such that
        \begin{equation}
        \label{eq:X.heat.operator.estimate.with.projection.kernel}
            \|X(1-\Pi_{0})e^{-tD^{2}}\|_{1}\leq C_{mq}\|X\|_{(m)}e^{-t\lambda_{0}} \quad \text{for all 
            $t\geq t_{q}$ and $X\in 
            \cD^{m}_{D}(\cA)$}. 
        \end{equation}
    \end{enumerate}
\end{lemma}
\begin{proof}
    The proof relies on the fact that $\cL^{1}(\cH)$ is a two-sided ideal and its norm is symmetric, i.e., 
    \begin{equation}
    \label{eq:Sobolev.1.norm.on.3.operators.estimate}
        \|A_{1}TA_{2}\|_{1}\leq \|A_{1}\| \|T\|_{1} \|A_{2}\| \qquad \forall T\in \cL^{1}(\cH) \ \forall A_{j}\in 
        \cL(\cH).
    \end{equation}
Let $q>p$ and $X\in \cD^{m}_{D}(\cA)$. For all $t>0$, we have
\begin{align}
\label{eq:X.heat.operator.sum}
    Xe^{-tD^{2}}&=X(1-\Pi_{0})e^{-tD^{2}}+X\Pi_{0} \nonumber\\ 
    & =t^{-\frac{m+q}{2}}(XD^{-m})D^{-q}(tD^{2})^{\frac{m+q}{2}}(1-\Pi_{0})e^{-tD^{2}}+X\Pi_{0}.
\end{align}
We note that $X\Pi_{0}$ has finite-rank, and hence is trace-class. In addition, the operators $XD^{-m}$, 
$(tD^{2})^{\frac{m+q}{2}}e^{-tD^{2}}$ and $D^{-q}$ are trace-class. Therefore, we see that $Xe^{-tD^{2}}$ is 
trace-class for all $t>0$. 

Using~(\ref{eq:Sobolev.1.norm.on.3.operators.estimate}) and the fact that $\Pi_{0}=(1+D^{2})^{-\frac{m}{2}}\Pi_{0}$ we get
\begin{equation}
\label{eq:X.Pi_0.inequality}
  \|X\Pi_{0}\|_{1}=\|X(1+D^{2})^{-\frac{m}{2}}\Pi_{0}\|_{1}\leq \|X\|_{(m)}\|\Pi_{0}\|_{1}.   
\end{equation}
Using~(\ref{eq:Sobolev.1.norm.on.3.operators.estimate}) and~(\ref{eq:X.heat.operator.sum}) we also obtain
\begin{equation}
        \label{eq:X.heat.operator.projection.kerne.inequality}
    \|X(1-\Pi_{0})e^{-tD^{2}}\|_{1}\leq t^{-\frac{m+q}{2}}\|XD^{-m}\| 
    \|D^{-q}\|_{1}\|(tD^{2})^{\frac{m+q}{2}}(1-\Pi_{0})e^{-tD^{2}}\|. 
\end{equation}
We observe that
\begin{gather}
\label{eq:X.D^-m.inequlity}
 \|XD^{-m}\|=\|X(1+D^{2})^{-\frac{m}{2}}\cdot (1+D^{2})^{\frac{m}{2}}D^{-m}\|\leq 
 \|X\|_{(m)}\|(1+D^{2})^{\frac{m}{2}}D^{-m}\| ,  \\
 \label{eq:inequality.further.estimate}
    \|(tD^{2})^{\frac{m+q}{2}}(1-\Pi_{0})e^{-tD^{2}}\| \leq \sup\{\mu^{\frac{m+q}{2}}e^{-\mu}; \mu\geq t\lambda_{0}\} 
    \leq  \sup\{\mu^{\frac{m+q}{2}}e^{-\mu}; \mu\geq 0\}.  
\end{gather}
Combining this with~(\ref{eq:X.Pi_0.inequality})--(\ref{eq:X.heat.operator.projection.kerne.inequality}) and the fact that $Xe^{-tD^{2}}=X(1-\Pi_{0})e^{-tD^{2}}+X\Pi_{0}$ 
we deduce there is a constant $C_{mq}>0$ such that, for all $t>0$ and $X\in 
\cD^{m}_{D}(\cA)$, we have
\begin{equation*}
    \|Xe^{-tD^{2}}\|_{1}\leq \|X(1-\Pi_{0})e^{-tD^{2}}\|_{1}+ \|X\Pi_{0}\|_{1} \leq 
    C_{mq}(t^{-\frac{m+q}{2}}+1)\|X\|_{(m)}.
\end{equation*}
We further observe that the function $\mu\rightarrow \mu^{\frac{m+q}{2}}e^{-\mu}$ has a single critical point $\mu_{q}$ 
on $(0,\infty)$ and is decreasing on $[\mu_{d},\infty)$. Therefore, if we set $t_{q}=\mu_{q}\lambda_{0}^{-1}$, then, 
for all $t\geq t_{q}$, the function $\mu\rightarrow \mu^{\frac{m+q}{2}}e^{-\mu}$ is decreasing on 
$[t\lambda_{0},\infty)$, and hence $\mu^{\frac{m+q}{2}}e^{-\mu}\leq (t\lambda_{0})^{\frac{m+q}{2}}e^{-t\lambda_{0}}$ 
for all $\mu\geq t\lambda_{0}$.  Combining this with~(\ref{eq:X.heat.operator.projection.kerne.inequality})--(\ref{eq:inequality.further.estimate}) we deduce there is a constant 
$C_{mq}'>0$ such that
\begin{equation*}
            \|X(1-\Pi_{0})e^{-tD^{2}}\|_{1}\leq C_{mq}'\|X\|_{(m)}e^{-t\lambda_{0}} \quad \text{for all 
            $t\geq t_{q}$ and $X\in 
            \cD^{m}_{D}(\cA)$}. 
\end{equation*}

To complete the proof it remains to show that, for any $X\in \cD_{D}^{m}(\cA)$, the family  $\left(Xe^{-tD^{2}}\right)_{t>0}$ is a continuous family in 
$\cL^{1}(\cH)$. To this end let $c>0$ and $t_{j}\in (c,\infty)$, $j=1,2$. Set $t=\min\{t_{1},t_{2}\}$ and 
$h=|t_{1}-t_{2}|$. In addition, let $f(x)$ be the function on 
$[0,\infty)$ defined by $f(x)=x^{-1}(1-e^{-x})$ for $x>0$ and $f(0)=1$. Note that $f(x)$ is a continuous 
function on $[0,\infty)$ with values in $(0,1]$. We then have
\begin{equation}
\label{eq:X.heat.operator.difference.in.t}
    Xe^{-t_{1}D^{2}}-Xe^{-t_{2}D^{2}}=\pm Xe^{-tD^{2}}(1-e^{-hD^{2}})=hXD^{2}e^{-cD^{2}}e^{-(t-c)D^{2}}f(hD^{2}).
\end{equation}
As $XD^{2}$ is contained in $\cD^{m+2}_{D}(\cA)$, the first part of the proof shows that 
$XD^{2}e^{-cD^{2}}$ is a trace-class operator. We also note that $\|e^{-(t-c)D^{2}}f(hD^{2})\|\leq \|e^{-(t-c)D^{2}}\| 
\|f(hD^{2})\| \leq \max f\leq 1$. Therefore, combining~(\ref{eq:Sobolev.1.norm.on.3.operators.estimate}) and~(\ref{eq:X.heat.operator.difference.in.t}) we obtain
\begin{equation*}
  \|Xe^{-t_{1}D^{2}}-Xe^{-t_{2}D^{2}}\|_{1}\leq h \|XD^{2}e^{-cD^{2}}\|_{1}\qquad \forall t_{j}\in (c,\infty).   
\end{equation*}
This proves that $\left(Xe^{-tD^{2}}\right)_{t>0}$ is a continuous family in 
$\cL^{1}(\cH)$. The proof is complete. 
\end{proof}

\begin{remark}
It follows from Lemma~\ref{lem:estimation.involving.heat.operator} that, for any $X\in \cD_{D}^{\bt}(\cA)$, the supertrace $\Str \left[ X e^{-tD^{2}}\right]$ is well defined 
for all $t>0$. Moreover, the inequalities~(\ref{eq:X.heat.operator.estimate})--(\ref{eq:X.heat.operator.estimate.with.projection.kernel}) imply that, for all $m\in \N_0$, we actually have 
continuous linear maps, 
\begin{gather*}
 \cD_D^m(\cA)\ni X \longrightarrow t^{\frac{m+q}{2}} \Str \left[ X e^{-tD^{2}}\right] \in C^0_b(0,1], \qquad q>p,\\
 \cD_D^m(\cA)\ni X \longrightarrow e^{t\lambda} \Str \left[ X (1-\Pi_0)e^{-tD^{2}}\right] \in C^0_b[1,\infty),
\end{gather*}
where $C^0_b(0,1]$ (resp., $C^0_b[1,\infty)$) is the Banach space of bounded continuous functions on $(0,1]$ (resp., $[1,\infty)$). 
\end{remark}

\begin{definition} Given a discrete subset $\Sigma$ of $(-\infty, p]$, we shall say that $(\cA,\cH,D)$ has
 the asymptotic expansion property relatively to $\Sigma$ when it is hypo-regular and, for any $X\in 
    \cD^{m}_{D}(\cA)$, $m\geq 0$, there is an asymptotic expansion of the form,
    \begin{equation}
    \label{eq:X.heat.operator.AE}
      \Str\left[ Xe^{-tD^{2}}\right]\sim \sum_{\sigma\in \Sigma} a_{\sigma}^{(m)}(X)t^{-\frac{1}{2}(\sigma+m)}\qquad \text{as 
      $t\rightarrow 0^{+}$}.
    \end{equation}
\end{definition}

\begin{remark}
 Let $\theta(t)$ be a function on $(0,\infty)$ which as $t\rightarrow 0^{+}$ admits an asymptotic expansion of the 
 form~(\ref{eq:X.heat.operator.AE}). We shall call by \emph{partie finie} (i.e., finite part),  and denote by $\FP \theta_X(t)$, the constant 
 coefficient in this asymptotic expansion. For instance, in the notation of~(\ref{eq:X.heat.operator.AE}), for any $X\in \cD^{m}_{D}(\cA)$, 
 $m\geq 0$, and $\sigma \in \Sigma$, we have 
    \begin{align}
       a_{\sigma}^{(m)}(X) &= \FP \left\{ t^{\frac{1}{2}(\sigma+m)} \Str\left[ Xe^{-tD^{2}}\right]\right\} \label{eq:AE.coefficient.FP}\\ 
       &= \lim_{t\rightarrow 
       0^{+}}t^{\frac{1}{2}(\sigma+m)}\biggl\{  
       \Str\left[ Xe^{-tD^{2}}\right]-  \sum_{\substack {\sigma' \in \Sigma\\ \sigma'>\sigma}} 
       a_{\sigma'}^{(m)}(X) t^{-\frac{1}{2}(\sigma'+m)}\biggl\} \label{eq:AE.coefficient.lim.diff}. 
    \end{align}
    Note also that both right-hand sides make sense when $\sigma\not \in \Sigma$. 
\end{remark}

\begin{remark}
 The asymptotic expansion property implies that, for any $m\in \N_{0}$, we define a linear map $R_N:\cD^m_D(\cA,D)\rightarrow C_b^0(0,1]$ by
\begin{equation}
 R_{N}(X)(t)=t^{-N}\biggl\{\Str\left[ Xe^{-tD^{2}}\right]- \sum_{\sigma \in 
        \Sigma_{N}^{(m)}} a_{\sigma}^{(m)}(X) t^{-\frac{1}{2}(\sigma+m)} \biggr\}, \quad X\in \cD_D^m(\cA), \ t\in (0,1],
        \label{eq:Heat-CM.RNXt}
\end{equation}where we have set $\Sigma_{N}^{(m)}:=\left\{\sigma\in \Sigma; \ \sigma+m>2N\right\}$. 
\end{remark}

\begin{lemma}
   \label{lemma:Residue.formula.relation.heat.operator}
   Assume that $(\cA,\cH,D)$ has the asymptotic expansion property relatively to $\Sigma$. Then
   \begin{enumerate}
       \item  $(\cA,\cH,D)$ has a discrete and simple dimension spectrum contained in $\Sigma$.
   
       \item  For any $X\in \cD_{D}^{\bt}(\cA)$ and $q\geq 0$, we have
       \begin{equation}
          \underset{z=0}{\Res}\left\{\Gamma(z)\Str \left[ X|D|^{-2(z+q)}\right]\right\}= \left\{
           \begin{array}{ll}
	 \FP t^{q}\Str\left[ Xe^{-tD^{2}}\right]   &  \text{if $q>0$},\\    
             \FP \Str\left[ Xe^{-tD^{2}}\right]  - \Str\left[X\Pi_{0}\right]&  \text{if $q=0$}. 
           \end{array}\right.
           \label{eq:Residue.formula.relation.heat.operator}
       \end{equation}
   \end{enumerate}
\end{lemma}
\begin{proof}
   Let $X\in \cD^{m}_{D}(\cA)$, $m \geq 0$. In the following, given any $s\in \C$ we set $\hat{s}=\frac{1}{2}(s+m)$. 
   Using the Mellin formula 
   $\Gamma(\hat{z})|D|^{-(z+m)}=\int_{0}^{\infty}t^{\hat{z}-1}(1-\Pi_{0})e^{-tD^{2}}dt$, $\Re \hat{z}>0$, and the 
   boundedness of $X|D|^{-m}$, it can be shown that, for $\Re z>0$,  we have
   \begin{equation*}
       \Gamma(\hat{z})X|D|^{-(z+m)}=\int_{0}^{\infty}t^{\hat{z}-1}X(1-\Pi_{0})e^{-tD^{2}}dt,
   \end{equation*}where the integral converges in $\cL(\cH)$. In fact, Lemma~\ref{lem:estimation.involving.heat.operator} further ensures us that, for $\Re z 
   >p$ the operators $t^{\hat{z}}X(1-\Pi_{0})e^{-tD^{2}}$, $t>0$, form a continuous family in $\cL^{1}(\cH)$ which 
   remains bounded as $t\rightarrow 0^{+}$ and decay exponentially fast as $t\rightarrow \infty$. Therefore, for $\Re z>p$, we have 
   \begin{equation}
   \label{eq:Str.X.coomplex.power.D.relation.heat.operator}
       \Gamma(\hat{z})\Str \left[X|D|^{-(z+m)}\right]=\int_{0}^{\infty}t^{\hat{z}-1}\Str 
       \left[X(1-\Pi_{0})e^{-tD^{2}}\right]dt.  
   \end{equation}
   That is, when regarded as a function of the variable $\hat{z}$, the function $\Gamma(\hat{z})\Str 
   \left[X|D|^{-(z+m)}\right]$ is the Mellin transform of the function, 
   \begin{equation*}
       \theta_{X}(t)=\Str \left[X(1-\Pi_{0})e^{-tD^{2}}\right]= \Str \left[Xe^{-tD^{2}}\right] - \Str \left[X\Pi_{0}\right], \qquad 
       t>0.
   \end{equation*}
   
   It is well known how to relate the short-time behavior of $ \theta_{X}(t)$ to the meromorphic singularities of its 
   Mellin transform (see, e.g., \cite{BGV:HKDO, GS:OCCTSM}). First, the fact that $t^{\hat{z}}X(1-\Pi_{0})e^{-tD^{2}}$, 
   $t\geq 1$, is a continuous family with rapid decay in $\cL^{1}(\cH)$ enables us to define an entire function
   $F(X)(z):=\int_{1}^{\infty}t^{\hat{z}-1}\theta_{X}(t)dt$, $z\in \C$. We then may rewrite~(\ref{eq:Str.X.coomplex.power.D.relation.heat.operator}) as
   \begin{equation}
      \label{eq:Str.X.coomplex.power.D.relation.term.of.AE}
      \Gamma(\hat{z})\Str \left[X|D|^{-(z+m)}\right]= \int_{1}^{\infty}t^{\hat{z}-1}\theta_{X}(t)dt , \qquad \Re z>p.
   \end{equation}
  Second, the asymptotic expansion~(\ref{eq:X.heat.operator.AE}) implies that, for all $N\in \N$, we can write
   \begin{equation}
   \label{eq:theta.sum.three.terms}
     \theta_{X}(t)= \sum_{\sigma \in \Sigma_{N}^{(m)}}t^{-\hat{\sigma}}a_{\sigma}^{(m)}(X)- \Str\left[X\Pi_{0}\right] 
     +t^{N}R_{N}(X)(t), \qquad 0<t\leq 1,  
   \end{equation}where $R_{N}(X)(t)$ is the bounded continuous function on $(0,1]$ defined by~(\ref{eq:Heat-CM.RNXt}). The boundedness of $R_{N}(X)(t)$ enables us
  to define a holomorphic function $H_{N}(X)(z):=\int_{0}^{1}t^{\hat{z}+N-1}R_{N}(X)(t)dt$ on the half-plane $\Re \hat{z}>-N$ (i.e., the 
   half-plane $\Re z >-(2N+m)$). Combining this with~(\ref{eq:Str.X.coomplex.power.D.relation.term.of.AE}) and~(\ref{eq:theta.sum.three.terms}) we see that  on the half-plane $\Re z >p$ the 
   function $\Gamma(\hat{z})\Str  \left[X|D|^{-(z+m)}\right]$ is equal to 
    \begin{multline*}
        \sum_{\sigma \in 
        \Sigma_{N}^{(m)}}\int_{1}^{\infty}t^{\frac{1}{2}(z-\sigma)-1}a_{\sigma}^{(m)}(X)dt- 
        \int_{1}^{\infty}t^{\frac{1}{2}(z+m)}\Str\left[X\Pi_{0}\right]dt +  H_{N}(X)(z) +F(X)(z)  \\
        =   \sum_{\sigma \in  \Sigma_{N}^{(m)}} \frac{2}{z-\sigma}a_{\sigma}(X) - \frac{2}{z-m}\Str\left[X\Pi_{0}\right] + 
         H_{N}(X)(z) +F(X)(z). 
    \end{multline*}
   This shows that, for all $N\in \N$, the function $ \Gamma(\hat{z})\Str \left[X|D|^{-(z+m)}\right] $ has a 
   meromorphic continuation to the half-plane $\Re z >-(2N+m)$ with at worst simple pole singularities on 
   $\Sigma_{N}^{(m)}\cup\{-m\}$. Furthermore, for all $\sigma\in\Sigma_{N}^{(m)}$ with $\sigma\neq -m$, we have
   \begin{equation}
   \label{eq:residue.not.at.-m}
     \underset{z=\sigma}{\Res} \left\{ \Gamma(\hat{z})\Str \left[X|D|^{-(z+m)}\right]\right\} =  2a_{\sigma}^{(m)}(X)= 
         2 \FP \left\{
         t^{\hat{\sigma}}\Str \left[X(1-\Pi_{0})e^{-tD^{2}}\right]\right\},  
   \end{equation}while for $\sigma=-m$ we obtain  
    \begin{equation}
    \label{eq:residue.at.m}
        \underset{z=-m}{\Res} \left\{ \Gamma(\hat{z})\Str \left[X|D|^{-(z+m)}\right]\right\} = 2 \FP 
        \Str \left[X(1-\Pi_{0})e^{-tD^{2}}\right] - 2\Str\left[X\Pi_{0}\right].        
    \end{equation} 
 As $\Gamma(\hat{z})^{-1}=\Gamma\left( \frac{1}{2}(z+m)\right)^{-1}$ is an entire function that vanishes on 
 $-m-2\N_{0}$, we then deduce that the function $\Str \left[X|D|^{-(z+m)}\right]$ has a meromorphic 
 extension to $\C$ with at worst simple pole singularities on $\bigcup_{N\geq 1} \Sigma_{N}^{(m)}=\Sigma$. This proves that 
 $(\cA,\cH,D)$ has a simple and discrete dimension spectrum contained in $\Sigma$. Finally, the 
 formula~(\ref{eq:Residue.formula.relation.heat.operator}) follows from~(\ref{eq:residue.not.at.-m})--(\ref{eq:residue.at.m}) and the fact that, for $q>0$ both sides of~(\ref{eq:Residue.formula.relation.heat.operator})  vanish when 
 $q \not \in \frac{1}{2}\left( \Sigma+m\right)$. The proof is complete.  
\end{proof}

Combining Lemma~\ref{lemma:Residue.formula.relation.heat.operator} with Theorem~\ref{thm:CM.cocyle.residue.formula} we then arrive at the following result. 

\begin{proposition}[\cite{Po:SmoothCM}]
\label{prop:CM.heat.trace.estimate}
    Assume that $(\cA,\cH,D)$ is $p^{+}$-summable, regular and has the asymptotic property. Then
    \begin{enumerate}
        \item  $(\cA,\cH,D)$ has a discrete and simple dimension spectrum, and hence the Connes-Chern character 
        $\Ch(D)$ is represented by the CM cocycle~(\ref{eq:CM.cocycle.residue.0})--(\ref{eq:CM.cocycle.residue.>0}). 
    
        \item  We have the following formulas for the components $\varphi_{2q}^{\CM}$, $q\geq 0$, of the CM cocycle:  
    \begin{gather}
    \label{eq:CM.cocycle.heat.operator.0}
         \varphi_{0}^{\CM}(a^{0})= \FP \Str\left[ a^{0}e^{-tD^{2}}\right], \qquad a^{0}\in \cA,\\
        \varphi_{2q}^{\CM}(a^{0},\ldots,a^{2q})= \sum_{\alpha \in \N_{0}^{2q}}c_{q,\alpha}\Gamma(|\alpha|+q)^{-1} 
        \FP \left\{t^{|\alpha|+q}\Str\left[ T_{q,\alpha}e^{-tD^{2}}\right]\right\}, \quad a^{j}\in \cA, \ q\geq 1, 
         \label{eq:CM.cocycle.heat.operator.>0}
    \end{gather}where $c_{q,\alpha}$ is given by~(\ref{eq:constrant.in.CM.cocycle}) and we have set 
    $T_{q,\alpha}=a^{0}[D,a^{1}]^{[\alpha_{1}]}\cdots [D,a^{2q}]^{[\alpha_{2q}]}$. 
    \end{enumerate}
\end{proposition}

In the rest of this section we further assume that $(\cA,\cH,D)$ is smooth (in addition to be $p^{+}$-summable and 
hypo-regular). 

\begin{definition}
\label{def:UAE.of.ST}
    Given a discrete subset $\Sigma\subset \R$, we say that $(\cA,\cH,D)$ has the uniform asymptotic expansion property 
    relatively to $\Sigma$ when $(\cA,\cH,D)$ is uniformly hypo-regular and it has the asymptotic expansion property relatively to $\Sigma$ in such way that,
    for any $m\in \N_{0}$ and pair $(\alpha,\beta)\in \cP(m)$,  the following properties are satisfied:
    \begin{enumerate}
\item[(i)] For all $\sigma \in \Sigma$, the $(m+1)$-linear form $(a^{0},\ldots,a^{m})\rightarrow 
   a^{(m)}_{\sigma}\left(P_{\alpha,\beta}(a^{0},\ldots,a^{m})\right)$ is continuous on $\cA^{m+1}$.
   
\item[(ii)] For all $N\in \N_0$, the  $(m+1)$-linear map $(a^0,...,a^m)
 \longrightarrow R_N\left(P_{\alpha\beta}(a^0,\cdots, a^m)\right)(t)$ is continuous from $\cA^{m+1}$ to $C^0_b(0,1]$ (where $R_N$ is defined as in~(\ref{eq:Heat-CM.RNXt})). 
   \end{enumerate}
\end{definition}

\begin{remark}\label{rem:Str.P.heat.operator.AE.meaning}
The condition (i) means there is a continuous seminorm $\cN^{(m)}_{\sigma \alpha \beta}$ on $\cA$ such that
\begin{equation}
\label{eq:a^m_sigmaX.estimate}
 \left| a^{(m)}_{\sigma}\left(P_{\alpha,\beta}(a^{0},\ldots,a^{m})\right)\right| \leq \cN^{(m)}_{\sigma \alpha \beta}(a^0)\cdots \cN^{(m)}_{\sigma \alpha \beta}(a^m) 
   \qquad \text{for all $a^{j}\in  \cA$}. 
\end{equation}
   Likewise, the condition (ii) means there is a continuous 
   semi-norm $\cN^{(m)}_{N\alpha\beta}$ on $\cA$ such that
   \begin{equation}
       \label{eq:Str.P.heat.operator.AE.meaning}
        \left| R_N\left( P_{\alpha\beta}(a^0,\cdots, a^m)\right)(t) \right| 
        \leq \cN^{(m)}_{N\alpha\beta}(a^{0})\cdots \cN^{(m)}_{N\alpha\beta}(a^{m})
         \qquad \text{for all $a^{j}\in  \cA$ and $t\in (0,1]$}.  
   \end{equation}
\end{remark}

As the following lemma shows, requiring uniformness in the heat trace asymptotics~(\ref{eq:X.heat.operator.AE}) turns out to be irrelevant when the $\cA$ has a 
barelled locally convex topology. We refer to~\cite{SW:TVS} for the precise definition of a barelled locally convex 
topology. For our purpose it is enough to know that Banach-Steinhaus theorem continues to hold for barelled topological 
vector spaces and main examples of barelled topological spaces include Baire topological vector spaces,  as well as 
inductive limits of such spaces (\emph{cf.}~\cite{SW:TVS}). In particular, Fr\'echet spaces and inductive limits of 
Fr\'echet spaces are barelled locally convex spaces.   

\begin{lemma}
\label{lem:UAE.rel.AE}
 Assume that the topology of $\cA$ is barelled,  and let $\Sigma$ be a discrete subset of $(-\infty, p]$. Then the following are equivalent:
 \begin{enumerate}
     \item  $(\cA,\cH,D)$ has the uniform asymptotic expansion property 
    relatively to $\Sigma$.
 
     \item  $(\cA,\cH,D)$ is uniformly hypo-regular and has the asymptotic expansion property 
    relatively to $\Sigma$.
 \end{enumerate}
\end{lemma}
\begin{proof}
    It is immediate that (1) implies (2), so we only need to prove the converse. Assume that $(\cA,\cH,D)$ is 
    uniformly hypo-regular and has the asymptotic expansion property relatively to $\Sigma$. Let $m\in \N_{0}$ and 
    $(\alpha,\beta)\in \cP(m)$. By Remark~\ref{rem:Str.P|D|^-z-m.holo}, the $(m+1)$-linear map 
    $(a^{0},\ldots,a^{m})\rightarrow P_{\alpha,\beta}(a^{0},\ldots,a^{m})$ is continuous from $\cA^{m+1}$ to 
    $\cD^{m}_{D}(\cA)$. Let $t>0$. As the estimate~(\ref{eq:X.heat.operator.estimate}) implies that the linear form $X\rightarrow 
    \Str\left[Xe^{-tD^{2}}\right]$ is continuous on $\cD^{m}_{D}(\cA)$, we deduce that the $(m+1)$-linear form $(a^{0},\ldots,a^{m})\rightarrow  
    \Str\left[P_{\alpha, \beta}(a^{0},\ldots,a^{m})e^{-tD^{2}}\right]$ is continuous on $\cA^{m+1}$. 
    
    Bearing this in mind,  let us enumerate $\Sigma$ 
    as a decreasing sequence $\sigma_{0}>\sigma_{1}>\cdots$. Then by~(\ref{eq:AE.coefficient.FP}) we have 
    \begin{equation*}
        a^{(m)}_{\sigma_{0}}\left(P_{\alpha, \beta}(a^{0},\ldots,a^{m})\right) = \lim_{t\rightarrow 
        0^{+}}t^{\frac{1}{2}(\sigma_{0}+m)} \Str\left[P_{\alpha, \beta}(a^{0},\ldots,a^{m})e^{-tD^{2}}\right] \qquad 
        \forall a^{j}\in \cA.
    \end{equation*}
    Therefore, we see that $a^{(m)}_{\sigma_{0}}\left(P_{\alpha, \beta}(a^{0},\ldots,a^{m})\right)$ is the pointwise limit of 
    continuous $(m+1)$-linear forms on $\cA$. As the topology of $\cA$ is barrelled, the Banach-Steinhaus theorem 
    holds, and hence ensures us that $(m+1)$-linear form $(a^{0},\ldots,a^{m})\rightarrow 
   a^{(m)}_{\sigma_{0}}\left(P_{\alpha,\beta}(a^{0},\ldots,a^{m})\right)$ is continuous on $\cA^{m+1}$. More generally, 
   using~(\ref{eq:AE.coefficient.lim.diff}), an induction on $j$ and repeated use of the Banach-Steinhaus theorem show that, for all 
   $j=0,1,2,\ldots$, the $(m+1)$-linear form $(a^{0},\ldots,a^{m})\rightarrow 
   a^{(m)}_{\sigma_{j}}\left(P_{\alpha,\beta}(a^{0},\ldots,a^{m})\right)$ is continuous on $\cA^{m+1}$. That is, 
   condition~(i) of Definition~\ref{def:UAE.of.ST} is satisfied. 
   
   Bearing this in mind, let $N\in \N_{0}$. An examination of~(\ref{eq:Heat-CM.RNXt}) shows that, for all $t\in (0,1]$, the 
   $(m+1)$-linear form $(a^{0},\ldots,a^{m})\rightarrow 
   R_{N}\left(P_{\alpha,\beta}(a^{0},\ldots,a^{m})\right)(t)$ is continuous on $\cA^{m+1}$. Moreover, the asymptotic 
   expansion property implies that $\lim_{t\rightarrow 0^{+}}R_{N}\left(P_{\alpha,\beta}(a^{0},\ldots,a^{m})\right)(t)$ 
   exists for all $a^{j}\in \cA$. The Banach-Steinhaus theorem then implies that we obtain an equicontinuous 
   family of $(m+1)$-linear forms on $\cA$ parametrized by $t\in (0,1]$. That is, the estimate~(\ref{eq:Str.P.heat.operator.AE.meaning}) holds. 
   Therefore, using Remark~\ref{rem:Str.P.heat.operator.AE.meaning} we see that condition (ii) of Definition~\ref{def:UAE.of.ST} holds as well. This shows that $(\cA,\cH,D)$ has 
   the uniform asymptotic expansion property. The proof is complete. 
\end{proof}

The following lemma provides us with a relationship between uniform asymptotic property and uniform dimension spectrum.  
\begin{lemma}
\label{lem:UAE=>simple.discreteUDS}
    If $(\cA,\cH,D)$ has the uniform asymptotic expansion property, then it has a simple and discrete uniform dimension spectrum. 
\end{lemma}
\begin{proof}
  We know by Lemma~\ref{lemma:Residue.formula.relation.heat.operator} that $(\cA,\cH,D)$ has a discrete and simple dimension spectrum contained in $\Sigma$. Therefore, we 
  only need to show that the dimension spectrum is uniform. Given 
  $m\in \N_{0}$, the proof of Lemma~\ref{lemma:Residue.formula.relation.heat.operator} shows that, given any $X\in \cD^{m}_{D}(\cA)$ and $N\in \N$, 
  for  $\Re z>p$ we have
   \begin{equation}
   \label{eq:Str.X.complex.power.D.sum.four}
  \Gamma(\hat{z})\Str  \left[X|D|^{-(z+m)}\right] =   \sum_{\sigma \in  \Sigma_{N}^{(m)}} \frac{2}{z-\sigma}a_{\sigma}(X) - \frac{2}{z-m}\Str\left[X\Pi_{0}\right] + 
        H_{N}(X)(z) +F(X)(z), 
    \end{equation}
   where $F_{X}(z)$ and $H_{N}{(X)}(z)$ are given by
   \begin{equation*}
      F(X)(z)= \int_{1}^{\infty}t^{\hat{z}-1}\Str\left[X(1-\Pi_{0})e^{-tD^{2}}\right]dt \quad \text{and} \quad 
       H_{N}(X)(z) =\int_{0}^{1}t^{\hat{z}+N} R_N(X)(z)dt.
   \end{equation*}
Recall also that $F(X)(z)$ is en entire function and the function $H_{N}(X)(z)$ is holomorphic on the half-plane $\Re z>-(2N+m)$.

Let $\alpha\in \N_{0}^{m+2}$, $|\alpha|=m$, and $\beta \in\{0,1\}^{m}$. By assumption, for all $\sigma \in \Sigma$, the $(m+1)$-linear map
  $(a^{0},\ldots,a^{m})\rightarrow a_{\sigma}^{(m)}\!\left(P_{\alpha,\beta}(a^{0},\ldots,a^{m})\right)$ is continuous on $\cA^{m+1}$.  
Moreover, it follows from the seminorm estimate~(\ref{eq:Str.P.heat.operator.AE.meaning}) that $H_{N}\left(P_{\alpha,\beta}(a^{0},\ldots,a^{m})\right)(z)$ satisfies a uniform
estimate of the form~(\ref{eq:Str.P.heat.operator.AE.meaning}) on any closed halfspace $\Re z\geq -(2N+m)+\epsilon$, $\epsilon>0$. In addition, as 
pointed out in Remark~\ref{rem:Str.P|D|^-z-m.holo}, the $(m+1)$-linear map $(a^{0},\ldots,a^{m})\rightarrow P_{\alpha,\beta}(a^{0},\ldots,a^{m})$ is continuous from 
$\cA^{m+1}$ to $\cD^{m}_{D}(\cA)$. Combining this with the estimate~(\ref{eq:X.heat.operator.estimate.with.projection.kernel}) we deduce that 
$F\left(P_{\alpha,\beta}(a^{0},\ldots,a^{m})\right)(z)$ satisfies a uniform
estimate of the form~(\ref{eq:Str.P.heat.operator.AE.meaning}) on any closed vertical stripe $c_{1}\leq \Re z \leq c_{2}$, $c_{j}\in \R$. 
Combining these  observations with~(\ref{eq:Str.X.complex.power.D.sum.four}) we then deduce that the conditions (i)--(ii) of 
Definition~\ref{def:CM.uniform-dimension-spectrum} are satisfied on any halfspace $\Re z\geq -(2N+m)$, $N\in \N$. 
This shows that $(\cA,\cH,D)$ has a simple and discrete uniform dimension spectrum. The proof is complete. 
\end{proof}

Combining Lemmas~\ref{lem:UAE.rel.AE} and~\ref{lem:UAE=>simple.discreteUDS} with Proposition~\ref{prop:CM.cocycle.HP.conditions} we then arrive at the final  result of this section. 

\begin{proposition}[\cite{Po:SmoothCM}]
\label{prop:UAE=>CMcocyle.rep.CCC}
    Assume that $(\cA,\cH,D)$ is smooth, $p^{+}$-summable, uniformly regular, and one the following conditions holds:
    \begin{enumerate}
        \item[(i)]  $(\cA,\cH,D)$ has the uniform asymptotic expansion property.
    
        \item[(ii)] The topology of $\cA$ is barelled and $(\cA,\cH,D)$ has the asymptotic expansion property.
    \end{enumerate}
Then the CM cocycle represents the Connes-Chern character $\bCh(D)$ in $\op{\mathbf{HP}}^{0}(\cA)$ and its components are computed by the formulas~(\ref{eq:CM.cocycle.heat.operator.0})--(\ref{eq:CM.cocycle.heat.operator.>0}).    
\end{proposition}

\section{The Connes-Chern Character of an Equivariant Dirac Spectral Triple}\label{sec:Connes-Chern-conformal}
The aim of this section is to compute the Connes-Chern character of an equivariant Dirac spectral triple by means of its representation by the CM cocycle. By Connes-Chern character we shall 
mean its version as a class in the periodic cyclic cohomology of \emph{continuous} cochains. 

Throughout this section we let $(M^{n}, g)$ be an even dimensional compact spin oriented Riemannian manifold. We denote by 
$\sD_{g}$ its Dirac operator acting on the spinor bundle $\sS=\sS^{+}\oplus \sS^{-}$. In 
addition, we let $G$ be a subgroup of the connected component of the group of orientation-preserving smooth isometries preserving the spin structure. For $\phi 
\in G$ we denote by $U_{\phi}$ the unitary operator of $L^{2}_{g}(M,\sS)$ defined by~(\ref{eq:unitary.operator.on.L^2}) using the  unique lift of $\phi$ 
to a unitary vector bundle isomorphism $\phi^{\sS}:\sS\rightarrow \phi_{*}\sS$.  The map $\phi\rightarrow U_{\phi}$ 
then provides us with a unitary representation of $G$ in the Hilbert space $L^{2}_{g}(M,\sS)$. 

Equipping $G$ with its discrete topology,  the crossed-product algebra $C^{\infty}(M)\rtimes G$ is the space 
$C^{\infty}_{c}(M\times G)$ with the product and involution given by
\begin{equation*}
    F_{1}*F_{2}(x,\phi)=\sum_{\phi_{1}\circ \phi_{2}=\phi}F_{1}(x,\phi_{1})F_{2}(\phi^{-1}_{1}(x),\phi_{2}),  \qquad 
  F^{*}(x,\phi)=\overline{F(x,\phi^{-1})}.
\end{equation*}
Alternatively, if we denote by $u_{\phi}$ the characteristic function of $M\times\{\phi\}$. Then $u_{\phi}\in 
C^{\infty}_{c}(M\times G)$ and any $F\in C^{\infty}_{c}(M\times G)$ is uniquely written as a finitely supported sum,
\begin{equation*}
    F=\sum_{\phi \in G}f_{\phi}u_{\phi},
\end{equation*}
where $f_{\phi}(x):=F(x,\phi)\in C^{\infty}(M)$. Moreover, we have the relations, 
\begin{gather}
    u_{\phi_{1}}u_{\phi_{2}}=u_{\phi_{1}\circ \phi_{2}}, \qquad u_{\phi_{j}}^{*}=u_{\phi_{j}^{-1}}=u_{\phi_{j}}^{-1}, 
    \qquad \phi_{j}\in G, 
    \label{eq:group.function.properties1}\\
    u_{\phi}f=(f\circ \phi^{-1})u_{\phi}, \qquad f\in C^{\infty}(M), \ \phi \in G.
  \label{eq:group.function.properties2}   
\end{gather}
In addition we shall endow $C^{\infty}(M)\rtimes G$ with its standard  locally convex $*$-algebra topology. As 
mentioned in Part~I, this topology is obtained as the inductive limit of the topologies of the Fr\'echet spaces 
$C^{\infty}(M)\rtimes F$, where $F$ ranges over finite subsets of $G$. In particular, the topology of  
$C^{\infty}(M)\rtimes G$ is barelled. Moreover, given any topological vector space 
$X$,  a linear map $\Phi: C^{\infty}(M)\rtimes G\rightarrow X$ is continuous if and only if, for all $\phi \in G$, the map  $f\rightarrow \Phi(fu_{\phi})$ is a 
continuous linear map from $C^{\infty}(M)$ to $X$. 

We also observe that the relations~(\ref{eq:group.function.properties1})--(\ref{eq:group.function.properties2}) are satisfied by the operators 
$U_{\phi}$, $\phi\in G$, and the functions $f\in C^{\infty}(M)$ are
represented as multiplication operators on $L^{2}_{g}(M,\sS)$. Therefore, we have a natural representation 
$fu_{\phi}\rightarrow fU_{\phi}$ of the  crossed-product algebra $C^{\infty}(M)\rtimes G$ as bounded operators on 
$L^{2}_{g}(M,\sS)$.

\begin{proposition}
\label{prop:equivariant.ST.smooth.n+.summable}
    $(C^{\infty}(M)\rtimes G, L^{2}_{g}(M,\sS), \sD_{g})$ is an $n^{+}$-summable smooth spectral triple. 
\end{proposition}
\begin{proof}
 We know that $(C^{\infty}(M), L^{2}_{g}(M,\sS), \sD_{g})$ is an $n^{+}$-summable 
spectral triple. As the Dirac operator $\sD_{g}$ commutes with the unitary operators $U_{\phi}$, $\phi\in G$, we see that 
$(C^{\infty}(M)\rtimes G, L^{2}_{g}(M,\sS), \sD_{g})$ is a spectral triple as well. Obviously, this spectral triple is 
$n^{+}$-summable. 

In order to show that the spectral triple $(C^{\infty}(M)\rtimes G, L^{2}_{g}(M,\sS), \sD_{g})$ is smooth we only need to show 
that, given any $\phi\in G$, the linear maps $f\rightarrow fU_{\phi}$ and $f\rightarrow 
[\sD_{g},fU_{\phi}]$ are continuous from $C^{\infty}(M)$ to $\cL\left(L^{2}_{g}(M,\sS)\right)$. The continuity of the 
former map is immediate and that of the latter is a consequence of the identities $ 
[\sD_{g},fU_{\phi}]=[\sD_{g},f]U_{\phi}=c(df)U_{\phi}$. 
The proof is complete. 
\end{proof}

As $(C^{\infty}(M)\rtimes G, L^{2}_{g}(M,\sS), \sD_{g})$ is an $n^{+}$-summable smooth spectral triple it has a well 
defined Chern-Connes character $\bCh(\sD_{g})$ in $\op{\mathbf{HP}}^{0}(C^{\infty}(M)\rtimes G)$. The first step is showing that this Connes-Chern character is represented in  $\op{\mathbf{HP}}^{0}(C^{\infty}(M)\rtimes G)$ by the CM cocycle.

In what follows, as in Section~\ref{sec:spectral-triples}, for $m\in \N_{0}$,  we let  $\cD^{m}(M,\sS)$ be the Fr\'echet space of $m$-th order differential 
operators on $M$ acting on the sections of $\sS$. 
We then have the following result.

\begin{lemma}
\label{lem:continuity}
    Let $m\in \N_{0}$ and $\phi\in G$. Then the linear map $P\rightarrow P\sD_{g}^{-m}U_{\phi}$ from 
    $\cD^{m}(M,\sS)$ to $\cL(L_g^{2}(M,\sS))$ is continuous. 
\end{lemma}
\begin{proof}
    Let $\Psi^{0}(M,\sS)$ be the space of zero-th order \psidos\ on $M$ acting on the sections of $\sS$. We equip it 
    with its standard Fr\'echet space topology (see, e.g., \cite[Appendix~A]{LMP:SZEFEO} for a description of this topology). 
    We note that with respect to this topology the following properties are satisfied:
    \begin{itemize}
        \item  The inclusion of $\Psi^{0}(M,\sS)$ into $\cL(L_g^{2}(M,\sS))$ is continuous.
    
        \item  For all $m\in \N_{0}$, the linear map $P\rightarrow P\sD_g^{-m}$ from $\cD^{m}(M,\sS)$ to 
        $\Psi^{0}(M,\sS)$ is continuous.
    \end{itemize}
   Using these two properties we deduce that, for all $m\in \N_{0}$ and $\phi \in G$, the linear map $P\rightarrow P\sD_g^{-m}U_{\phi}$ is continuous from 
    $\cD^{m}(M,\sS)$ to $\cL(L^{2}(M,\sS))$, proving the lemma. 
\end{proof}

\begin{lemma}
\label{lem:equiv.ST.UHR.UR}
    The spectral triple $(C^{\infty}(M)\rtimes G, L^{2}_{g}(M,\sS),\sD_{g})$ is uniformly hypo-regular and uniformly regular. 
\end{lemma}
\begin{proof}
 Let $m\in \N_{0}$ and $\phi \in G$. As the operators $U_{\phi}$ and $\sD_g$ commute with each other we see that, for any
 $f\in C^{\infty}(M)$, we have 
 \begin{equation}
 \label{eq:D_g.U_phi.equalities}
     \sD_g^{m}fU_{\phi}\sD_g^{-m}=(\sD_g^{m}f)\sD_g^{-m}U_{\phi} \quad \text{and} \quad 
     \sD_g^{m}[\sD_g,f]U_{\phi}\sD_g^{-m}=(\sD_g^{m}c(df))\sD_g^{-m}U_{\phi}.
 \end{equation}
 As the linear maps $f\rightarrow \sD_g^{m}f$ and $f\rightarrow \sD_g^{m}[\sD_g,f]$ are continuous from $C^{\infty}(M)$ to 
 $\cD^{m}(M,\sS)$, using Lemma~\ref{lem:continuity} we then deduce that, for all $m\in \N_{0}$ and $\phi\in G$, the linear maps 
 $f\rightarrow \sD_g^{m}f\sD_g^{-m}$ and $f\rightarrow \sD_g^{m}[\sD_g,f]\sD_g^{-m}$ are continuous from $C^{\infty}(M)$ to $\cL(L_g^{2}(M,\sS))$. 
 This proves that  the spectral triple $(C^{\infty}(M)\rtimes G, L^{2}_{g}(M,\sS),\sD_{g})$ is uniformly hypo-regular. 
 
 Similarly, given $m\in \N_{0}$ and $\phi \in G$, for all $f\in C^{\infty}(M)$, we have
 \begin{equation*}
    (fU_{\phi})^{[m]}\sD_g^{-m}=f^{[m]}\sD_g^{-m}U_{\phi} \quad \text{and} \quad   
    \left([\sD_g,f]U_{\phi}\right)^{[m]}\sD_g^{-m}=\left(c(df)\right)^{[m]}\sD_g^{-m}U_{\phi}. 
 \end{equation*}
 As the principal symbol of~\mbox{$\sD_g^{2}$} is scalar, we see that
 $f^{[m]}$ and $\left(c(df)\right)^{[m]}$ are $m$-th order differential operators. Incidentally, the linear maps 
 $f\rightarrow f^{[m]}$ and $f\rightarrow \left(c(df)\right)^{[m]}$  are continuous from $C^{\infty}(M)$ to 
 $\cD^{m}(M,\sS)$. Combining this with the equalities~(\ref{eq:D_g.U_phi.equalities}) and using Lemma~\ref{lem:continuity} we deduce that, for all $m\in \N_{0}$ and $\phi\in G$, the linear maps 
 $f\rightarrow (fU_{\phi})^{[m]}\sD_g^{-m}$ and $f\rightarrow  \left([\sD_g,f]U_{\phi}\right)^{[m]}\sD_g^{-m}$ are continuous from $C^{\infty}(M)$ to $\cL(L_g^{2}(M,\sS))$. 
 It then follows that  $(C^{\infty}(M)\rtimes G, L^{2}_{g}(M,\sS),\sD_{g})$ is uniformly regular. The proof is complete. 
\end{proof}

\begin{lemma}
\label{lem:equivariant.ST.UAE}
 Set  $\Sigma=\left\{\frac{1}{2}(n-\ell);\ \ell\in \N_{0}\right\}$. Then $(C^{\infty}(M)\rtimes G, 
 L^{2}_{g}(M,\sS),\sD_{g})$ has the asymptotic expansion property relatively to $\Sigma$. 
\end{lemma}
\begin{proof}
Set $\cA=C^{\infty}(M)$ and $\cA_{G}=C^{\infty}(M)\rtimes G$. In addition, for $m\in \N_{0}$, we shall denote by 
$\cD^{m}(M,\sS)\rtimes G$ unbounded operators on $\cH$ that are linear combinations of operators 
of the form $PU_{\phi}$, where $P$ ranges over $\cD^{m}(M,\sS)$ and $\phi$ ranges over $G$. Note that $\cD^{0}_{\sD_{g}}(\cA_{G})$ is spanned by 
operators of the form $ fU_{\phi}$ and $[\sD_{g},fU_{\phi}]=c(df)U_{\phi}$, with $f\in \cA$ and $\phi \in G$. 
Therefore, the space $\sD^{1}_{\sD_{g}}(\cA_{G})$ is spanned by $\cD^{0}_{\sD_{g}}(\cA_{G})$ and operators of the form
\begin{equation*}
\sD_{g}fU_{\phi}, \qquad fU_{\phi}\sD_{g}=f\sD_{g}U_{\phi}, \qquad   \sD_{g}c(df)U_{\phi}, \qquad 
c(df)U_{\phi}\sD_{g}=c(df)\sD_{g}U_{\phi}, 
 \end{equation*}where $f$ ranges over $\cA$ and $\phi$ ranges over $G$. We thus see that $\cD^{0}_{\sD_{g}}(\cA_{G})$ 
 (resp., $\cD^{1}_{\sD_{g}}(\cA_{G})$) is contained in $\cD^{0}(M,\sS)\rtimes G$ (resp., $\cD^{1}(M,\sS)\rtimes G$). An 
 induction on $m$ then shows that
 \begin{equation*}
  \cD^{m}_{\sD_{g}}(\cA_{G})   \subset \cD^{m}(M,\sS)\rtimes G \qquad \text{for all $m \in \N_{0}$}. 
 \end{equation*}
Bearing this in mind, let $\phi\in G$ and $P\in \cD^{m}(M,\sS)$, $m \in \N_{0}$. As $\sD_{g}$ commutes with the 
action of $G$, we see that the unitary operator $U_{\phi}$ commutes with the heat semigroup $e^{-t\sD_{g}^{2}}$. Thus, 
\begin{equation}
    \Str\left[ PU_{\phi}e^{-t\sD_{g}^{2}}\right]  = \Str\left[Pe^{-t\sD_{g}^{2}}U_{\phi}\right] \qquad \text{for all 
    $t>0$}.  
    \label{eq:CM.commutation-G-heat-semigroup}
\end{equation}
Using Proposition~\ref{TraceOfHeatKernelVB} we then see that, as $t\rightarrow 0^{+}$, we have 
\begin{equation*}
     \Str\left[ PU_{\phi}e^{-t\sD_{g}^{2}}\right]  
       \sim \sum_{\substack{0\leq a\leq n\\ \textup{$a$ even}}} \sum_{j \geq 0} t^{-\left(\frac{a}{2}+\left[\frac{m}{2}\right]\right)+j} 
     \int_{M_a^{\phi}}\Str\left[ \phi^{\sS}(x)I_{P(\sD_g^{2}+\partial_{t})^{-1}}^{(j)}(x)\right] |dx|. 
 \end{equation*}
 Combining this with~(\ref{eq:X.heat.operator.AE}) shows that $(C^{\infty}(M)\rtimes G, 
 L^{2}_{g}(M,\sS),\sD_{g})$ has the asymptotic expansion property relatively to $\Sigma=\left\{\frac{1}{2}(n-\ell);\ 
 \ell\in \N_{0}\right\}$. The lemma is proved.  
\end{proof}

As the topology of $C^{\infty}(M)\rtimes G$ is barelled, combining Lemma~\ref{lem:equiv.ST.UHR.UR} and Lemma~\ref{lem:equivariant.ST.UAE} 
with Proposition~\ref{prop:UAE=>CMcocyle.rep.CCC} and Proposition~\ref{prop:equivariant.ST.smooth.n+.summable} we then obtain the following statement. 

\begin{proposition}\label{prop:CCC.CMC-represents-CCC}
    The Connes-Chern character of $(C^{\infty}(M)\rtimes G, L^{2}_{g}(M,\sS), \sD_{g})$ is represented in 
    $\op{\mathbf{HP}}^{0}(C^{\infty}(M)\rtimes G)$ by the CM cocycle. Moreover, the formulas~(\ref{eq:CM.cocycle.heat.operator.0})--(\ref{eq:CM.cocycle.heat.operator.>0}) compute this 
    cocycle.  
\end{proposition}

It then remains to compute the CM cocycle by using the formulas~(\ref{eq:CM.cocycle.heat.operator.0})--(\ref{eq:CM.cocycle.heat.operator.>0}). 
As we shall see the computation will follow from the differentiable version of the local 
equivariant index theorem provided by Theorem~\ref{thm:Diff.local.equiv.index.thm} . 

In the following, given a differential form $\omega$ on $M$, for $a=0,2,\ldots,n$,
 we shall denote by $\int_{M_a^\phi} \omega $ the integral of top component of $(\iota_{M_a^\phi})^*\omega$ over 
 $M_a^\phi$, where $\iota_{M_a^\phi}$ is the inclusion of $M_a^\phi$  into $M$. We shall also denote by 
 $|\omega|^{(a,0)}$ the Berezin integral $|(\iota_{M_a^\phi})^*\omega|^{(a,0)}$. 
 We note that
 \begin{equation}
 \int_{M_a^\phi} \omega = \int_{M_a^\phi} (\iota_{M_a^\phi})^*\omega= \int_{M_a^\phi} |\omega|^{(a,0)}|dx|.
 \label{eq:CCCharacter.Berezin-integration}
\end{equation}

\begin{proposition}
\label{prop:small.time.equiv.heat.op.diff.ver}
Let $\phi\in G$ and $f^{0},\ldots,f^{2q}$  in $C^{\infty}(M)$. When $q\geq 1$ and $\alpha\in \N_{0}^{2q}$ set
\begin{equation*}
    P_{q,\alpha}=f^{0}[\sD_{g},f^{1}]^{[\alpha_{1}]}\cdots [\sD_{g},f^{2q}]^{[\alpha_{2q}]}. 
\end{equation*}In addition, set $P_{0,0}=f^{0}$ when $q=0$. 
\begin{enumerate}
    \item  If $q\geq 1$ and $\alpha\neq 0$, then 
    \begin{equation}
    \label{eq:Str.P.heat.op.phi.1}
        \Str\left[ P_{q,\alpha}e^{-t\sD_g^{2}}U_{\phi}\right]=\op{O}\left(t^{-(|\alpha|+q)+1}\right) \qquad \text{as $t\rightarrow 0^{+}$}.
    \end{equation}

    \item  If $\alpha=0$, then, as $t\rightarrow 0^{+}$, we have
    \begin{equation}
     \label{eq:Str.P.heat.op.phi.2}
          \Str\left[ P_{q,0}e^{-t\sD_g^{2}}U_{\phi}\right]=  
          t^{-q}\sum_{\substack{0\leq a\leq n\\ \textup{$a$ even}}} 
         (2\pi)^{-\frac{a}{2}} \int_{M^{\phi}_{a}}\Upsilon_{q} +  \op{O}\left(t^{-q+1}\right),
    \end{equation}
     where we have set $\Upsilon_{q}=(-i)^{\frac{n}{2}}f^{0}df^{1} \wedge \cdots \wedge 
              df^{2q}\wedge\hat{A}(R^{TM^{\phi}})\wedge \nu_{\phi}\left(R^{\cN^{\phi}}\right)$.
\end{enumerate}
\end{proposition}
\begin{remark}
    The above result is proved in~\cite{CH:ECCIDO} (see Theorem~2 of~\cite{CH:ECCIDO} for~(\ref{eq:Str.P.heat.op.phi.1}) and Corollary~3.16  
    of~\cite{CH:ECCIDO} for~(\ref{eq:Str.P.heat.op.phi.2}); see also~\cite{Az:RMJM}). Our aim here is to show that these 
    asymptotics are simple consequences of the differentiable version of the local equivariant index theorem 
    provided by Theorem~\ref{thm:Diff.local.equiv.index.thm} . 
\end{remark}
\begin{proof}[Proof of Proposition~\ref{prop:small.time.equiv.heat.op.diff.ver}] 
    Let $x_{0}\in M_{a}^{\phi}$, $a=0,2,\ldots,n$, and let us work in admissible normal coordinates centered at $x_{0}$. 
    The results of Section~\ref{sec:proof-key-thm} show that
\begin{itemize}
    \item  The multiplication operator $f^{0}$ has Getzler order $0$ and model operator $f^{0}(0)$.

    \item  Each Clifford multiplication operator $c(df^{j})$ has Getzler order $1$ and model operator 
    $df^{j}(0)$.      
    
    \item The operator $\sD_g^{2}$ has Getzler order $2$ and model operator $H_{R}$.
\end{itemize}
Therefore, using Lemma~\ref{lem:index.top-total-order-symbol-composition} we deduce that, at $x=0$, we have
\begin{equation}
\label{eq:G.symbol.equi.heat.op}
    \sigma\left[P_{q,\alpha}\right]= f^{0}(0)df^{1}(0)^{[\alpha_{1}]}\cdots 
    df^{2q}(0)^{[\alpha_{2q}]} + \op{O}_{G}\left(2(|\alpha|+q)-1\right), 
\end{equation}
where $df^{j}(0)^{[\alpha_{j}]}$ is the $\alpha_{j}$-th iterated commutator with $H_{R}$. 
In fact, as $H_{R}$ takes coefficients in forms of degree 0 and degree 2, it commutes with the forms 
$df^{j}(0)$,  and hence $df^{j}(0)^{[\alpha_{j}]}=0$ whenever $\alpha_{j}\neq 0$. Therefore, if $\alpha\neq 0$, then we 
see that $P_{q,\alpha}$ has Getzler order~$2(|\alpha|+q)-1$. As $2(|\alpha|+q)-1$ is an odd integer, 
Theorem~\ref{thm:Diff.local.equiv.index.thm}  shows that, when $\alpha\neq 0$,  we have
    \begin{equation*}
        \Str\left[ P_{q,\alpha}e^{-t\sD_g^{2}}U_{\phi}\right]=\op{O}\left(t^{-(|\alpha|+q)+1}\right) \qquad \text{as $t\rightarrow 0^{+}$}.
    \end{equation*}

Suppose now that $\alpha=0$ and set $\omega=f^{0}df^{1}\wedge \cdots \wedge df^{2q}$. 
Then~(\ref{eq:G.symbol.equi.heat.op}) shows that $P_{q,0}$ has Getzler order $2q$ and model operator 
$\left(P_{q,0}\right)_{(2q)}=\omega(0)$. It then follows that 
\begin{equation}
    K_{\left(P_{q,0}\right)_{(2q)}(H_{R}+\partial_{t})^{-1}}(x,y,t)=K_{\omega(0)\wedge(H_{R}+\partial_{t})^{-1}}(x,y,t)=\omega(0)\wedge K_{(H_{R}+\partial_{t})^{-1}}(x,y,t).   
    \label{eq:CCCharacter.kernel-Pq0}
\end{equation}
Thus, 
\begin{equation}
    I_{\left(P_{q,0}\right)_{(2q)}(H_{R}+\partial_{t})^{-1}}(x,t)= \omega(0)\wedge I_{(H_{R}+\partial_{t})^{-1}}(x,t).
    \label{eq:CCCharacter.IPq0}
\end{equation}
Combining this with~(\ref{eq:LEIT.horizontal-component}) we deduce that
\begin{equation*}
    I_{\left(P_{q,0}\right)_{(2q)}(H_{R}+\partial_{t})^{-1}}(0,1)^{(\bt,0)}= 
    \frac{(4\pi)^{-\frac{a}{2}}}{
        {\det}^{\frac12}\left(1-\phi^{\cN}\right)} \omega(0)^{(\bt,0)} \wedge \hat{A}(R^{TM^{\phi}}(0))\wedge 
            \nu_{\phi}\left(R^{\cN^{\phi}}(0)\right).
\end{equation*}
Therefore, we obtain
\begin{align}
    \gamma_{\phi}(P_{q,0};\sD_{g})(0) & = (-i)^{\frac{n}{2}}2^{\frac{a}{2}}{\det}^{\frac{1}{2}}(1-\phi^{\cN}(0)) 
    \left|I_{\left(P_{q,0}\right)_{(2q)}(H_{R}+\partial_{t})^{-1}}(0,1)\right|^{(a,0)} \nonumber \\
     & = (-i)^{\frac{n}{2}}(2\pi)^{-\frac{a}{2}} \left|\omega(0)\wedge \hat{A}(R^{TM^{\phi}}(0))\wedge 
            \nu_{\phi}\left(R^{\cN^{\phi}}(0)\right)\right|^{(a,0)}. 
                 \label{eq:CCCharacter.UpsilonPq0}
\end{align}
As $P_{q,0}$  has Getzler order $2q$, combining this with Theorem~\ref{thm:Diff.local.equiv.index.thm}  and using~(\ref{eq:CCCharacter.Berezin-integration}) gives the asymptotic~(\ref{eq:Str.P.heat.op.phi.2}). 
 The proof is complete. 
\end{proof}

We are now in a position to prove the main result of this section. 

\begin{theorem}\label{thm:Connes-Chern-conformal-DiracST}
The Connes-Chern character  of $(C^{\infty}(M)\rtimes G, L^{2}_{g}(M,\sS), \sD_{g})$
is represented in $\op{\mathbf{HP}}^{0}(C^{\infty}(M)\rtimes 
G)$ by the CM cocycle $\varphi^{\CM}=(\varphi_{2q}^{\CM})$. Moreover, for all $f^0, ..., f^{2q}$ in $C^\infty(M)$ and $\phi_0,...,\phi_{2q}$ in $G$, we have 
 \begin{multline}
 \varphi_{2q}^{\CM}(f^{0}U_{\phi_{0}},\ldots, f^{2q}U_{\phi_{2q}})= \\
 \frac{(-i)^{\frac{n}{2}}}{(2q)!}\sum_{\substack{0\leq a \leq n\\ \textup{$a$ even}}}    
 (2\pi)^{-\frac{a}{2}}\int_{M^{\phi_{(2q)}}_{a}}   f^{0}d\hat{f}^{1} \wedge \cdots \wedge 
             d\hat{f}^{2q} \wedge \hat{A}(R^{TM^{\phi_{(q)}}})\wedge 
            \nu_{\phi_{(2q)}}\left(R^{\cN^{\phi_{(2q)}}}\right), 
             \label{eq:Conformal.CM-cocycle}
\end{multline}where we have set $\phi_{(j)}:=\phi_{0}\circ \cdots \circ \phi_{j}$ and $\hat{f}^{j}:=f^{j}\circ 
\phi_{(j-1)}^{-1}$. 
\end{theorem}
\begin{proof}
  Thanks to Proposition~\ref{prop:small.time.equiv.heat.op.diff.ver} we know that  the Connes-Chern character  $\bCh(\sD_{g})$ is represented in 
  $\op{\mathbf{HP}}^{0}(C^{\infty}(M)\rtimes G)$ by the CM cocycle $\varphi^{\CM}=(\varphi_{2q}^{\CM})$. Moreover, this CM cocycle is given by the 
  formulas~(\ref{eq:Str.P.heat.op.phi.1})--(\ref{eq:Str.P.heat.op.phi.2}). It then remains to compute the CM cocycle 
  $\varphi^{\CM}=(\varphi_{2q}^{\CM})$ by using these formulas. Note that $\varphi_{2q}^{\CM}=0$ for $q\geq \frac{1}{2}n+1$, since 
  $(C^{\infty}(M)\rtimes G, L^{2}_{g}(M,\sS), \sD_{g})$ is $n^{+}$-summable, so we only have to compute 
  $\varphi_{2q}^{\CM}$ for $q=0,\ldots,\frac{1}{2}n$. 
  
  Let $\phi_{0}\in G$ and $f^{0}\in C^{\infty}(M)$. Using~(\ref{eq:Str.P.heat.op.phi.1}) and (\ref{eq:CM.commutation-G-heat-semigroup}) we get 
 \begin{equation*}
     \varphi_{0}^{\CM}(fU_{\phi_{0}})= \FP \Str \left[f^{0}U_{\phi_{0}}e^{-t\sD_g^{2}}\right]= \FP \Str 
     \left[f^{0}e^{-t\sD_g^{2}}U_{\phi_{0}}\right]. 
 \end{equation*}Combining this with~(\ref{eq:LEIT.local-equiv.-index-thm}) then gives
 \begin{equation*}
     \varphi_{0}^{\CM}(f^{0}U_{\phi_{0}}) =   (-i)^{\frac{n}{2}}\sum_{\substack{0\leq a \leq n\\ \textup{$a$ even}}}    
 (2\pi)^{-\frac{a}{2}}\int_{M^{\phi_{0}}_{a}} f^{0} \hat{A}(R^{TM^{\phi_{0}}})\wedge 
            \nu_{\phi_{0}}\left(R^{\cN^{\phi_{0}}}\right), 
 \end{equation*}which is the formula~(\ref{eq:Conformal.CM-cocycle}) for $q=0$. 
 
Given $q\in\left\{1,\ldots,\frac{1}{2}n\right\}$, let $\phi_{j}\in G$ and $f^{j}\in 
  C^{\infty}(M)$, $j=0,\ldots,q$. In addition, for $\alpha \in \N_{0}^{2q}$, set $T_{q,\alpha}= f^{0}U_{\phi_{0}} 
    \left[\sD_g,f^{1}U_{\phi_{1}}\right]^{[\alpha_{1}]} \cdots 
          \left[\sD_g,f^{2q}U_{\phi_{2q}}\right]^{[\alpha_{2q}]}$. We observe that, as $\sD_{g}$ commutes with the action 
          of $G$, given any $\psi_{1}$ and $\psi_{2}$ in $G$ and $f\in C^{\infty}(M)$,  for all $j\in \N$, we have
          \begin{equation}
              U_{\psi_{1}}[\sD_{g},fU_{\psi_{2}}]^{[j]}= 
              [\sD_{g},U_{\psi_{1}}fU_{\psi_{1}}^{-1}]^{[j]}U_{\psi_{1}}U_{\psi_{2}} = [\sD_{g},f\circ 
              \psi_{1}^{-1}]^{[j]}U_{\psi_{1}\psi_{2}}. 
              \label{eq:CCCharacter.commutationUphi-f}
          \end{equation}
 Repeated use of these equalities enables us to rewrite $T_{q,\alpha}$ as
 \begin{equation*}
     T_{q,\alpha}=P_{q,\alpha}U_{\phi_{(2q)}}, \qquad P_{q,\alpha}:= f^{0} 
     \left[\sD_g,\hat{f}^{1}\right]^{[\alpha_{1}]}\cdots  \left[\sD_{g},\hat{f}^{2q}\right]^{[\alpha_{2q}]}, 
 \end{equation*}where we have set $\phi_{(j)}:=\phi_{0}\circ \cdots \circ \phi_{j}$ and $\hat{f}^{j}:=f^{j}\circ 
\phi_{(j-1)}^{-1}$. Therefore, using~(\ref{eq:Str.P.heat.op.phi.2}) and~(\ref{eq:CM.commutation-G-heat-semigroup}) we 
obtain
\begin{equation*}
        \varphi_{2q}^{\CM}(f^{0}U_{\phi_{0}},\ldots,f^{2q}U_{\phi_{2q}})= \sum_{\alpha \in \N_{0}^{2q}}c_{q,\alpha}\Gamma(|\alpha|+q)^{-1} 
        \FP \left\{t^{|\alpha|+q}\Str\left[ P_{q,\alpha}e^{-t\sD_g^{2}}U_{\phi_{(2q)}}\right]\right\} .    
\end{equation*}Combining this with Proposition~\ref{prop:small.time.equiv.heat.op.diff.ver} then gives 
\begin{equation*}
  \varphi_{2q}^{\CM}(f^{0}U_{\phi_{0}},\ldots,f^{2q}U_{\phi_{2q}})=  c_{q,0}\Gamma(q)^{-1}   \sum_{\substack{0\leq a \leq n\\ \textup{$a$ even}}}    
 (2\pi)^{-\frac{a}{2}}\int_{M^{\phi_{(2q)}}_{a}} \Upsilon_{q},
\end{equation*}
where we have set $\Upsilon_{q}=(-i)^{\frac{n}{2}}f^{0}d\hat{f}^{1} \wedge \cdots \wedge d\hat{f}^{2q}\wedge 
\hat{A}(R^{TM^{\phi_{(2q)}}})\wedge \nu_{\phi_{(2q)}}\left(R^{\cN^{\phi_{(2q)}}}\right)$. As 
$c_{q,0}\Gamma(q)^{-1}$ is equal to $((2q)!)^{-1}$ this gives the formula~(\ref{eq:Conformal.CM-cocycle}) for 
$q=1,\ldots,\frac{1}{2}n$. The proof is complete. 
\end{proof}

\begin{remark}
 To understand the formula~(\ref{eq:Conformal.CM-cocycle}) it is worth looking at the top-degree component 
$\varphi_{n}^{\CM}$. For $q=\frac{1}{2}n$ the r.h.s.~of~(\ref{eq:Conformal.CM-cocycle}) reduces to an integral 
over $M_{n}^{\phi_{(n)}}$ and this submanifold is empty unless $\phi_{(n)}=\op{id}$. Thus, 
\begin{equation*}
    \varphi_{n}^{\CM}(f^{0}U_{\phi_{0}},\ldots, f^{n}U_{\phi_{n}})= \left\{ 
    \begin{array}{ll} {\displaystyle \frac{(2i\pi)^{-\frac{n}{2}}}{n!}     \int_{M}  
            f^{0}d\hat{f}^{1} \wedge \cdots \wedge d\hat{f}^{n} } & \text{if $\phi_{0}\circ \cdots \circ 
            \phi_{n}=\op{id}$},\\
0     & \text{otherwise}.  
    \end{array}\right.
\end{equation*}That is, $\varphi_{n}^{\CM}$ agrees with the transverse fundamental class cocycle of
Connes~\cite{Co:CCTFF}. This implies that the Hochschild class of the Connes-Chern character agrees with Connes' 
transverse fundamental class (see also~\cite[Proposition~3.7]{Mo:LIFTST}).
\end{remark}

\begin{remark}\label{rmk:CM.odd}
    We refer to~\cite{Po:Odd} for the computation of the Connes-Chern character of a crossed-product Dirac spectral 
    triple $(C^{\infty}(M)\rtimes G, L^{2}_{g}(M,\sS), \sD_{g})$ in odd dimensions. In this case the spectral triple is 
    an \emph{odd} spectral triple. The Connes-Chern character then is defined as a class $\bCh(D)\in 
    \bHP^{1}\left(C^{\infty}(M)\rtimes G\right)$ and is represented by a CM cocycle. The formula for the CM cocycle is 
    similar to that in the even dimensional case given by~(\ref{eq:Conformal.CM-cocycle}). 
\end{remark}

\section{The JLO Cocycle of a Dirac Spectral Triple}\label{sec:JLO}
In this section, as a further application of Theorem~\ref{thm:LEIT.Volterra-LEIT}, we shall compute the short-time limit of the JLO 
cocycle of an equivariant Dirac spectral triple. As we shall see, with our approach, the computation of JLO-type 
cochains associated to a Dirac spectral triple is not more difficult than the computation of the CM cocycle. 

 \subsection{The JLO cocycle of a spectral triple} Let $(\cA,\cH,D)$ be a $p^{+}$-summable spectral triple. In what follows, given  
 operators $X^{0},\ldots,X^{m}$ in $\cD_{D}^{\bt}(\cA)$, for all $t>0$, we define
  \begin{equation*}
    H_{t}(X^{0},\ldots,X^{m})=  
    \int_{\Delta_{m}}X^{0}e^{-s_{0}tD^{2}}X^{1}e^{-s_{1}tD^{2}}\cdots   
    X^{m}e^{-s_{m}tD^{2}}d\mathbf{s}. 
 \end{equation*}where $\Delta_{m}=\{\mathbf{s}=(s_{0},\ldots,s_{m})\in [0,1]^{m+1}; s_{0}+\cdots + s_{m}=1\}$. The JLO 
 cochain~\cite{JLO:QKTCC} is actually a family of infinite-supported even cochains $\varphi^{\JLO}_{t}=(\varphi^{\JLO}_{2q,t} 
 )_{q\geq 0}$, where $\varphi^{\JLO}_{2q,t}$ is the $2q$-cochain on $\cA$ given by
 \begin{equation*}
   \varphi^{\JLO}_{t}(a^{0},\ldots,a^{2q})= t^{q}\Str\left[ H_{t}\left(a^{0},[D,a^{1}],\ldots,[D,a^{2q}]\right) 
   \right], \qquad a^{j}\in \cA, \ t>0.  
 \end{equation*}
 Let $A$ be the Banach $*$-algebra obtained as the completion of $\cA$ with respect to the norm $a\rightarrow 
 \|a\| +\|[D,a]\|$. Then it can be shown that $\varphi^{\JLO}_{t}$ gives rise to a cocycle in the entire cyclic cohomology of $A$ whose class is independent of $t$
 (see~\cite{JLO:QKTCC, GS:OCCTSM}). Moreover, this  class agrees with the Connes-Chern character in 
 entire cyclic cohomology~\cite{Co:ECCBACTSFM, Co:CCTSM}.  In fact, as pointed out by 
 Quillen~\cite{Qu:ACCC}, the JLO cocycle can be naturally interpreted as the Chern character of a suitable superconnection
 with values in cochains. It should also be mentioned that the definition of the JLO cocycle and its aforementioned properties 
 only require the spectral triple $(\cA,\cH,D)$ to be $\theta$-summable, which is a weaker condition than 
 $p^{+}$-summability or $p$-summability. 
 
 When $(\cA,\cH,D)$ is $p$-summable, Connes-Moscovici~\cite{CM:TCCFDKC} showed that the JLO cocycle retracts to a periodic 
 cyclic cocycle representing the Connes-Chern character in $\HP^{0}(\cA)$. They also proved that, under suitable short-time
 asymptotic properties for the supertraces $\Str\left[ H_{t}\left(X^{0},X^{1},\ldots,X^{m}\right)\right]$, $X^{j}\in 
 \cD_{D}^{1}(\cA)$, we can define \emph{parties finies} $\Pf_{t=0^+} \varphi^{\JLO}_{2q,t}(a^{0},\ldots,a^{2q})$, $a^j\in \cA$, in such way to obtain
 a periodic cyclic cocycle $\Pf_{t=0^+}\varphi^{\JLO}_t$ representing the Connes-Chern character in 
 $\HP^{0}(\cA)$. In fact, this cocycle is naturally identified with the CM cocycle (see~\cite{CM:LIFNCG}). 
 
 We stress out that the results of~\cite{CM:TCCFDKC} do not require the spectral triple $(\cA,\cH,D)$ to be regular. 
 Therefore, computing $\Pf_{t=0^+} \varphi^{\JLO}_{t}$ is an alternative to the computation of the CM cocycle 
 for spectral triples that are not regular. Note that there are natural geometric examples of spectral triples associated to 
 hypoelliptic operators on contact manifolds or even Carnot manifolds, which are not regular and satisfy the 
 asymptotic expansion assumptions of~\cite{CM:TCCFDKC}. In fact, the regularity property is an operator-thereotic 
 reformulation of the scalarness of the principal symbol of the square of the Dirac operator. Therefore, it should 
 not be surprising that this property may fail in some ``highly" noncommutative examples. 
 
Block-Fox~\cite{BF:APDOIT} computed  $\Pf_{t=0^+} \varphi^{\JLO}_{t}$ for a Dirac spectral triple by using the asymptotic pseudodifferential calculus of 
Getzler~\cite{Ge:POSASIT}. 
As we shall explain below, for Dirac spectral triples the computation of $\Pf_{t=0^+} \varphi^{\JLO}_{t}$ is not much more 
 difficult than the computation of the CM cocycle given in Section~\ref{sec:Connes-Chern-conformal} (or in~\cite{Po:CMP} in the non-equivariant setting).
 
\subsection{Equivariant Dirac spectral triple}
As a further example of application of the local equivariant index theorem for Volterra \psidos\ provided by 
 Theorem~\ref{thm:LEIT.Volterra-LEIT}, we shall now show how this result enables us to establish the short-time limit of the JLO cocycle $\varphi^{\JLO}_{t}$ of an equivariant 
 Dirac spectral 
 triple $( C^{\infty}(M)\rtimes G, L^{2}_{g}(M,\sS),\sD_{g})$. As in previous sections, $(M^{n},g)$ is an even dimensional compact 
 spin oriented Riemannian manifold, $\sD_{g}$ is its Dirac operators acting on spinors, and $G$ is a subgroup of the 
 connected component of the identity component of the group of orientation-preserving isometries preserving the spin 
 structure. More precisely, our goal is to prove the following result. 
 
  \begin{theorem}\label{thm:JLO.PfJLO}
 For all $f^0,...,f^{2q}$ in $C^\infty(M)$ and $\phi_0,...,\phi_{2q}$ in $G$, as $t\rightarrow 0^{+}$, we have 
 \begin{multline}
 \lim_{t\rightarrow 0^{+}}\varphi_{2q,t}^{\JLO}(f^{0}U_{\phi_{0}},\ldots, f^{2q}U_{\phi_{2q}})= \\
 \frac{(-i)^{\frac{n}{2}}}{(2q)!}\sum_{\substack{0\leq a \leq n\\ \textup{$a$ even}}}    
 (2\pi)^{-\frac{a}{2}}\int_{M^{\phi}_{a}}   f^{0}d\hat{f}^{1} \wedge \cdots \wedge 
             d\hat{f}^{2q} \wedge \hat{A}(R^{TM^{\phi}})\wedge 
            \nu_{\phi}\left(R^{\cN^{\phi}}\right),
             \label{eq:JLO.JLO-coycle-Dirac}
\end{multline}where we have set $\phi=\phi_{0}\circ \cdots \circ \phi_{2q}$ and $\hat{f}^{j}:=f^{j}\circ 
    (\phi_{0}\circ \cdots \circ \phi_{j-1})^{-1}$, $j=1,\ldots,2q$. 
 \end{theorem}
 
\begin{remark}
 When $G=\{\op{id}\}$ we recover the result of Block-Fox~\cite{BF:APDOIT} on the short-time limit of the JLO cocycle of a non-equivariant Dirac spectral triple 
 $(C^{\infty}(M), L^2_g(M,\sS),\sD_g)$. 
\end{remark}
 
In what follows, given differential operators $X^{0},\ldots,X^{m}$ on $M$ acting on spinors we set
\begin{equation*}
    Q(X^{0},\ldots,X^{m})=X^{0}(\sD_{g}^{2}+\partial_{t})^{-1}\cdots X^{m}(\sD_{g}^{2}+\partial_{t})^{-1}.
\end{equation*}
Note that $Q(X^{0},\ldots,X^{m})$ is a Volterra \psido\ of order $\ord X^{0}+\cdots + \ord X^{m}-2m-2$.  We will deduce 
Theorem~\ref{thm:JLO.PfJLO} from the following result. 

\begin{proposition}\label{prop:JLO.strIQ}
   Given $\phi \in G$, let $f\in C^{\infty}(M)$ and $\omega^{j}\in C^{\infty}(M,T^{*}_{\C}M)$, $j=1,\ldots,2q$. In 
   addition, set 
   $Q=Q\left(f,c(\omega^{1}),\ldots,c(\omega^{2q})\right)$. Then, as $t\rightarrow 0^{+}$ and uniformly on each 
   fixed-point submanifold $M_{a}^{\phi}$, we have
   \begin{equation}
       \str_{\sS}\left[\phi^{\sS}(x)I_{Q}(x,t)\right]=
      \frac{1}{(2q)!} t^{q}\left| f\wedge \omega^{1}\wedge \cdots \wedge \omega^{2q}\wedge \hat{A}(R^{TM^{\phi}})\wedge 
       \nu_{\phi}(R^{\cN^{\phi}})\right|^{(a,0)}+\op{O}(t^{q+1}). 
       \label{eq:JLO.strIQ}
   \end{equation}
\end{proposition}

The proof of Proposition~\ref{prop:JLO.strIQ} is a direct application of Theorem~\ref{thm:JLO.PfJLO}. Before getting to this let us explain how 
Proposition~\ref{prop:JLO.strIQ} enables us to prove Theorem~\ref{thm:JLO.PfJLO}. The  key observation is the following elementary lemma.

 \begin{lemma}\label{lem:JLO.Ht-Q}
     Let $X^{0},\ldots,X^{m}$ be differential operators on $M$ acting on spinors. For $t>0$ denote by 
     $h_{t}(X^{0},\ldots,X^{m})(x,y)$ the kernel of $H_{t}(X^{0},\ldots,X^{m})$ in the sense of~(\ref{eq:Heat.heat-kernel-smooth-function}). Then
     \begin{equation*}
         h_{t}(X^{0},\ldots,X^{m})(x,y)=t^{-m}K_{Q(X^{0},\ldots,X^{m})}(x,y,t) \qquad \text{for all $t>0$}.   
     \end{equation*}
 \end{lemma}
 \begin{proof}
 Let $u\in C^{\infty}_{+}(\R, L^{2}(M,\sS))$. Then 
 \begin{align*}
  Q(X^{0},\ldots,X^{m})u(t)    &=  \int_{0}^{\infty} X^{0}e^{-t_{0}\Delta}Q(X^{1}, \ldots, X^{m})u(t-t_{0})dt_{0}\\
      & = \int_{0}^{\infty}\int_{0}^{\infty}X^{0}e^{-t_{0}\Delta}X^{1}e^{-t_{1}\Delta}Q(X^{2}, \ldots, X^{m})u(t-t_{0}-t_{1})dt_{0}dt_{1}.
 \end{align*}
 An induction then shows that
 \begin{equation*}
     Q(X^{0},\ldots,X^{m})u(t)=  \int_{0}^{\infty}\cdots \int_{0}^{\infty}X^{0}e^{-t_{0}\Delta}\cdots 
     X^{m}e^{-t_{m}\Delta}u(t-t_{0}-\cdots -t_{m})dt_{0}\cdots dt_{m}.
 \end{equation*}
 The change of variables $\sigma=t_{0}+\cdots +t_{m}$ and $s_{j}=\sigma^{-1}t_{j}$, $j=0,\ldots,m$, gives
 \begin{align}
  Q(X^{0},\ldots,X^{m})u(t)    &=   \int_{0}^{\infty}\int_{\Delta_{m}}X^{0}e^{-s_{0}\sigma\Delta}\cdots 
     X^{m}e^{-s_{m}\sigma\Delta}u(t-\sigma)\sigma^{m}d\mathbf{s}d\sigma  \nonumber\\
      & =  \int_{0}^{\infty} \sigma^{m} H_{\sigma}(X^{0},\ldots,X^{m})u(t-\sigma)d\sigma.
      \label{eq:JLO.Q-Ht}
 \end{align}
 In the same way as we obtained~(\ref{eq:Heat.KQ0-heat-kernel}) this shows that
     \begin{equation*}
       K_{Q(X^{0},\ldots,X^{m})}(x,y,t)=t^{m}h_{t}(X^{0},\ldots,X^{m})(x,y) \qquad \text{for  all $t>0$}.   
     \end{equation*}
 The lemma is proved.  
 \end{proof}
 
\begin{remark}
    The formula~(\ref{eq:JLO.Q-Ht}) is remeniscent of the resolvent formula for the JLO cocycle given by Connes~\cite[Eq.~(17)]{Co:CCTSM}. In 
    fact, at least at a formal level, we go from one formula to the other by a conjugation by the Laplace transform 
    with respect to the variable $t$, since this transforms the inverse heat operator $(\sD_{g}^{2}+\partial_{t})^{-1}$ into the resolvent 
    $(\sD^{2}_{g}-\lambda)^{-1}$, $\lambda>0$. 
\end{remark}

 \begin{proof}[Proof of Theorem~\ref{thm:JLO.PfJLO}] Let $f^{j}\in C^{\infty}(M)$ and $\phi_{j}\in G$,  $j=0,1,\ldots,2q$. By 
    definition, for 
      all $t>0$, we have
    \begin{equation*}
      \varphi^{\JLO}_{2q,t}(f^{0}U_{\phi_{0}},\ldots, 
      f^{2q}U_{\phi_{2q}})=t^{q}\Str\left[H_{t}\left(f^{0}U_{\phi_{0}},[\sD_{g},f^{1}U_{\phi_{1}}],\ldots, 
      [\sD_{g},f^{2q}U_{\phi_{2q}}]\right)\right] .
    \end{equation*}
    As the unitary operators $U_{\phi_{j}}$ commute with the heat semigroup $e^{-t\sD_{g}^{2}}$, $t>0$, by arguing as 
    in~(\ref{eq:CCCharacter.commutationUphi-f}) it can be shown that,  for all $t>0$, we have
    \begin{align*}
        H_{t}\left(f^{0}U_{\phi_{0}},[\sD_{g},f^{1}U_{\phi_{1}}],\ldots, 
      [\sD_{g},f^{2q}U_{\phi_{2q}}]\right)& = H_{t}\left(f^{0},[\sD_{g},\hat{f}^{1}],\ldots, 
      [\sD_{g},\hat{f}^{2q}]\right)U_{\phi}\\
      & = H_{t}\left(f^{0}, c(d\hat{f}^{1}),\ldots, c(d\hat{f}^{2q})\right)U_{\phi},
    \end{align*}where we have set $\phi=\phi_{0}\circ \cdots \circ \phi_{2q}$ and $\hat{f}^{j}:=f^{j}\circ 
    (\phi_{0}\circ \cdots \circ \phi_{j-1})^{-1}$, $j=1,\ldots,2q$. Therefore, using Lemma~\ref{lem:JLO.Ht-Q} and its notation, we see 
    that, for all $t>0$, we have
    \begin{align}
      \varphi^{\JLO}_{2q,t}(f^{0}U_{\phi_{0}},\ldots, 
      f^{2q}U_{\phi_{2q}})    & = t^{q}\Str\left[H_{t}\left(f^{0}, c(d\hat{f}^{1}),\ldots, 
      c(d\hat{f}^{2q})\right)U_{\phi}\right] \nonumber\\
         & = t^{q}\int_{M} \str_{\sS}\left[ \phi^{\sS}(x)h_{t}\left(f^{0}, c(d\hat{f}^{1}),\ldots, 
      c(d\hat{f}^{2q})\right)(x,\phi(x))\right]|dx| 
      \label{eq:JLO.JLO-strIQ}\\
         & = t^{-q}\int_{M} \str_{\sS}\left[ \phi^{\sS}(x)K_{Q}(x,\phi(x),t)\right]|dx|,\nonumber
    \end{align}where we have set $Q=Q\left(f^{0}, c(d\hat{f}^{1}),\ldots, c(d\hat{f}^{2q})\right)$. 
    Moreover, Lemma~\ref{lem:Heat-localization} and Proposition~\ref{prop:JLO.strIQ} shows that, as $t\rightarrow 0^{+}$, we have
    \begin{align*}
      \int_{M} \str_{\sS}\left[ \phi^{\sS}(x)K_{Q}(x,\phi(x),t)\right]|dx|   & = \int_{M} \str_{\sS}\left[ 
      \phi^{\sS}(x)I_{Q}(x,t)\right]|dx|  +\op{O}(t^{\infty}) \\
         & = \frac{1}{(2q)!} t^{q}\sum_{\substack{0\leq a \leq n\\ \textup{$a$ even}}} 
         (2\pi)^{-\frac{a}{2}}\int_{M^{\phi}_{a}}\Upsilon_{q} + \op{O}(t^{q+1}), 
    \end{align*} where we have set $\Upsilon_{q}=(-i)^{\frac{n}{2}} f^{0}d\hat{f}^{1} \wedge \cdots \wedge 
             d\hat{f}^{2q} \wedge \hat{A}(R^{TM^{\phi}})\wedge 
            \nu_{\phi}\left(R^{\cN^{\phi}}\right)$. Combining this with~(\ref{eq:JLO.JLO-strIQ}) proves~(\ref{eq:JLO.JLO-coycle-Dirac}). The 
            proof is complete. 
\end{proof}
 
It remains to prove Proposition~\ref{prop:JLO.strIQ}. The rest of the subsection is devoted to proving this result. 

\begin{proof}[Proof of Proposition~\ref{prop:JLO.strIQ}] Given $\phi \in G$, let $f\in 
C^{\infty}(M)$ and $\omega^{j}\in C^{\infty}(M,T^{*}_{\C}M)$, $j=1,\ldots,2q$. In addition, set
\begin{equation*}
   Q=Q\left(f,c(\omega^{1}),\ldots,c(\omega^{2q})\right)=f (\sD_{g}^{2}+\partial_{t})^{-1}c(\omega^{2q})(\sD_{g}^{2}+\partial_{t})^{-1}
   \cdots c(\omega^{2q})(\sD_{g}^{2}+\partial_{t})^{-1}.
\end{equation*}
Here $f$ has Getzler order $0$, each Clifford multiplication operator $c(\omega^{j})$ has Getzler order~$1$, and the 
inverse heat operator $(\sD_{g}^{2}+\partial_{t})^{-1}$ has Getzler order~$-2$, and so using Lemma~\ref{lem:index.top-total-order-symbol-composition} we see that $Q$ 
has Getzler order $2q-2(2q+1)=-2(q+1)$. 
Therefore, Theorem~\ref{thm:LEIT.Volterra-LEIT} ensures us that, as $t\rightarrow 0^{+}$ and uniformly on each fixed-point submanifold $M_{a}^{\phi}$, we have 
 \begin{equation}
            \str_{\sS}\left[ \phi^{\sS}(x)I_{Q}(x,t)\right]=  t^{q} \gamma_{\phi}(Q)(x) + 
            \op{O}\left( t^{q+1}\right) .
            \label{eq:JLO.IQ-UpsilonQ}
\end{equation}
To complete the proof it then remains to identify $\gamma_{\phi}(Q)(x)$. 

Let $x_{0}$ be a point in some fixed-point submanifold $M_{a}^{\phi}$, 
$a=0,2,\ldots,n$, and let us work in admissible normal coordinates centered at $x_{0}$. At $x=0$ the respective model 
operators of $f$, $c(\omega^{j})$ and $(\sD_{g}^{2}+\partial_{t})^{-1}$ are $f(0)$, $\omega^{j}(0)$ and 
$(H_{R}+\partial_{t})^{-1}$. Therefore, using Lemma~\ref{lem:index.top-total-order-symbol-composition} we see that $Q$ has model operator
\begin{equation*}
    Q_{(-2q-2)}=f(0)(H_{R}+\partial_{t})^{-1}\omega^{1}(0)(H_{R}+\partial_{t})^{-1}\cdots 
    \omega^{2q}(0)(H_{R}+\partial_{t})^{-1}.
\end{equation*}
As pointed out in the proof of Theorem~\ref{thm:LEIT.local-equiv.-index-thm-pointwise} in Section~\ref{sec:proof-key-thm}, the operator $(H_{R}+\partial_{t})^{-1}$ commutes with 
with the forms $\omega^{j}(0)$. Thus, setting $\omega=f\omega^{1}\wedge \cdots \wedge \omega^{2q}$, we can rewrite $Q_{(-2q-2)}$ as
\begin{equation*}
   Q_{(-2q-2)}=f(0)\omega^{1}(0)\wedge  \cdots \wedge \omega^{2q}(0)\wedge(H_{R}+\partial_{t})^{-(2q+1)}= 
   \omega(0)\wedge (H_{R}+\partial_{t})^{-(2q+1)}. 
\end{equation*}Therefore, arguing as in~(\ref{eq:CCCharacter.kernel-Pq0})--(\ref{eq:CCCharacter.IPq0}) shows that
\begin{equation}
    I_{Q_{(-2q-2)}}(x,t)=\omega (0)\wedge I_{(H_{R}+\partial_{t})^{-(2q+1)}}(x,t). 
    \label{eq:JLO.IQ(-2q-2)-IHR}
\end{equation}

\begin{claim}
    Let $m\in \N_{0}$. Then
    \begin{equation}
        I_{(H_{R}+\partial_{t})^{-(m+1)}}(x,t)=\frac{1}{m!}t^{m} I_{(H_{R}+\partial_{t})^{-1}}(x,t).
        \label{eq:JLO.IHRm} 
    \end{equation}
\end{claim}
 \begin{proof}
     The proof is based on the following observation: for any $Q\in \pvdo^{\bt}(\R^{n}\times \R)\otimes 
     \Lambda(n)$, the commutator $[t,Q]$ has kernel $(t-s)K_{Q}(x,y,t-s)$, and hence
     \begin{equation}
         K_{[t,Q]}(x,y,t)=t K_{Q}(x,y,t)  \qquad \text{and} \qquad I_{[t,Q]}(x,t)=tI_{Q}(x,t).
         \label{eq:JLO.I[t,Q]}
     \end{equation}
   Bearing this in mind, we shall prove~(\ref{eq:JLO.IHRm}) by induction on $m$. It is immediate that~(\ref{eq:JLO.IHRm}) is true for 
   $m=0$. Assume it is true for $m\geq 0$. As $[H_{R}+\partial_{t},t]=[\partial_{t},t]=1$, we have 
     \begin{equation*}
         [t,(H_{R}+\partial_{t})^{-1}]=(H_{R}+\partial_{t})^{-1}[H_{R}+\partial_{t},t](H_{R}+\partial_{t})^{-1}=  (H_{R}+\partial_{t})^{-2}. 
     \end{equation*}
     This implies that $[t,(H_{R}+\partial_{t})^{-m}] $ is equal to
     \begin{equation*}
          \sum_{0\leq j \leq 
         m-1}(H_{R}+\partial_{t})^{-j}[t,(H_{R}+\partial_{t})^{-1}](H_{R}+\partial_{t})^{-m+j+1}=m(H_{R}+\partial_{t})^{-(m+1)}.
     \end{equation*}
     Combining this with~(\ref{eq:JLO.I[t,Q]}) we then get
     \begin{equation*}
         I_{(H_{R}+\partial_{t})^{-(m+1)}}(x,t)=\frac{1}{m}I_{[t,(H_{R}+\partial_{t})^{-m}]}(x,t)=\frac{1}{m}tI_{(H_{R}+\partial_{t})^{-m}}(x,t). 
     \end{equation*}
     As formula~(\ref{eq:JLO.IHRm}) is true for $m$, we deduce that 
     \begin{equation*}
        I_{(H_{R}+\partial_{t})^{-(m+1)}}(x,t)=\frac{1}{(m+1)!}t^{m+1}I_{(H_{R}+\partial_{t})^{-1}}(x,t).
     \end{equation*}
      This shows that formula~(\ref{eq:JLO.IHRm}) is true for $m+1$. The proof of the claim is thus complete. 
     \end{proof}

Let us go back to the proof of Proposition~\ref{prop:JLO.strIQ}. Combining~(\ref{eq:JLO.IQ(-2q-2)-IHR}) with~(\ref{eq:JLO.IHRm}) shows that 
\begin{equation*}
    I_{Q_{(-2q-2)}}(x,t)=\frac{1}{(2q)!}\omega (0)\wedge  I_{(H_{R}+\partial_{t})^{-1}}(x,t). 
\end{equation*}Therefore, by arguing as in~(\ref{eq:CCCharacter.UpsilonPq0}) we  obtain
\begin{align*}
   \gamma_{\phi}(Q)(0)  = \frac{1}{(2q)!}(-i)^{\frac{n}{2}}(2\pi)^{-\frac{a}{2}} \left|\omega^{(\bt,0)} \wedge \hat{A}(R^{TM^{\phi}}(0))\wedge 
            \nu_{\phi}\left(R^{\cN^{\phi}}(0)\right)\right|^{(a,0)}. 
\end{align*}
Combining this with~(\ref{eq:JLO.IQ-UpsilonQ}) gives the asymptotic~(\ref{eq:JLO.strIQ}). The proof of Proposition~\ref{prop:JLO.strIQ} is complete. 
\end{proof}

\begin{remark}\label{rmk:JLO.family}
 There is no major difficulty to extend this approach to the computation of the JLO cocycle  of a Dirac spectral triple to various equivariant and 
 non-equivariant  family settings, as those discussed in Remark~\ref{rmk:LEIT-family-settings}. In particular, this approach can be 
 used to compute the bivariant JLO cocycle of an equivariant Dirac spectral triple with coefficients in suitable algebras  
 (compare~\cite{Az:REBCTECCC}). As pointed out by Wu~\cite{Wu:BCCCHGIT} (who introduced the bivariant JLO cocycle), this enables us to 
 recover the higher-index theorem of Connes-Moscovici~\cite{CM:CCNCHG} (in the formulation of Lott~\cite{Lo:SCHIT}). This approach 
 can also be used to simplify the computations of the infinitesimal equivariant JLO cocycle and proof of the 
 integrability of the transgressed infinitesimal equivariant JLO cocycle in~\cite{YWang:JKT14}. 
\end{remark}

\begin{remark}\label{rmk:JLO.eta-cochain}
    The eta cochain of Wu~\cite{Wu:CCCDOMB} implements the explicit homotopy between the large-time limit of the JLO cocycle and its 
    short-time finite part. It also naturally appears in the description of the Connes-Chern character of a spin manifold with 
    boundary equipped with a $b$-metric~(see \cite{Ge:CHAPSIT, Wu:CCCDOMB, LMP:CCMBEC}). In particular, in the odd dimensional case, its first degree component agrees with the eta 
    invariant of Atiyah-Patodi-Singer~\cite{APS:SARG1}. The eta cochain is formally defined as the integral over $[0,\infty)$ of a 
    transgressed version of the JLO cocycle. The main issue at stake in this defintion is the 
    integrability near $t=0$ of the transgressed JLO cocycle~\cite{Wu:CCCDOMB}. We refer to~\cite{Po:Odd} for further applications of 
    Lemma~\ref{lem:JLO.Ht-Q} and Theorem~\ref{thm:LEIT.Volterra-LEIT} to a new proof of the 
    integrability at $t=0$ of the transgressed JLO cocycle of an equivariant Dirac spectral triple. In particular, this 
    bypasses the crossing with $S^{1}$ and the use of a Grassmannian variable from the previous approaches of Wu~\cite{Wu:CCCDOMB} in 
    the non-equivariant case and Yong Wang~\cite{YWang:KT06} in the equivariant case. 
\end{remark}

\begin{remark}\label{rmk:JLO.boundary}
    By combining the heat $b$-calculus of Melrose~\cite{Me:APSIT} with a version of Theorem~\ref{thm:LEIT.Volterra-LEIT} for the 
    $b$-differential operators we also can compute the short-time limit of the relative JLO cocycle of 
    Lesch-Moscovici-Pflaum~\cite{LMP:CCMBEC} associated to a Dirac operator on a spin manifold with boundary equipped with a $b$-metric. It would be interesting 
    to extend the results~\cite{LMP:CCMBEC} to the equivariant setting.  
\end{remark}

\end{document}